%% file: Multifusion.tex
\documentclass[12pt]{article}

\usepackage{amsmath, amsthm, amssymb}
\usepackage{enumerate}
\usepackage{pdflscape}
\usepackage{caption}

\usepackage{ifpdf}
\ifpdf
\usepackage[pdftex]{graphicx}
\else
\usepackage[dvips]{graphicx}
\fi
\usepackage{tikz}
 	 \usetikzlibrary{arrows,backgrounds}
\usepackage[all]{xy}

\usepackage{multicol}

\usepackage{tocvsec2}

\input xy
\xyoption{all}

\usepackage[pdftex,plainpages=false,hypertexnames=false,pdfpagelabels]{hyperref}
\newcommand{\arxiv}[1]{\href{http://arxiv.org/abs/#1}{\tt arXiv:\nolinkurl{#1}}}
\newcommand{\arXiv}[1]{\href{http://arxiv.org/abs/#1}{\tt arXiv:\nolinkurl{#1}}}
\newcommand{\doi}[1]{\href{http://dx.doi.org/#1}{{\tt DOI:#1}}}

\newcommand{\googlebooks}[1]{(preview at \href{http://books.google.com/books?id=#1}{google books})}

\usepackage{xcolor}
\definecolor{dark-red}{rgb}{0.7,0.25,0.25}
\definecolor{dark-blue}{rgb}{0.15,0.15,0.55}
\definecolor{medium-blue}{rgb}{0,0,.8}
\definecolor{DarkGreen}{RGB}{0,150,0}
\definecolor{rho}{named}{red}
\hypersetup{
   colorlinks, linkcolor={purple},
   citecolor={medium-blue}, urlcolor={medium-blue}
}

\usepackage{longtable}
\usepackage{fullpage}

\theoremstyle{plain}
\newtheorem{thm}{Theorem}[section]
\newtheorem*{thm*}{Theorem}
\newtheorem{thmalpha}{Theorem}

\newtheorem{cor}[thm]{Corollary}

\newtheorem*{cor*}{Corollary}
\newtheorem{conj}[thm]{Conjecture}

\newtheorem*{conj*}{Conjecture}
\newtheorem{lem}[thm]{Lemma}

\newtheorem{prop}[thm]{Proposition}

\newtheorem*{quest*}{Question}
\newtheorem*{claim*}{Claim}

\theoremstyle{definition}
\newtheorem{defn}[thm]{Definition}

\newtheorem{ex}[thm]{Example}
\newtheorem{sub-ex}[thm]{Sub-Example}
\newtheorem{rem}[thm]{Remark}
\newtheorem*{rem*}{Remark}

\DeclareMathOperator{\Ad}{Ad}

\DeclareMathOperator{\Aut}{Aut}
\DeclareMathOperator{\Bim}{Bim}
\DeclareMathOperator{\coev}{coev}

\DeclareMathOperator{\End}{End}
\DeclareMathOperator{\ev}{ev}
\DeclareMathOperator{\Hom}{Hom}
\DeclareMathOperator{\Int}{Int}

\DeclareMathOperator{\Out}{Out}

\DeclareMathOperator{\Mod}{Mod}
\DeclareMathOperator{\id}{id}

\DeclareMathOperator{\Irr}{Irr}
\DeclareMathOperator{\Tr}{Tr}

\DeclareMathOperator{\tr}{tr}

\newcommand{\comment}[1]{\textcolor{red}{[stuff commented out]}}

\newcommand{\be}{\begin{enumerate}[label=(\arabic*)]}
\newcommand{\ee}{\end{enumerate}}

\renewcommand{\b}{\mathfrak{b}}

\newcommand{\Cstar}{{\rm C^*}}
\newcommand{\Wstar}{{\rm W^*}}

\def\semicolon{;}
\def\applytolist#1{
    \expandafter\def\csname multi#1\endcsname##1{
        \def\multiack{##1}\ifx\multiack\semicolon
            \def\next{\relax}
        \else
            \csname #1\endcsname{##1}
            \def\next{\csname multi#1\endcsname}
        \fi
        \next}
    \csname multi#1\endcsname}

\def\calc#1{\expandafter\def\csname c#1\endcsname{{\mathcal #1}}}
\applytolist{calc}QWERTYUIOPLKJHGFDSAZXCVBNM;
\def\bbc#1{\expandafter\def\csname bb#1\endcsname{{\mathbb #1}}}
\applytolist{bbc}QWERTYUIOPLKJHGFDSAZXCVBNM;
\def\bfc#1{\expandafter\def\csname bf#1\endcsname{{\mathbf #1}}}
\applytolist{bfc}QWERTYUIOPLKJHGFDSAZXCVBNM;
\def\sfc#1{\expandafter\def\csname s#1\endcsname{{\sf #1}}}
\applytolist{sfc}QWERTYUIOPLKJHGFDSAZXCVBNM;

\renewcommand{\Vec}{{\sf Vec}}
\newcommand{\Hilb}{{\sf Hilb}}
\newcommand{\vN}{{\sf vN}}

\newcommand{\noshow}[1]{}
\newcommand{\MR}[1]{}

\usepackage[utf8]{inputenc}
\DeclareUnicodeCharacter{2693}{\anchor}
\usepackage{dingbat}

\usetikzlibrary{shapes}
\usetikzlibrary{backgrounds}
\usetikzlibrary{decorations,decorations.pathreplacing,decorations.markings}
\usetikzlibrary{fit,calc,through}
\usetikzlibrary{external}
\usetikzlibrary{cd}
\tikzset{
	super thick/.style={line width=3pt}
}
\tikzstyle{shaded}=[fill=red!10!blue!20!gray!30!white]
\tikzstyle{unshaded}=[fill=white]
\tikzstyle{empty box}=[circle, draw, thick, fill=white, opaque, inner sep=2mm]
\tikzstyle{annular}=[scale=.7, inner sep=1mm, baseline]
\tikzstyle{rectangular}=[scale=.75, inner sep=1mm, baseline=-.1cm]
\tikzstyle{mid>}=[decoration={markings, mark=at position 0.5 with {\arrow{>}}}, postaction={decorate}]
\tikzstyle{mid<}=[decoration={markings, mark=at position 0.5 with {\arrow{<}}}, postaction={decorate}]
\tikzstyle{over}=[double, draw=white, super thick, double=]

\newcommand{\roundNbox}[6]{
	\draw[rounded corners=5pt, very thick, #1] ($#2+(-#3,-#3)+(-#4,0)$) rectangle ($#2+(#3,#3)+(#5,0)$);
	\coordinate (ZZa) at ($#2+(-#4,0)$);
	\coordinate (ZZb) at ($#2+(#5,0)$);
	\node at ($1/2*(ZZa)+1/2*(ZZb)$) {#6};
}

  \newcommand{\tikzmath}[2][]
     {\vcenter{\hbox{\begin{tikzpicture}[#1]#2
                     \end{tikzpicture}}}
     }

\newcommand{\alphacolor}{blue}
\newcommand{\betacolor}{orange}
\newcommand{\gammacolor}{yellow}

\newcommand{\ROneColor}{white}
\newcommand{\RTwoColor}{gray!20}
\newcommand{\RThreeColor}{gray!55}
\newcommand{\RFourColor}{gray!90}

\begin{document}
\title{Representations of fusion categories and their commutants}
\author{Andr\'{e} Henriques and David Penneys}
\date{\today}
\maketitle
\begin{abstract}
A bicommutant category is a higher categorical analog of a von Neumann algebra.
We study the bicommutant categories which arise as the commutant $\mathcal{C}'$ of a fully faithful representation $\mathcal{C}\to\operatorname{Bim}(R)$ of a unitary fusion category $\mathcal{C}$.
Using results of Izumi, Popa, and Tomatsu about existence and uniqueness of representations of unitary (multi)fusion categories,
we prove that if $\mathcal{C}$ and $\mathcal{D}$ are Morita equivalent unitary fusion categories, then their commutant categories $\mathcal{C}'$ and $\mathcal{D}'$
are equivalent as bicommutant categories. In particular, they are equivalent as tensor categories:
\[
\Big(\,\,\mathcal{C}\,\,\simeq_{\text{Morita}}\,\,\mathcal{D}\,\,\Big)
\qquad\Longrightarrow\qquad
\Big(\,\,\mathcal{C}'\,\,\simeq_{\text{tensor}}\,\,\mathcal{D}'\,\,\Big).
\]
This categorifies the well-known result according to which the commutants (in some representations) of Morita equivalent finite dimensional $\rm C^*$-algebras 
are isomorphic von Neumann algebras, provided the representations are `big enough'.

We also introduce a notion of positivity for bi-involutive tensor categories.
For dagger categories, positivity is a property (the property of being a $\rm C^*$-category).
But for bi-involutive tensor categories, positivity is extra structure.
We show that unitary fusion categories and $\operatorname{Bim}(R)$ admit distinguished positive structures,
and that fully faithful representations $\mathcal{C}\to\operatorname{Bim}(R)$ automatically respect these positive structures.

\end{abstract}

\tableofcontents

\input{Chapters/Introduction.tex}
\input{Chapters/Background.tex}

\input{Chapters/Representations.tex}

\input{Chapters/Biinvolutive.tex}

\input{Chapters/Positivity.tex}

\input{Chapters/Commutants.tex}


\bibliographystyle{amsalpha}
{\footnotesize{
\bibliography{../../../bibliography}
}}
\end{document}

%% file: Chapters/Introduction.tex

\section{Introduction}

Given a von Neumann algebra $R$, 
we write $\Bim(R)$ for its category of separable bimodules.
We view
\[
\Bim(R)=\End(R\text{-Mod})
\]
as a categorical analog of $B(H)$, the $*$-algebra of bounded operators on a separable Hilbert space (and $R\text{-Mod}$ as a categorical analog of a Hilbert space).
The category of $R$-$R$-bimodules is a \emph{bi-involutive tensor category} \cite[Def.~\ref{def of bi-involutive tensor category}]{MR3663592}; the two involutions are given by the dagger structure $f\mapsto f^*$,
and by the functor which sends a bimodule to its complex conjugate (with left and right actions given by $a\bar \xi b=\overline{b^*\xi a^*}$).

If $\cC$ is a semisimple rigid $\Cstar$-tensor category or if $\cC=\Bim(R)$ for some von Neumann algebra $R$,
then it comes naturally equipped with cones
$\cP_{a,b}\,\subset\,\Hom(a\otimes\bar a, b\otimes\bar b)$ for every $a, b\in\cC$.
We call such a collection of cones  $\cP_{a,b}$ (subject to various axioms) a \emph{positive structure} on $\cC$ (Def.~\ref{def positive structure on bi-involutive tensor category}).

A \emph{representation} of a $\Cstar$-tensor category $\cC$ is a $\Cstar$-tensor functor
\[
\alpha : \cC \to \Bim(R).
\]
We call the representation \emph{fully faithful} if the underlying functor $\alpha$ is fully faithful.
If $\cC$ is equipped with bi-involutive and positive structures, and if $\alpha$ is compatible with those structures,
then we call $\alpha$ a \emph{positive representation}.

One of our results is that when $\cC$ is a rigid $\Cstar$-tensor category, and when we restrict to fully faithful representations, then the notions of representation and of positive representation are equivalent:

\begin{thmalpha}[Thm.~\ref{thm: (i) and (iii) are equivalent}]
\label{THM-A}
Let $\cC$ be a semisimple rigid $\Cstar$-tensor category.
Then every fully faithful representation $\alpha:\cC\to \Bim(R)$ extends uniquely to a positive representation.
Moreover, every isomorphism of representations is an isomorphism of positive representations.
\end{thmalpha}

Monoidal categories categorify monoids, and fusion categories categorify finite dimensional semisimple algebras (indeed, if $\cC$ is a fusion category, then its Grothendieck ring $K_0(\cC) \otimes_\bbZ \bbC$ is a multimatrix algebra \cite[1.2(a)]{MR933415}).
Similarly, \emph{bicommutant categories}, a notion due to the first author \cite{MR3747830}, are categorical analogs of von Neumann algebras.%
\footnote{In the original definition of bicommutant categories \cite{MR3747830}, positive structures were not mentioned. We believe that this was a mistake.
We fix this by slightly altering the definition.}

Given a category $\cC$ equipped with a positive representation $\alpha:\cC \to \Bim(R)$,
its \emph{commutant category}
\[
\cC' = \big\{(X, e_X) \,\big|\, X\in \Bim(R),\, \text{$e_X$ exhibits $X$ as commuting with the image of $\cC$} \big\}
\]
is the category whose objects are pairs $(X, e_X)$ with $X\in \Bim(R)$, and
\[
e_X = \big\{ e_{X,c} : X\boxtimes_R \alpha(c) \stackrel{\cong}\longrightarrow \alpha(c)\boxtimes_R X\big\}_{c\in\cC}
\]
is a unitary half-braiding.
The commutant category is again a bi-involutive tensor category with positive structure, and the forgetful functor $(X, e_X)\mapsto X$ provides a positive representation $\cC' \to \Bim(R)$.
Iterating, we can form the double commutant $\cC'' := (\cC')'$.
We call $(\cC, \alpha)$ a \emph{bicommutant category} if the canonical inclusion $\iota: \cC \to \cC''$ is an equivalence (Def.~\ref{def:bicommutant category}).

In \cite{MR3663592}, given a fully faithful representation $\cC \to \Bim(R)$ of a unitary fusion category, we proved that the double commutant $\cC''$ is isomorphic to $\cC\otimes_{\Vec}\Hilb$, and that it is a bicommutant category.
We also proved \cite[Lem~6.1 and Thm~A]{MR3663592} that $\cC'$ is a bicommutant category.
Note that, unlike with the usual von Neumann bicommutant theorem, given $\cC$ and  $\alpha$ as above there is no formal argument guaranteeing that $\cC''$ is a bicommutant category.

In the present article, we further study these examples of bicommutant categories.
Our main theorem is a categorical analog of the following well known result about finite dimensional Neumann algebras: if $H$ is a separable infinite dimensional Hilbert space, and $A,B\subset B(H)$ are finite dimensional subalgebras which are Morita equivalent, and whose central projections are infinite, then~$A'\cong B'$.
 
\begin{thmalpha}[Thm.~\ref{thm: main thm last section}]
\label{thm:MoritaEquivalentCommutantsEquivalent}
Let $\cC_0$ and $\cC_1$ be Morita equivalent unitary fusion categories.
Let $\alpha_0:\cC_0\to \Bim(R_0)$ and $\alpha_1:\cC_1\to \Bim(R_1)$
be fully faithfully representations, where $R_0$ and $R_1$ are hyperfinite factors which are either both of type ${\rm II}$ or both of type ${\rm III}_1$.
And let $\cC_i'$ be the commutant category of $\cC_i$ inside $\Bim(R_i)$.

Then $\cC_0'$ and $\cC_1'$ are equivalent as bicommutant categories.
In particular, they are equivalent as tensor categories:
\[
\Big(\,\,\cC_0\,\,\simeq_{\text{Morita}}\,\,\cC_1\,\,\Big)
\qquad\,\,\Longrightarrow\qquad\,\,
\Big(\,\,\cC_0'\,\,\simeq_{\text{tensor}}\,\,\cC_1'\,\,\Big).
\]
\end{thmalpha}

Here, two unitary fusion categories $\cC_0$ and $\cC_1$ are said to be \emph{Morita equivalent} if they can be realised as corners inside a unitary $2\times 2$ multifusion category
\begin{equation*}
\cC 
= 
\begin{pmatrix}
\cC_0 & \cM
\\
\cM^* & \cC_1
\end{pmatrix}.
\end{equation*}
The proof of Theorem \ref{thm:MoritaEquivalentCommutantsEquivalent} goes along the following lines.
We first use Popa's reconstruction theorem \cite[Thm.~3.1]{MR1334479} to construct a fully faithful representation
$
\alpha:\cC \to\Bim(R^{\oplus 2})
$,
where $R$ is a factor isomorphic to $R_1$.
We then use Izumi--Popa--Tomatsu's uniqueness theorem \cite{MR1339767, MR1055708, MR3635673,1812.04222} to show that 
$\alpha|_{\cC_i}$ and $\alpha_i$ are isomorphic representations of $\cC_i$.
In particular, the commutant category of $\cC_i$ inside $\Bim(R)$ is equivalent to the commutant category of $\cC_i$ inside $\Bim(R_i)$.
To finish, we invoke the following general theorem about the commutant categories of unitary $k\times k$ multifusion categories:

\begin{thmalpha}[Thm.~\ref{thm: C' --> C_i' is an equivalence of categories}]
\label{thm:CommutantsEquivalent}
Let $\cC$ be a unitary $k\times k$ multifusion category, and let
\[
\cC\longrightarrow \Bim(R_1\oplus\ldots\oplus R_k)
\]
be a fully faithful representation, where the $R_i$ are factors.
Let $\cC'$ be the commutant of $\cC$ inside $\Bim(R_1\oplus\ldots\oplus R_k)$, and
$\cC_i'$ the commutant of the $i$th corner $\cC_i\subset \cC$ inside $\Bim(R_i)$.
Then the obvious functor $\cC'\to\cC_i'$ is an equivalence of categories.
\end{thmalpha}

%

\paragraph{Acknowledgements.}
The authors would like to thank Reiji Tomatsu for helping us understand his article \cite{1812.04222}.

We thank the Isaac Newton Institute for Mathematical Sciences, Cambridge, for support and hospitality during the programme Operator Algebras: Subfactors and their Applications where work on this paper was undertaken. 
This work was supported by EPSRC grant no EP/K032208/1.
Andr\'e Henriques was supported by the Leverhulme trust and the EPSRC grant ``Quantum Mathematics and Computation''.
David Penneys was partially supported by NSF DMS grants 1500387/1655912 and 1654159.

%% file: Chapters/Background.tex

\section{Unitary multifusion categories}

\subsection{Tensor categories}\label{sec: Tensor categories}

We briefly discuss the various types of categories that appear in our article.
For more details, we refer the reader to 
\cite{MR808930,MR2767048,MR3242743}
and
\cite[\S2-3]{MR3663592}.
All our categories are assumed to be linear over $\bbC$ (i.e., enriched over complex vector spaces), admit direct sums, and be idempotent complete, and
all our functors are assumed to be linear.\medskip

\noindent
A category $\cC$ is called:
\begin{itemize}
\item
a \emph{dagger category} if for every $x,y\in \cC$ we are given an antilinear map $*: \cC(x, y) \to \cC(y, x)$, called the adjoint, satisfying $f^{**} = f$ and $(f\circ g)^* = g^* \circ f^*$.
A morphism $f$ is called \emph{unitary} if $f^*=f^{-1}$.
\item
a \emph{$\Cstar$-category} if it is a dagger category and admits a faithful dagger functor $\cC \to \Hilb$ whose image is norm closed at the level of hom-spaces.
\item
a \emph{tensor category} if we are given a bifunctor $\otimes : \cC\times \cC \to \cC$ together with a unit object $1\in\cC$, an associator and left and right unitors isomorphisms which are natural and satisfy the pentagon and triangle axioms.
We will suppress these coherence isomorphisms whenever possible.
\item
a \emph{dagger tensor category} if $\cC$ has both the structures of a dagger category and of a tensor category, the associators and unitors are unitary, and $(f\otimes g)^* = f^* \otimes g^*$. 
\item
a \emph{$\Cstar$-tensor category} if $\cC$ is a dagger tensor category whose underlying dagger category is a $\Cstar$-category.
\end{itemize}

\noindent
A functor $F: \cC \to \cD$ is called:
\begin{itemize}
\item
a \emph{dagger functor} between dagger categories if $F(f^*) = F(f)^*$ for all $f\in \cC(x,y)$.

\item
a \emph{tensor functor} between tensor categories if $F$ is equipped with natural isomorphisms $\mu_{x,y} : F(x) \otimes F(y) \to F(x\otimes y)$ and $i: 1_\cD\to F(1_\cC)$ satisfying
\begin{align*}
&\mu_{x,y\otimes z}\circ(\id_{F(x)}\otimes\mu_{y,z}) = \mu_{x\otimes y, z}\circ (\mu_{x,y}\otimes \id_{F(z)})
\quad\text{and}
\\ &\mu_{1,x}\circ (i\otimes \id_{F(x)}) = \id_{F(x)} = \mu_{x,1}\circ(\id_{F(x)}\otimes i).
\end{align*}
\item
an \emph{anti-tensor functor} if it is equipped with natural isomorphisms $\nu_{x,y} : F(x)\otimes F(y) \to F(y\otimes x)$ and $j: 1_\cD\to F(1_\cC)$ satisfying similar conditions.
Equivalently, this is a tensor functor $\cC^{\mathrm{mp}} \to \cD$, where $\cC^{\mathrm{mp}}$ denotes the tensor category with opposite monoidal structure.
\item
a \emph{dagger tensor functor} between dagger tensor categories if $F$ is both a dagger functor and a tensor functor, and the natural isomorphisms $\mu$ and $i$ are unitary.
(There is an analogous definition for a dagger anti-tensor functor.)
\end{itemize}

A tensor category $\cC$ is called \emph{rigid} if every object $x\in \cC$ admits both a left and a right dual, that is, if there exist objects ${x\!\!\;\text{\v{}}}$ and ${\text{\v{}}\!x}$, and morphisms $\ev_x:{x\!\!\;\text{\v{}}}\otimes x\to 1$, $\coev_x:1\to x\otimes {x\!\!\;\text{\v{}}}$, $\ev_{\text{\v{}}\!x}:x\otimes {\text{\v{}}\!x}\to 1$, $\coev_{\text{\v{}}\!x}:1\to {\text{\v{}}\!x} \otimes x$ satisfying the duality equations $(\id\otimes \ev)(\coev\otimes \id)=\id$ and $(\ev\otimes \id)(\id\otimes \coev)=\id$.

If $\cC$ is a rigid $\Cstar$-tensor category, then the above equations only determine ${x\!\!\;\text{\v{}}}$ and ${\text{\v{}}\!x}$ up to canonical isomorphism (as opposed to canonical unitary isomorphism).
However, when the unit object of $\cC$ admits a direct sum decomposition into simple objects, the above issue can be remedied.

Let $\cC$ be a rigid $\Cstar$-tensor category whose unit object is a sum of simple objects $1=\bigoplus 1_i$ (also known as a semisimple rigid $\Cstar$-tensor category), and
let $p_i:1\to 1$ be the corresponding orthogonal projections.
Given a morphism $f:x\to y$ in $\cC$, we write $f_{ij}:x\to y$ for $p_i\otimes f\otimes p_j$.

\begin{lem}[{\cite[Thms.~4.12 and 4.22]{MR3342166}}]\label{lem : BDH Theorem 4.12 and Theorem 4.22}
Let $\cC$ be as above.
Then, for every object $x\in \cC$, there exists an object $\overline x$ along with morphisms $\ev_x:\overline x\otimes x\to 1$ and $\coev_x:1\to x\otimes \overline x$ satisfying
the duality equations\footnote{We omit various unitors and associators to keep the equations compact.}
\begin{equation}\label{eq: duality equations (unnormalized)}
(\id\otimes \ev_x)(\coev_x\otimes \id)=\id\qquad
(\ev_x\otimes \id)(\id\otimes \coev_x)=\id
\end{equation}
and the balancing condition (which is a version of sphericality): for all $f:x\to x$, if
\begin{equation}\label{eq: duality equations (balancing condition)}
\ev_x(\id\otimes f_{ij})\ev^*_x = \lambda p_j
\quad\text{and}\quad
\coev_x^*(f_{ij} \otimes \id)\coev_x = \lambda' p_i
\quad\text{then}\quad \lambda=\lambda'.
\end{equation}
The object $\overline x$ is determined up to unique unitary isomorphism by $\ev_x$ and $\coev_x$ subject to the above equations.
\end{lem}

The object $\overline x$ is called the \emph{conjugate} of $x$.
The conjugate $\overline x$ is canonically isomorphic to ${x\!\!\;\text{\v{}}}$ and to ${\text{\v{}}\!x}$ (and it is meaningless to ask whether the isomorphisms $\overline x\to {x\!\!\;\text{\v{}}}$ and $\overline x\to {\text{\v{}}\!x}$ are unitary given that $x\!\!\;\text{\v{}}$ and $\text{\v{}}\!x$ are only well defined up to canonical isomorphism).

Let us now consider the situation of a dagger tensor functor $F:\cC\to \cD$ between rigid tensor $\Cstar$-categories, and let us
assume as before that the unit objects of $\cC$ and of $\cD$ are direct sums of simples.
In general, for $x$ an object of $\cC$, the morphisms\footnote{We omit the structure isomorphisms $\mu$ and $i$ for better readability.} $F(\ev_x)$ and $F(\coev_x)$ do not exhibit $F(\overline x)$ as the conjugate of $F(x)$, because they might fail the balancing condition.
A typical example where this problem occurs is the functor $\cC\to \mathrm{End}(\cC)$ given by the left action of a unitary fusion category $\cC$ on itself.

However, if $F$ is full (i.e.\ surjective at the level of morphisms)\footnote{Note that, by Lemma~\ref{lem: faithfulness is automatic}, such functors are almost always fully faithful.}, then 
$F$ sends balanced solutions to balanced solutions, and the canonical isomorphism $\chi_x:F(\overline x)\to \overline{F(x)}$ is unitary:

\begin{lem}
\label{lem:FullyFaithfulPreservesStandardPairings}
Let $F:\cC\to \cD$ be a full dagger tensor functor between semisimple rigid $\Cstar$-tensor category. Then:
\begin{itemize}
\item
For every $x\in\cC$, with $\ev_x:\overline x\otimes x\to 1_\cC$ and $\coev_x:1_\cC\to x\otimes \overline x$ as in Lemma~\ref{lem : BDH Theorem 4.12 and Theorem 4.22},
\begin{gather*}
\,\,\,\tilde\ev_{F(x)}\,:=F(\ev_x)\circ \mu_{\overline{x},x}:F(\overline x)\otimes F(x)\to 1_\cD\quad\\
\text{and}\,\,\,\,\tilde\coev_{F(x)}:=\mu^{-1}_{x, \overline{x}}\circ F(\coev_x):1_\cD\to F(x)\otimes F(\overline x)
\end{gather*}
form balanced solutions of the duality equations (are solutions of \eqref{eq: duality equations (unnormalized)} and \eqref{eq: duality equations (balancing condition)}).
\item
The isomorphism
\[
\chi_x :=(\tilde\ev_{F(x)}\otimes\id_{\overline{F(x)}})\circ(\id_{F(\overline x)}\otimes\coev_{F(x)}):F(\overline x)\to \overline{F(x)}
\]
is unitary, and is also given by
$\chi_x=(\id_{\overline{F(x)}}\otimes \tilde\coev^*_{F(x)})\circ(\ev^*_{F(x)}\otimes\id_{F(\overline x)})$.
\end{itemize}\end{lem}
\begin{proof}
It is easy to see that $\tilde\ev_{F(x)}$ and $\tilde\coev_{F(x)}$ satisfy condition \eqref{eq: duality equations (unnormalized)}.
To check \eqref{eq: duality equations (balancing condition)}, let $g:F(x) \to F(x)$ be an endomorphism and let $g_{ij}:=q_i\otimes g\otimes q_j$,
where $q_i, q_j\in \mathrm{End}(1_\cD)$ are minimal projections.
Let $p_i, p_j\in \mathrm{End}(1_\cC)$ be the corresponding minimal projections in $\cC$, satisfying $F(p_i)=q_i$ and $F(p_j)=q_j$.
Since $F$ is full, $g=F(f)$ for some $f: x\to x$.
By assumption, we have
$\ev_x(\id_{\overline x}\otimes f_{ij})\ev^*_x = \lambda p_j$ and $\coev_x^*(f_{ij} \otimes \id_{\overline x})\coev_x = \lambda' p_i$
with $\lambda=\lambda'$.
It follows that
\begin{align*}
\tilde\ev_{F(x)}
\circ
(\id_{F(\overline{x})}\otimes g_{ij})
\circ
\tilde\ev_{F(x)}^*
&=
(F(\ev_x)\circ \mu_{\overline{x},x})
\circ
(\id_{F(\overline{x})}\otimes F(f_{ij}))
\circ
((F(\ev_x)\circ \mu_{\overline{x},x})^*
\\&=
F(\ev_x \circ (\id_{\overline{x}}\otimes f_{ij}) \circ \ev_x^*)
=F(\lambda p_j)=\lambda q_j
\end{align*}
and
\begin{align*}
\tilde\coev_{F(x)}^*  \circ (g_{ij}\otimes \id_{F(\overline{x})}) \circ \tilde\coev_{F(x)}
&=
F(\mu^{-1}_{x, \overline{x}}\circ \coev_x)^*  \circ (F(f_{ij})\otimes \id_{F(\overline{x})}) \circ (\mu^{-1}_{x, \overline{x}}\circ \coev_x)
\\&=
F(\coev_x^* \circ (f\otimes \id_{\overline{x}}) \circ \coev_x)
= F(\lambda' p_i)=\lambda' q_i
\end{align*}
with $\lambda=\lambda'$, as desired.

The isomorphism $\chi_x =(\tilde\ev_{F(x)}\otimes\id_{\overline{F(x)}})\circ(\id_{F(\overline x)}\otimes\coev_{F(x)})$ is unitary by \cite[Thm.~4.22]{MR3342166}. 
It is equal to
$(\id_{\overline{F(x)}}\otimes \tilde\coev^*_{F(x)})\circ(\ev^*_{F(x)}\otimes\id_{F(\overline x)})$
since the latter is visibly the inverse of $\chi_x^*$.
\end{proof}

\subsection{Unitary multifusion categories}\label{sec: Unitary multifusion categories}

A category is called \emph{semisimple} if every object is a finite direct sum of simple objects.

\begin{defn}
A semisimple rigid tensor category with finitely many isomorphism classes of simple objects, and whose unit object is simple is called a \emph{fusion category}.
If we drop the condition that the unit object is simple, then we get the notion of of a multifusion category:
a \emph{multifusion category} is a tensor category which is rigid, semisimple, and which has finitely many isomorphism classes of simple objects.

A \emph{unitary (multi)fusion category} is a dagger tensor category whose underlying dagger category is a $\Cstar$-category, and whose underlying tensor category is a (multi)fusion category.
\end{defn}

A multifusion category is called \emph{indecomposable} if it is not the direct sum of two non-trivial multifusion categories.

Let $\cC$ be an indecomposable multifusion category and let $1=\bigoplus_{i=1}^k 1_i$ be the decomposition of its unit object into simple objects.
Then the subcategories $\cC_{ij} := 1_i\otimes \cC \otimes 1_j$ are all non-zero, and we may write $\cC=\bigoplus_{ij} \cC_{ij}$.
The tensor product $\otimes:\cC_{ij}\times \cC_{j'k}\to \cC$ takes values in $\cC_{ik}$, and is non-zero if and only if $j=j'$.
Each $\cC_i:=\cC_{ii}$ is a fusion category, and each $\cC_{ij}$ is an invertible bimodule (a Morita equivalence) between $\cC_i$ and $\cC_j$ \cite[Def.~4.2]{MR1966524} \cite[Prop.~4.2]{MR2677836}.
We call $\cC_{i}$ a \emph{corner} of $\cC$.

There is an analogy between indecomposable multifusion categories and matrix algebras.
Inspired by that analogy, if $\cC$ is an indecomposable multifusion category whose unit object is a sum of $k$ irreducible objects, then we call $\cC$ a \emph{$k\times k$ multifusion category},

The datum of a $2\times 2$ multifusion category is equivalent to the data of a fusion category $\cC_0$, a fusion category $\cC_1$, and a Morita equivalence $\cM$ between $\cC_0$ and $\cC_1$.
Letting $\cM^*$ be the inverse bimodule of $\cM$ (which is unique up to contractible choice), the associated multifusion category is given by
\begin{equation}\label{eq: 2x2 multifusion}
\cC 
=
\begin{pmatrix}
\cC_0 & \cM
\\
\cM^* & \cC_1
\end{pmatrix}.
\end{equation}
Its unit object $1_\cC = 1_0 \oplus 1_1$ is the sum of the unit object of $\cC_0$ and that of $\cC_1$.

When $\cC_0$ and $\cC_1$ are unitary fusion categories, we define a \emph{unitary Morita equivalence} to be a unitary $2\times 2$ multifusion category as in \eqref{eq: 2x2 multifusion}.
This is the analog of \cite[Def.~5.15]{MR1966524} within the context of tensor $\Cstar$-categories (M\"{u}ger only discusses dagger and pivotal tensor categories).

Let $\cC$ be a unitary multifusion category.
Then the \emph{dimension} of an object $x\in \cC_{ij}$ is the unique number  $d_x$ that satisfies
\begin{equation}\label{eq: d is the number which satisfies...}
\ev_x \circ\ev^*_x = d_x p_j
\qquad\text{and}\qquad
\coev_x^*\circ\coev_x = d_x p_i.
\end{equation}
It is such that $d_x=0$ iff $x=0$, $d_x\ge 1$ for $x\not=0$, $d_x=1$ iff $x$ is invertible, $d_{x\oplus y}=d_x+d_y$ for $x,y\in \cC_{ij}$, and $d_{x\otimes y}=d_x\cdot d_y$ for $x\in \cC_{ij}$ and $y\in \cC_{jk}$ \cite[Prop.~5.2]{MR3342166} (see also \cite{MR1444286}).
Extending the assignments $x\mapsto d_x$ by additivity to all the objects of $\cC$, we get a ring homomorphism $K_0(\cC) \to M_n(\bbR)$.

The \emph{global dimension} of a fusion category is the sum of the squares of the dimensions of its simple objects.
By \cite[Prop.~5.17]{MR1966524}, Morita equivalent unitary fusion categories always have same global dimension.
As a consequence, if $\cC$ is an indecomposable unitary multifusion category, then all its corners $\cC_i$ have same global dimension.
We denote this common quantity by $D=D(\cC)$ and call it the global dimension of $\cC$.

The following result generalises \cite[Prop.~5.17]{MR1966524} and appears as \cite[Prop.~2.17]{MR2183279}.
For every $i$ and $j$, let $\Irr(\cC_{ij})$ be a set of representatives of the irreducible objects of $\cC_{ij}$.

\begin{lem}\label{lem: Dim(C_ij)}
Let $\cC$ be an indecomposable unitary multifusion category.
Then, for every $i$ and $j$, we have
\[
\sum_{x\in \Irr(\cC_{ij})} d_x^2=D(\cC).
\]
\end{lem}

\noindent
We will provide a proof of this lemma in the next section for completeness, and for the convenience of the reader.

\subsection{Graphical calculus for unitary multifusion categories}

Recall that, in the familiar string diagram formalism for tensor categories, objects are denoted by strands and morphisms are denoted by coupons
\cite{MR1113284,MR2767048}. 
There is an analogous string diagram formalism for 2-categories, where objects are denoted by shaded regions, 1-morphisms by strands, and 2-morphisms by coupons \cite[\S8]{MR2767048}\cite[\S2]{MR3342166}.
A multifusion category can be thought of as a 2-category whose objects are the irreducible summands of $1$, and whose morphisms are given by $\Hom(1_i,1_j)=\cC_{ij}$.
As such, we can apply  the graphical calculus for 2-categories to depict objects and morphisms in a multifusion category.
The shadings of the regions denote the various irreducible summands of the unit object:
$$
\tikz[baseline=.1cm]{\draw[fill=\ROneColor, rounded corners=5, very thin, baseline=1cm] (0,0) rectangle (.5,.5);}=1_i
\qquad
\tikz[baseline=.1cm]{\draw[fill=\RTwoColor, rounded corners=5, very thin, baseline=1cm] (0,0) rectangle (.5,.5);}=1_j
\qquad
\tikz[baseline=.1cm]{\draw[fill=\RThreeColor, rounded corners=5, very thin, baseline=1cm] (0,0) rectangle (.5,.5);}=1_k
\qquad
\tikz[baseline=.1cm]{\draw[fill=\RFourColor, rounded corners=5, very thin, baseline=1cm] (0,0) rectangle (.5,.5);}=1_\ell
$$
A line between shading $i$ and shading $j$ indicates an object of $\cC_{ij}$, and a coupon 
$
\begin{tikzpicture}[scale=.8, baseline=-.4cm]
	\fill [\RTwoColor] (0,-.8+.05) rectangle (.5,.2-.05);
	\draw (0,.2-.05) -- (0,-.8+.05);
	\draw[fill=\betacolor] (0,-.3) circle (.05cm);	
	\roundNbox{unshaded}{(0,-.3)}{.22}{.1}{.1}{$\scriptstyle a$};
\end{tikzpicture}
$
is a morphism in that category.

A coupon
$
\begin{tikzpicture}[scale=.8, baseline=-.4cm]
	\fill [\RFourColor] (.5,-.8+.07) -- (.15,-.8+.07) -- (.15,-.3) -- (.18,-.3) -- (.18,.2-.07) -- (.5,.2-.07);
	\fill [\RThreeColor] (-.5,-.8+.07) -- (-.15,-.8+.07) -- (-.15,-.3) -- (-.18,-.3) -- (-.18,.2-.07) -- (-.5,.2-.07);
	\fill [\RTwoColor] (.15,-.8+.07) rectangle (-.15,-.3);
	\fill [\RFourColor] (-.18,.-.3) rectangle (0,.2-.07);
	\fill [\RTwoColor] (.18,.-.3) rectangle (0,.2-.07);
	\draw (.15,-.8+.07) -- (.15,-.4) (-.15,-.4) -- (-.15,-.8+.07);
	\draw (0,.2-.07) -- (0,-.4);
	\draw (.18,.2-.07) -- (.18,-.4)(-.18,.2-.07) -- (-.18,-.4);
	\draw[fill=\betacolor] (0,-.3) circle (.07cm);	
	\roundNbox{unshaded}{(0,-.3)}{.22}{.1}{.1}{$\scriptstyle a$};
\end{tikzpicture}
$
with multiple lines on the bottom and on the top denotes a morphism from a tensor product of objets of $\cC$ (each of them living in some $\cC_{ij}$) to another such tensor product.

Let now $\cC$ be a unitary multifusion category. 
If $x\in \cC_{ij}$, $y\in \cC_{jk}$, and $z\in \cC_{ki}$ are irreducible objects, we let 
$e_\alpha\in\mathrm{Hom}(1,x\otimes y\otimes z)$
and 
$e^\alpha\in\mathrm{Hom}(1,\overline z\otimes \overline y\otimes \overline x)$ denote dual bases, and consider the canonical element
\[
\sqrt{d_xd_yd_z}\cdot
\sum_\alpha e_\alpha\otimes e^\alpha.
\]
As in \cite[\S2.5]{MR3663592}, we use the convention that a pair of colored nodes denotes a labelling by the above canonical element:
\begin{equation}\label{eq:pair of shaded nodes}
\begin{tikzpicture}[baseline=.2cm]
	\fill [\RThreeColor] (.4,.6) -- (.2,.6) -- (0,.3) -- (0,0) -- (.4,0);
	\fill [\RTwoColor] (.2,.6) -- (0,.3) -- (-.2,.6);
	\draw (.2,.6) -- (0,.3) -- (-.2,.6);
	\draw (0,0) -- (0,.3);
	\draw[fill=\betacolor] (0,.3) circle (.05cm);	
	\node at (-.2,.8) {\scriptsize{$x$}};
	\node at (.2,.8) {\scriptsize{$y$}};
	\node at (0,-.2) {\scriptsize{$z$}};
\end{tikzpicture}
\,\,
\otimes
\begin{tikzpicture}[baseline=-.4cm]
	\fill [\RThreeColor] (.4,-.6) -- (.2,-.6) -- (0,-.3) -- (0,0) -- (.4,0);
	\fill [\RTwoColor] (.2,-.6) -- (0,-.3) -- (-.2,-.6);
	\draw (.2,-.6) -- (0,-.3) -- (-.2,-.6);
	\draw (0,0) -- (0,-.3);
	\draw[fill=\betacolor] (0,-.3) circle (.05cm);	
	\node at (-.2,-.8) {\scriptsize{$x$}};
	\node at (.2,-.8) {\scriptsize{$y$}};
	\node at (0,.2) {\scriptsize{$z$}};
\end{tikzpicture}
:=
\sqrt{d_xd_yd_z}\cdot
\sum_\alpha\,\,\,
\begin{tikzpicture}[baseline=.2cm]
	\fill [\RThreeColor] (.5,.8) -- (.15,.8) -- (.15,.3) -- (0,.3) -- (0,-.2) -- (.5,-.2);
	\fill [\RTwoColor] (.15,.8) -- (.15,.3) -- (-.15,.3) -- (-.15,.8);
	\draw (.15,.8) -- (.15,.4) (-.15,.4) -- (-.15,.8);
	\draw (0,-.2) -- (0,.4);
	\draw[fill=\betacolor] (0,.3) circle (.05cm);	
	\node at (-.15,1) {\scriptsize{$x$}};
	\node at (.15,1) {\scriptsize{$y$}};
	\node at (0,-.4) {\scriptsize{$z$}};
	\roundNbox{unshaded}{(0,.3)}{.3}{0}{0}{$e_\alpha$};
\end{tikzpicture}
\,\otimes\,
\begin{tikzpicture}[baseline=-.4cm]
	\fill [\RThreeColor] (.5,-.8) -- (.15,-.8) -- (.15,-.3) -- (0,-.3) -- (0,.2) -- (.5,.2);
	\fill [\RTwoColor] (.15,-.8) -- (.15,-.3) -- (-.15,-.3) -- (-.15,-.8);
	\draw (.15,-.8) -- (.15,-.4) (-.15,-.4) -- (-.15,-.8);
	\draw (0,.2) -- (0,-.4);
	\draw[fill=\betacolor] (0,-.3) circle (.05cm);	
	\node at (-.15,-1) {\scriptsize{$x$}};
	\node at (.15,-1) {\scriptsize{$y$}};
	\node at (0,.4) {\scriptsize{$z$}};
	\roundNbox{unshaded}{(0,-.3)}{.3}{0}{0}{$e^\alpha$};
\end{tikzpicture}
\end{equation}

The following lemma is identical to \cite[Lemma 2.16]{MR3663592}.
We include it without proof.
Let $N_{x,y}^z:=\dim(\Hom(x\otimes y,z))$.

\begin{lem}\label{lem: graph calc with pairs of nodes}
Let $\tikz[baseline=.1cm]{\draw[fill=\ROneColor, rounded corners=5, very thin, baseline=1cm] (0,0) rectangle (.5,.5);}=1_i$,
$\tikz[baseline=.1cm]{\draw[fill=\RTwoColor, rounded corners=5, very thin, baseline=1cm] (0,0) rectangle (.5,.5);}=1_j$,
$\tikz[baseline=.1cm]{\draw[fill=\RThreeColor, rounded corners=5, very thin, baseline=1cm] (0,0) rectangle (.5,.5);}=1_k$,
$\tikz[baseline=.1cm]{\draw[fill=\RFourColor, rounded corners=5, very thin, baseline=1cm] (0,0) rectangle (.5,.5);}=1_\ell$.
Then the following relations hold:\medskip\\
\centerline{\begin{tabular}{rclc}
\begin{tikzpicture}[baseline=-.1cm]
	\fill[fill=white] (-.3,-.6) rectangle (0,.6);
	\fill[fill=\RThreeColor] (0,-.6) rectangle (.5,.6);
	\fill[fill=\RTwoColor] (0,-.3) .. controls ++(45:.2cm) and ++(-45:.2cm) .. (0,.3) .. controls ++(225:.2cm) and ++(235:.2cm) .. (0,-.3);
	\draw (0,-.6) -- (0,-.3);
	\draw (0,-.3) .. controls ++(45:.2cm) and ++(-45:.2cm) .. (0,.3);
	\draw (0,-.3) .. controls ++(135:.2cm) and ++(225:.2cm) .. (0,.3);
	\draw (0,.6) -- (0,.3);
	\draw[fill=\betacolor] (0,-.3) circle (.05cm);
	\draw[fill=\betacolor] (0,.3) circle (.05cm);	
	\node at (0,-.8) {\scriptsize{$z$}};
	\node at (.3,0) {\scriptsize{$y$}};
	\node at (-.3,0) {\scriptsize{$x$}};
	\node at (0,.8) {\scriptsize{$z$}};
\end{tikzpicture}
&$=$&
$\displaystyle \sqrt{d_xd_yd_z^{-1}}\cdot N_{x,y}^z$
\,
\begin{tikzpicture}[baseline=-.1cm]
	\fill[fill=\RThreeColor] (-.2,-.6) rectangle (.1,.6);
	\draw (-.2,-.6) -- (-.2,.6);
	\node at (-.2,-.8) {\scriptsize{$z$}};
\end{tikzpicture}
&
\rm (Bigon 1)
\\
\begin{tikzpicture}[baseline=-.1cm]
	\fill[fill=white] (-.3,-.6) rectangle (0,.6);
	\fill[fill=\RThreeColor] (0,-.6) rectangle (.5,.6);
	\fill[fill=\RTwoColor] (0,-.3) .. controls ++(45:.2cm) and ++(-45:.2cm) .. (0,.3) .. controls ++(225:.2cm) and ++(235:.2cm) .. (0,-.3);
	\draw (0,-.6) -- (0,-.3);
	\draw (0,-.3) .. controls ++(45:.2cm) and ++(-45:.2cm) .. (0,.3);
	\draw (0,-.3) .. controls ++(135:.2cm) and ++(225:.2cm) .. (0,.3);
	\draw (0,.6) -- (0,.3);
	\draw[fill=\betacolor] (0,-.3) circle (.05cm);
	\draw[fill=\alphacolor] (0,.3) circle (.05cm);	
	\node at (0,-.8) {\scriptsize{$z$}};
	\node at (.3,0) {\scriptsize{$y$}};
	\node at (-.3,0) {\scriptsize{$x$}};
	\node at (0,.8) {\scriptsize{$z$}};
\end{tikzpicture}
$\otimes$
\begin{tikzpicture}[baseline=.2cm]
	\fill [\RThreeColor] (.4,.6) -- (.2,.6) -- (0,.3) -- (0,0) -- (.4,0);
	\fill [\RTwoColor] (.2,.6) -- (0,.3) -- (-.2,.6);
	\draw (.2,.6) -- (0,.3) -- (-.2,.6);
	\draw (0,0) -- (0,.3);
	\draw[fill=\alphacolor] (0,.3) circle (.05cm);	
	\node at (-.2,.8) {\scriptsize{$x$}};
	\node at (.2,.8) {\scriptsize{$y$}};
	\node at (0,-.2) {\scriptsize{$z$}};
\end{tikzpicture}
$\otimes$
\begin{tikzpicture}[baseline=-.4cm]
	\fill [\RThreeColor] (.4,-.6) -- (.2,-.6) -- (0,-.3) -- (0,0) -- (.4,0);
	\fill [\RTwoColor] (.2,-.6) -- (0,-.3) -- (-.2,-.6);
	\draw (.2,-.6) -- (0,-.3) -- (-.2,-.6);
	\draw (0,0) -- (0,-.3);
	\draw[fill=\betacolor] (0,-.3) circle (.05cm);	
	\node at (-.2,-.8) {\scriptsize{$x$}};
	\node at (.2,-.8) {\scriptsize{$y$}};
	\node at (0,.2) {\scriptsize{$z$}};
\end{tikzpicture}
&$=$&
$\displaystyle \sqrt{d_xd_yd_z^{-1}}\cdot$
\begin{tikzpicture}[baseline=-.1cm]
	\draw (-.2,-.6) -- (-.2,.6);
	\node at (-.2,-.8) {\scriptsize{$z$}};
\end{tikzpicture}
$\otimes$
\begin{tikzpicture}[baseline=.2cm]
	\fill [\RThreeColor] (.4,.6) -- (.2,.6) -- (0,.3) -- (0,0) -- (.4,0);
	\fill [\RTwoColor] (.2,.6) -- (0,.3) -- (-.2,.6);
	\draw (.2,.6) -- (0,.3) -- (-.2,.6);
	\draw (0,0) -- (0,.3);
	\draw[fill=\betacolor] (0,.3) circle (.05cm);	
	\node at (-.2,.8) {\scriptsize{$x$}};
	\node at (.2,.8) {\scriptsize{$y$}};
	\node at (0,-.2) {\scriptsize{$z$}};
\end{tikzpicture}
$\otimes$
\begin{tikzpicture}[baseline=-.4cm]
	\fill [\RThreeColor] (.4,-.6) -- (.2,-.6) -- (0,-.3) -- (0,0) -- (.4,0);
	\fill [\RTwoColor] (.2,-.6) -- (0,-.3) -- (-.2,-.6);
	\draw (.2,-.6) -- (0,-.3) -- (-.2,-.6);
	\draw (0,0) -- (0,-.3);
	\draw[fill=\betacolor] (0,-.3) circle (.05cm);	
	\node at (-.2,-.8) {\scriptsize{$x$}};
	\node at (.2,-.8) {\scriptsize{$y$}};
	\node at (0,.2) {\scriptsize{$z$}};
\end{tikzpicture}\qquad\quad
&
\rm (Bigon 2)
\end{tabular}}				
\centerline{\begin{tabular}{rclc}	
$\displaystyle \sum_{z\in \Irr(\cC_{ik})}
\sqrt{d_{z}}$
\begin{tikzpicture}[baseline=-.1cm]
	\fill [\RTwoColor] (.2,.6) -- (0,.3) -- (-.2,.6);
	\fill [\RTwoColor] (.2,-.6) -- (0,-.3) -- (-.2,-.6);
	\fill [\RThreeColor] (.2,-.6) -- (0,-.3) -- (0,.3) -- (.2,.6) -- (.4,.6) -- (.4,-.6);
	\draw (.2,-.6) -- (0,-.3) -- (-.2,-.6);
	\draw (.2,.6) -- (0,.3) -- (-.2,.6);
	\draw (0,-.3) -- (0,.3);
	\draw[fill=\betacolor] (0,-.3) circle (.05cm);
	\draw[fill=\betacolor] (0,.3) circle (.05cm);	
	\node at (-.2,-.8) {\scriptsize{$x$}};
	\node at (.2,-.8) {\scriptsize{$y$}};
	\node at (-.2,.8) {\scriptsize{$x$}};
	\node at (.2,.8) {\scriptsize{$y$}};
	\node at (.2,0) {\scriptsize{$z$}};
\end{tikzpicture}
&$=$&
$\displaystyle \sqrt{d_xd_y}\cdot$
\begin{tikzpicture}[baseline=-.1cm]
	\fill [\RTwoColor] (-.2,-.6) rectangle (.2,.6);
	\fill [\RThreeColor] (.2,-.6) rectangle (.5,.6);
	\draw (.2,-.6) -- (.2,.6);
	\draw (-.2,-.6) -- (-.2,.6);
	\node at (-.2,-.8) {\scriptsize{$x$}};
	\node at (.2,-.8) {\scriptsize{$y$}};
\end{tikzpicture}
&
\rm (Fusion)
\\
$\displaystyle \sum_{v\in \Irr(\cC_{i\ell})}$
\begin{tikzpicture}[baseline=-.1cm]
	\fill [\RTwoColor] (.2,-.6) -- (0,-.3) -- (-.2,-.6);
	\fill [\RThreeColor] (.2,.6) -- (0,.3) -- (-.2,.6);
	\fill [\RFourColor] (.2,-.6) -- (0,-.3) -- (0,.3) -- (.2,.6) -- (.4,.6) -- (.4,-.6);
	\draw (.2,-.6) -- (0,-.3) -- (-.2,-.6);
	\draw (.2,.6) -- (0,.3) -- (-.2,.6);
	\draw (0,-.3) -- (0,.3);
	\draw[fill=\betacolor] (0,-.3) circle (.05cm);
	\draw[fill=\alphacolor] (0,.3) circle (.05cm);
	\node at (-.2,-.8) {\scriptsize{$w$}};
	\node at (.2,-.8) {\scriptsize{$x$}};
	\node at (-.2,.8) {\scriptsize{$y$}};
	\node at (.2,.8) {\scriptsize{$z$}};
	\node at (.2,0) {\scriptsize{$v$}};
\end{tikzpicture}
$\otimes$
\begin{tikzpicture}[baseline=-.1cm]
	\fill [\RTwoColor] (.2,-.6) -- (0,-.3) -- (-.2,-.6);
	\fill [\RThreeColor] (.2,.6) -- (0,.3) -- (-.2,.6);
	\fill [\RFourColor] (-.2,-.6) -- (0,-.3) -- (0,.3) -- (-.2,.6) -- (-.4,.6) -- (-.4,-.6);
	\draw (.2,-.6) -- (0,-.3) -- (-.2,-.6);
	\draw (.2,.6) -- (0,.3) -- (-.2,.6);
	\draw (0,-.3) -- (0,.3);
	\draw[fill=\betacolor] (0,-.3) circle (.05cm);
	\draw[fill=\alphacolor] (0,.3) circle (.05cm);
	\node at (-.2,-.8) {\scriptsize{$\overline x$}};
	\node at (.2,-.8) {\scriptsize{$\overline w$}};
	\node at (-.2,.8) {\scriptsize{$\overline z$}};
	\node at (.2,.8) {\scriptsize{$\overline y$}};
	\node at (.2,0) {\scriptsize{$\overline v$}};
\end{tikzpicture}
&$=$&
$\displaystyle \sum_{u\in \Irr(\cC_{jk})}$
\begin{tikzpicture}[baseline=-.1cm, rotate=90]
	\fill [\RTwoColor] (-.4,-.5) -- (-.03,-.3) -- (.03,.3) -- (-.4,.5);
	\fill [\RThreeColor] (.4,-.5) -- (-.03,-.3) -- (.03,.3) -- (.4,.5);
	\fill [\RFourColor] (.4,-.7) -- (.4,-.5) -- (0,-.3) -- (-.4,-.5) -- (-.4,-.7);
	\draw (.4,-.5) -- (-.03,-.3) -- (-.4,-.5);
	\draw (.4,.5) -- (.03,.3) -- (-.4,.5);
	\draw (-.03,-.3) -- (.03,.3);
	\draw[fill=\betacolor] (-.03,-.3) circle (.05cm);
	\draw[fill=\alphacolor] (.03,.3) circle (.05cm);
	\node at (-.6,-.5) {\scriptsize{$x$}};
	\node at (.6,-.5) {\scriptsize{$z$}};
	\node at (-.6,.5) {\scriptsize{$w$}};
	\node at (.6,.5) {\scriptsize{$y$}};
	\node at (.2,0) {\scriptsize{$u$}};
\end{tikzpicture}
$\otimes$
\begin{tikzpicture}[baseline=-.1cm, rotate=90]
	\fill [\RTwoColor] (-.4,-.5) -- (.03,-.3) -- (-.03,.3) -- (-.4,.5);
	\fill [\RThreeColor] (.4,-.5) -- (.03,-.3) -- (-.03,.3) -- (.4,.5);
	\fill [\RFourColor] (.4,.7) -- (.4,.5) -- (0,.3) -- (-.4,.5) -- (-.4,.7);
	\draw (.4,-.5) -- (.03,-.3) -- (-.4,-.5);
	\draw (.4,.5) -- (-.03,.3) -- (-.4,.5);
	\draw (.03,-.3) -- (-.03,.3);
	\draw[fill=\alphacolor] (.03,-.3) circle (.05cm);
	\draw[fill=\betacolor] (-.03,.3) circle (.05cm);
	\node at (-.6,-.5) {\scriptsize{$\overline w$}};
	\node at (.6,-.5) {\scriptsize{$\overline y$}};
	\node at (-.6,.5) {\scriptsize{$\overline x$}};
	\node at (.6,.5) {\scriptsize{$\overline z$}};
	\node at (.2,0) {\scriptsize{$\overline u$}};
\end{tikzpicture}\qquad\quad\;\!
&
\rm (I=H)
\end{tabular}}
\end{lem}

The next lemma is a version of \cite[Lemma 2.17]{MR3663592}.

\begin{lem}
\label{lem:Sumdxdy}
Let $\tikz[baseline=.1cm]{\draw[fill=\ROneColor, rounded corners=5, very thin, baseline=1cm] (0,0) rectangle (.5,.5);}=1_i$,
$\tikz[baseline=.1cm]{\draw[fill=\RTwoColor, rounded corners=5, very thin, baseline=1cm] (0,0) rectangle (.5,.5);}=1_j$,
$\tikz[baseline=.1cm]{\draw[fill=\RThreeColor, rounded corners=5, very thin, baseline=1cm] (0,0) rectangle (.5,.5);}=1_k$.
Then the following relation holds:
\begin{equation}\label{eq: two bigons}
\sum_{\substack{a\in\Irr(\cC_{ij})\\b\in\Irr(\cC_{jk})}}\begin{tikzpicture}[baseline=-.1cm]
	\fill[fill=\RThreeColor] (0,-.6) -- (0,.6) -- (.5,.6) -- (.5,-.6);
	\fill[fill=\RTwoColor] (0,-.3) .. controls ++(45:.2cm) and ++(-45:.2cm) .. (0,.3) .. controls ++(-135:.2cm) and ++(135:.2cm) .. (0,-.3);
	\draw (0,-.6) -- (0,-.3);
	\draw (0,-.3) .. controls ++(45:.2cm) and ++(-45:.2cm) .. (0,.3);
	\draw (0,-.3) .. controls ++(135:.2cm) and ++(225:.2cm) .. (0,.3);
	\draw (0,.6) -- (0,.3);
	\draw[fill=\betacolor] (0,-.3) circle (.05cm);
	\draw[fill=\alphacolor] (0,.3) circle (.05cm);	
	\node at (0,-.8) {\scriptsize{$x$}};
	\node at (.27,.02) {\scriptsize{$b$}};
	\node at (-.3,0) {\scriptsize{$a$}};
	\node at (0,.8) {\scriptsize{$y$}};
\end{tikzpicture}
\otimes
\begin{tikzpicture}[baseline=-.1cm]
	\fill[fill=\RThreeColor] (0,-.6) -- (0,.6) -- (-.5,.6) -- (-.5,-.6);
	\fill[fill=\RTwoColor] (0,-.3) .. controls ++(45:.2cm) and ++(-45:.2cm) .. (0,.3) .. controls ++(-135:.2cm) and ++(135:.2cm) .. (0,-.3);
	\draw (0,-.6) -- (0,-.3);
	\draw (0,-.3) .. controls ++(45:.2cm) and ++(-45:.2cm) .. (0,.3);
	\draw (0,-.3) .. controls ++(135:.2cm) and ++(225:.2cm) .. (0,.3);
	\draw (0,.6) -- (0,.3);
	\draw[fill=\betacolor] (0,-.3) circle (.05cm);
	\draw[fill=\alphacolor] (0,.3) circle (.05cm);	
	\node at (0,-.8) {\scriptsize{$\overline x$}};
	\node at (.3,0) {\scriptsize{$\overline a$}};
	\node at (-.3,.05) {\scriptsize{$\overline b$}};
	\node at (0,.83) {\scriptsize{$\overline y$}};
\end{tikzpicture}
\,=\,\,
D
\cdot\delta_{x,y}
\begin{tikzpicture}[baseline=-.1cm]
	\fill[fill=\RThreeColor] (-.2,-.6) -- (-.2,.6) -- (.1,.6) -- (.1,-.6);
	\draw (-.2,-.6) -- (-.2,.6);
	\node at (-.2,-.8) {\scriptsize{$x$}};
\end{tikzpicture}
\otimes
\begin{tikzpicture}[baseline=-.1cm]
	\fill[fill=\RThreeColor] (-.2,-.6) -- (-.2,.6) -- (-.5,.6) -- (-.5,-.6);
	\draw (-.2,-.6) -- (-.2,.6);
	\node at (-.2,-.8) {\scriptsize{$\overline x$}};
\end{tikzpicture}
\end{equation}
\end{lem}

\begin{proof}[Proof of Lemma \ref{lem: Dim(C_ij)} and of Lemma \ref{lem:Sumdxdy}]

We use the I=H relation to rewrite the left hand side of (\ref{eq: two bigons}) as
\[
\sum_{\substack{a\in\Irr(\cC_{ij})\\b\in\Irr(\cC_{i})}}
\begin{tikzpicture}[baseline=-.1cm]
	\fill[fill=\RThreeColor] (0,-.5) rectangle (.25,.5);
	\draw (0,-.5) -- (0,.5) (0,0) -- (-.4,0);
	\draw[fill=\RTwoColor] (-.4,0) .. controls ++(180+50:.8cm) and ++(180-50:.8cm) .. (-.4,0);
	\draw[fill=\betacolor] (0,0) circle (.05cm);
	\draw[fill=\alphacolor] (-.4,0) circle (.05cm);	
	\node at (0,-.7) {\scriptsize{$x$}};
	\node at (-.21,.19) {\scriptsize{$b$}};
	\node at (-.95,0) {\scriptsize{$a$}};
	\node at (0,.7) {\scriptsize{$y$}};
\end{tikzpicture}
\otimes
\begin{tikzpicture}[baseline=-.1cm, xscale=-1]
	\fill[fill=\RThreeColor] (0,-.5) rectangle (.25,.5);
	\draw (0,-.5) -- (0,.5) (0,0) -- (-.4,0);
	\draw[fill=\RTwoColor] (-.4,0) .. controls ++(180+50:.8cm) and ++(180-50:.8cm) .. (-.4,0);
	\draw[fill=\betacolor] (0,0) circle (.05cm);
	\draw[fill=\alphacolor] (-.4,0) circle (.05cm);	
	\node at (0,-.7) {\scriptsize{$\overline x$}};
	\node at (-.21,.19+.02) {\scriptsize{$\overline b$}};
	\node at (-.95,0) {\scriptsize{$\overline a$}};
	\node at (0,.73) {\scriptsize{$\overline y$}};
\end{tikzpicture}
\]
The only terms which contribute to the sum are the ones with $b = 1_i$. We are left with
\[
\delta_{x,y}\sum_{a\in\Irr(\cC_{ij})}
\begin{tikzpicture}[baseline=-.1cm]
	\draw[fill=\RTwoColor] circle(.2);
	\node at (0,-.39) {\scriptsize{$a$}};
\end{tikzpicture}
\cdot
\begin{tikzpicture}[baseline=-.1cm]
	\draw[fill=\RTwoColor] circle(.2);
	\node at (0,-.4) {\scriptsize{$\overline a$}};
\end{tikzpicture}
\cdot\!
\begin{tikzpicture}[baseline=-.1cm]
	\fill[fill=\RThreeColor] (0,-.5) rectangle (.25,.5);
	\draw (0,-.5) -- (0,.5);
	\node at (0,-.8) {\scriptsize{$x$}};
\end{tikzpicture}
\otimes
\begin{tikzpicture}[baseline=-.1cm]
	\fill[fill=\RThreeColor] (0,-.5) rectangle (-.25,.5);
	\draw (0,-.5) -- (0,.5);
	\node at (0,-.8) {\scriptsize{$\overline x$}};
\end{tikzpicture}
\,=\,\,
\delta_{x,y}\Big(\sum_{a\in \Irr(\cC_{ij})} d_a^2\Big)\cdot\!
\begin{tikzpicture}[baseline=-.1cm]
	\fill[fill=\RThreeColor] (0,-.5) rectangle (.25,.5);
	\draw (0,-.5) -- (0,.5);
	\node at (0,-.8) {\scriptsize{$x$}};
\end{tikzpicture}
\otimes
\begin{tikzpicture}[baseline=-.1cm]
	\fill[fill=\RThreeColor] (0,-.5) rectangle (-.25,.5);
	\draw (0,-.5) -- (0,.5);
	\node at (0,-.8) {\scriptsize{$\overline x$}};
\end{tikzpicture}.
\]
By symmetry, the left hand side of (\ref{eq: two bigons}) is also equal to
\[
\delta_{x,y}\Big(\sum_{b\in \Irr(\cC_{jk})} d_b^2\Big)\cdot\!
\begin{tikzpicture}[baseline=-.1cm]
	\fill[fill=\RThreeColor] (0,-.6) rectangle (.25,.6);
	\draw (0,-.6) -- (0,.6);
	\node at (0,-.8) {\scriptsize{$x$}};
\end{tikzpicture}
\otimes
\begin{tikzpicture}[baseline=-.1cm]
	\fill[fill=\RThreeColor] (0,-.6) rectangle (-.25,.6);
	\draw (0,-.6) -- (0,.6);
	\node at (0,-.8) {\scriptsize{$\overline x$}};
\end{tikzpicture}.
\]
It follows that $\sum_{x\in \Irr(\cC_{ij})}d_x^2=\sum_{x\in \Irr(\cC_{jk})}d_x^2$.

For every $i,j,k,\ell$, we conclude (using that $\cC$ is indecomposable, i.e., that none of the $\cC_{ij}$ is zero) that
\[
\sum_{x\in \Irr(\cC_{ij})}d_x^2=\sum_{x\in \Irr(\cC_{jk})}d_x^2=\sum_{x\in \Irr(\cC_{k\ell})}d_x^2.
\qedhere
\]
\end{proof}

%% file: Chapters/Representations.tex

\section{Representations of unitary multifusion categories}
\label{sec:Representations}

A representation of a $\Cstar$-algebra $A$ is a $*$-homomorphism $A\to B(H)$, for $H$ some Hilbert space.
We take the perspective that a $\Cstar$-tensor category $\cC$ is a higher categorical analog of $\Cstar$-algebra,
and that a good higher categorical analog of a Hilbert space is a $\Wstar$-category, i.e., a category of the form $R\text{-Mod}$, for $R$ some von Neumann algebra.
A representation of $\cC$ on $R\text{-Mod}$ is then a dagger tensor functor $\cC\to \End(R\text{-Mod})=\Bim(R)$ into the category of $R$-$R$-bimodules.

\subsection{Representations of \texorpdfstring{$\Cstar$}{C*}-tensor categories}\label{sec: Representations of Cstar-tensor categories}

Given a von Neumann algebra $R$ with separable predual, let us write $\Bim(R)$ for the category of bimodues whose underlying Hilbert space is separable.
It is a $\Cstar$-tensor category when equipped with the \emph{Connes fusion product} \cite{MR703809},
\cite[Appendix B.$\delta$]{MR1303779},
\cite{1705.05600}
\[
\boxtimes_R:\,\Bim(R)\times \Bim(R) \,\to\, \Bim(R).
\]
Here, given a right module $X$ and a left module $Y$, their fusion $X\boxtimes_R Y$ is the completion of the vector space
$\Hom_R(L^2R,X)\otimes_R Y$ with respect to the inner product
\begin{equation*} 
\langle a_1\otimes y_1,a_2\otimes y_2\rangle:=\langle\langle a_2|a_1\rangle_{\!R\,} y_1,y_2\rangle,
\end{equation*}
where $\langle a_2|a_1\rangle_{\!R\,}:=a_2^*\circ a_1\in \Hom(L^2R_R,L^2R_R)=R.$
The fusion can be equivalently described as a completion of
$X\otimes_R \Hom_R(L^2R,Y)$, and as a completion of $\Hom_R(L^2R,X)\otimes_RL^2R\otimes_R \Hom_R(L^2R,Y)$.

The unit object of the above monoidal structure is provided by the standard form of the von Neumann algebra $L^2R\in\Bim(R)$ \cite{MR0407615,MR1943006}.
It is an $R$-$R$-bimodule which, for every faithful state $\varphi$, is canonically isomorphic to the GNS Hilbert space $L^2(R,\varphi)$.
The image of $1\in R\hookrightarrow L^2(R,\varphi) \cong L^2(R)$ is denoted $\sqrt{\varphi}\in L^2(R)$.

\begin{defn}\label{def:Cstar-representation}
Let $\cC$ be a $\Cstar$-tensor category.
A \emph{$\Cstar$-representation} of $\cC$ is a dagger tensor functor
\[
\cC\to\Bim(R)
\]
for some von Neumann algebra $R$.

Given two $\Cstar$-representations $\alpha : \cC\to \Bim(R)$ and $\beta : \cC \to \Bim(S)$,
a \emph{morphism} from $\alpha$ to $\beta$ consists of a bimodule ${}_S\Phi_R$ along with unitary isomorphisms
$\phi_c: \Phi\boxtimes_R \alpha(c) \to \beta(c)\boxtimes_S \Phi$, natural in $c\in\cC$,
which satisfy the following half-braiding condition:
\begin{equation}\label{eq: half-braiding condition}
\begin{tikzcd}
\Phi \boxtimes_{R} \alpha(c) \boxtimes_{R} \alpha(d)
\ar[d, "\id{\scriptscriptstyle\boxtimes}\mu^{\alpha}"]
\ar[r, "\phi_c{\scriptscriptstyle\boxtimes}\id"]
&
\beta(c)\boxtimes_{S} \Phi \boxtimes_{R} \alpha(d)
\ar[r, "\id{\scriptscriptstyle\boxtimes}\phi_d"]
&
\beta(c)\boxtimes_{S} \beta(d) \boxtimes_S \Phi
\ar[d, "\mu^\beta{\scriptscriptstyle\boxtimes}\id"]
\\
\Phi\boxtimes_{R}\alpha(c\otimes d)
\ar[rr, "\phi_{c\otimes d}"]
&&
\beta(c\otimes d) \boxtimes_S \Phi.
\end{tikzcd}\quad
\end{equation}
An \emph{isomorphism} between $\alpha$ and $\beta$ is a morphism $(\Phi,\phi):\alpha\to\beta$, where $\Phi$ is an invertible bimodule.
\end{defn}

\begin{ex}
Let $R$ be a factor and $g$ an outer automorphism that squares to the identity.
Let $\alpha:\Hilb_{\rm fd}[\bbZ/2]\to\Bim(R)$ be the trivial representation, and let $\beta:\Hilb_{\rm fd}[\bbZ/2]\to\Bim(R)$ be the representation induced by $g$.
Then $\Phi=L^2R\oplus L^2R_g$, together with the obvious isomorphisms $\phi_x:\Phi\boxtimes_R\alpha(x)\to\beta(x)\boxtimes_R\Phi$ is a morphism from $\alpha$ to $\beta$.
\end{ex}

Recall that a functor $F:\cC\to\cD$ is \emph{faithful} if the maps
$\Hom_\cC(x,y) \to \Hom_\cD(F(x), F(y))$
are injective, 
and \emph{fully faithful} if these maps are isomorphisms.
A $\Cstar$-representation $\alpha : \cC \to \Bim(R)$ 
is called fully faithful if the functor $\alpha$ is fully faithful.
When $\cC$ is semisimple, this is equivalent to the condition that the
simple objects of $\cC$ remain simple in $\Bim(R)$, and that non-isomorphic objects of $\cC$ remain non-isomorphic in $\Bim(R)$.

\begin{lem}
\label{lem:NonFullyFaithfulCriterion}
Suppose $\cC$ and $\cD$ are semisimple rigid tensor categories with simple units, and $F: \cC \to \cD$ is a tensor functor.
Then $F$ is fully faithful if and only if
\[
\nexists c\in\Irr(\cC)\setminus \{1_\cC\}\text{ such that $F(c)$ contains $1_\cD$ as a summand.}
\]
\end{lem}
\begin{proof}
If $F$ is fully faithful, then the conclusion of the lemma is clearly satisfied.
If $F$ is not fully faithful, then either a simple object in $\cC$ has non-simple image in $\cD$, or two non-isomorphic simples in $\cC$ have isomorphic images in $\cD$.
If $c\in \cC$ is simple but $F(c)$ is not simple, then
$c\otimes {c\!\!\;\text{\v{}}} \ominus 1_\cC \mapsto F(c) \otimes {F(c)\!\!\;\text{\v{}}} \ominus 1_\cD$, and the latter contains $1_\cD$ since $F(c)$ is not simple.
If $c,d\in \cC$ are non-isomorphic simples with $F(c) \cong F(d)$, then 
$1_\cC$ is not a summand of $c\otimes {d\!\!\;\text{\v{}}}$, but
$F(c\otimes {d\!\!\;\text{\v{}}})\cong F(c) \otimes {F(c)\!\!\;\text{\v{}}}$ contains $1_\cD$.
\end{proof}

\begin{lem} \label{lem: faithfulness is automatic}
Let $\cC$ be an indecomposable multifusion category (or, more generally, an indecomposable semisimple rigid $\Cstar$-tensor category) and let $F:\cC\to\cD$ be a tensor functor, where $\cD$ is not the zero category.
Then $F$ is a faithful functor.
\end{lem}

\begin{proof}
Decomposing each object of $\cC$ as a direct sum of simples, it is easy to check that the map
$\Hom_\cC(x,y)\to\Hom_\cD(F(x),F(y))$ is injective for every $x,y\in \cC$ if and only if it is injective for every $x,y\in \Irr(\cC)$.
Let $x,y\in\Irr(\cC)$ be simple objects.
If $x\not= y$, there is nothing to show.
If $x=y$, we need to show that $F(\id_x)\not =0$.
Equivalently, we need to show that $F(x)\not\cong 0$ in $\cD$.

Let $x\in \Irr(\cC_{ij})$ be a simple object.
We assume by contradiction that $F(x) \cong 0$, and show that $\cD$ is the zero category.
The coevaluation map $1_i\to x\otimes {x\!\!\;\text{\v{}}}$ exhibits $1_i$ as a direct summand of $x\otimes {x\!\!\;\text{\v{}}}$.
Therefore $F(1_i)$ is a direct summand of $F(x\otimes {x\!\!\;\text{\v{}}})$.
The latter is isomorphic to $F(x)\otimes F({x\!\!\;\text{\v{}}})\cong 0\otimes F({x\!\!\;\text{\v{}}})\cong 0$, so $F(1_i) \cong 0$.
For every object $y\in \cC_{ki}$, we therefore have
\[
F(y) \cong F(y\otimes 1_i) \cong F(y)\otimes F(1_i) \cong F(y)\otimes 0 \cong 0.
\]
Now, $1_k$ is a direct summand of $y\otimes {y\!\!\;\text{\v{}}}$, so $F(1_k)$ is a direct summand of $F(y\otimes {y\!\!\;\text{\v{}}})$.
The latter is zero because $F(y)\cong 0$, so $F(1_k)$ is zero.
All the $1_k$'s go to zero.
It follows that $1_\cD \cong F(\bigoplus_k 1_k) \cong \bigoplus_kF(1_k)\cong 0$, so $\cD$ is the zero category.
\end{proof}

\subsection{Uniqueness of representations of unitary fusion categories}
\label{sec:UniquenessFor1x1Representations}

Given a von Neumann algebra $M$ and a normal $*$-endomorphisms $\rho:M\to M$, let ${}_\rho L^2M\in \Bim(M)$ denote the bimodule obtained from $L^2M$ by twisting its left action by $\rho$.
It is well known that when $M$ is a factor of type ${\rm III}$ or ${\rm II}_\infty$, the map $\rho\mapsto {}_\rho L^2M$ induces an equivalence of $\Cstar$-tensor categories
\begin{equation}
\label{eq:EndomorphismsEquivalentToBimodules}
\End(M)\cup\{0\} \stackrel{\simeq}\longrightarrow \Bim(M).
\end{equation}
Here, $\End(M)$ is the category whose objects are normal $*$-endomorphisms, whose morphisms are given by
\[
\Hom(\rho,\sigma):=\big\{a\in M\,\big|\,a\rho(x)=\sigma(x)a\,\, \forall x\in M\big\},
\]
with $*$-operation $a\mapsto a^*$,
and tensor product operation given by the composition of endomorphisms: $\sigma\otimes \rho:=\rho\circ \sigma$.
More generally, for any properly infinite von Neumann algebra~$M$, $\Bim(M)$ is equivalent to the idempotent completion of $\End(M)$.

In \cite[Thm.\,\,2.2]{MR3635673}, Izumi uses Popa's subfactor uniqueness theorem \cite{MR1339767} to prove that representations of unitary fusion categories as endomorphisms of a hyperfinite type ${\rm III}_1$ factor with separable predual are unique up to isomorphism (for the same notion of isomorphism as the one in Definition \ref{def:Cstar-representation}).
We could use Popa's uniqueness theorem for hyperfinite finite depth $\rm II_1$ subfactors \cite{MR1055708} to prove the analogous result for hyperfinite $\rm II_1$ and $\rm II_\infty$ factors.
We will instead give a unified proof for uniquness of representations into hyperfinite $\rm II_1$, $\rm II_\infty$, and $\rm III_1$ factors 
based on the following powerful theorem recently proven by Tomatsu (we phrase it in the special case when $M$ is a factor):

\begin{thm}[{\cite[Thm.~D]{1812.04222}}]
\label{thm:Tomatsu}
Let $\cC$ be an amenable rigid $\rm C^*$-tensor category. 
Let $\alpha$ and $\beta$ be centrally free cocycle actions of $\cC$ on a properly infinite 
factor 
$M$ with separable predual. 
Suppose that $\alpha(c)$ and $\beta(c)$ are approximately unitarily equivalent for all $c\in \cC$. 
Then $\alpha$ and $\beta$ are strongly cocycle conjugate.
\end{thm}

\noindent
Let us first explain the terms which appear in the above theorem.

Here, \emph{cocycle actions} are exactly dagger tensor functors $\cC \to \End(M)\cup\{0\}$, and two such being \emph{cocycle conjugate} \cite[Def.~5.6]{1812.04222} means that there exists an automorphism $\Phi$ of $M$ and a unitary monoidal equivalence $\phi:\Ad(\Phi)\circ \alpha\Rightarrow \beta$
\begin{equation}\label{eq: equiv or reps into End(M)}
\begin{matrix}\begin{tikzpicture}
\node (1) at (-2,0) {$\cC$};
\node[inner ysep=2] (2) at (1,.7) {$\End(M)\cup\{0\}$};
\node[inner ysep=2] (3) at (1,-.7) {$\End(M)\cup\{0\}$};
\draw[->] ($(1.east)+(0,.07)$) --node[above]{$\scriptstyle\alpha$} ($(2.west)+(0,-.12)$);
\draw[->] ($(1.east)+(0,-.07)$) --node[below]{$\scriptstyle\beta$} ($(3.west)+(0,.12)$);
\draw[->] (2) --node[right]{$\scriptstyle\Ad(\Phi)$} (3);
\node[rotate=55] at (.3-.05,0) {$\Leftarrow$};
\node at (-.05-.05,0) {$\scriptstyle\phi$};
\end{tikzpicture}\end{matrix}
\end{equation}
(The adverb \emph{strongly} in \emph{strongly cocycle conjugate} means we can take $\Phi$ to be approximately inner; this condition is not necessary for our purposes here.)
The definition of an \emph{amenable} rigid $\rm C^*$-tensor category appears in \cite{MR1644299} as a certain F{\o}lner type condition, and unitary fusion categories are visibly amenable. 

An endomorphism $\rho \in \End(M)$ is called \emph{centrally trivial} if the induced endomorphism $\rho^\omega$ on the ultrapower $M^\omega$ restricts to the identity endomorphism on the central sequence algebra $M_\omega$ \cite[\S2.4 and Def.~4.1]{MR2483716}.
It is called \emph{properly centrally non-trivial} if no direct summand of $\rho$ is centrally trivial.\footnote{%
The equivalence between this definition and the one presented in \cite[Def.~2.16]{1812.04222} follows along the same lines as the proof of \cite[Lem.~8.3]{MR2322913}.}
Finally, a dagger tensor functor $\alpha: \cC \to \End(M)\cup \{0\}$ is \emph{centrally free} if $\alpha(c)$ is properly centrally non-trivial for every $c\in \Irr(\cC)\setminus \{1_\cC\}$.
The following result is an immediate corollary of Lemma \ref{lem:NonFullyFaithfulCriterion}:

\begin{lem}\label{lem: centrally free}
Let $\cC$ be a rigid $\rm C^*$-tensor category, $M$ an infinite factor, and $\alpha: \cC \to \End(M)\cup \{0\}$ a $\Cstar$-representation.
Then $\alpha$ is centrally free if and only if it is fully faithful and
for every $c\in \Irr(\cC)\setminus\{1_\cC\}$, its image $\alpha(c)\in \End(M)$ is not centrally trivial. \hfill $\square$
\end{lem}

The following result of Masuda and Tomatsu characterises centrally trivial endomorphims of hyperfinite factors.

\begin{thm}[{\cite[Thm.~4.12, Lem. 4.10]{MR2483716}}]\label{thm:Masuda-Tomatsu}
Let $R$ be a hyperfinite factor of type $\rm II_\infty$ or $\rm III_\lambda$ for $\lambda\in (0,1]$.
Then the set of irreducible centrally trivial endomorphisms of $R$ is exactly the set of automorphisms of the form $Ad(u) \circ \sigma^\phi_t$, where $\phi\in R_*$ is a normal state, and $(\sigma^\phi_t)_{t\in\bbR}$ denotes the modular automorphism group.
\end{thm}

As a corollary of Lemma~\ref{lem: centrally free} and Theorem~\ref{thm:Masuda-Tomatsu}, we get:

\begin{cor}
\label{cor:UFC centrally free}
Let $\cC$ be a unitary fusion category, $R$ a hyperfinite factor of type $\rm II_\infty$ or $\rm III_1$, and $\alpha: \cC \to \End(R)\cup \{0\}$ a $\Cstar$-representation.
Then $\alpha$ is centrally free if and only if it is fully faithful. 
\end{cor}
\begin{proof}
When $R$ is of type $\rm II_\infty$, the automorphisms $\sigma^\phi_t$ are all inner, so the conditions in Lemma~\ref{lem: centrally free} are trivially satisfied.
When $R$ is of type $\rm III_1$, the automorphisms $\sigma^\phi_t$ have infinite order in $\Out(M)$.
Since $\cC$ has only finitely many types of simple objects, no non-trivial $\sigma^\phi_t$ can be in the image of $\alpha$.
\end{proof}

For the definition of \emph{approximately unitarily equivalent}, we refer the reader to \cite[\S2.4]{1812.04222}.
We will only use this concept in the following lemma:

\begin{lem}
\label{lem:ApproxUnitaryEq}
Let $R$ be a hyperfinite factor with separable predual, and let 
$\sigma, \rho\in\End(R)$ be dualizable finite depth\footnote{An endomorphism $\rho\in\End(R)$ has \emph{finite depth} if the $\Cstar$-tensor category generated by $\rho$ is fusion.} endomorphisms.
\begin{itemize}
\item 
If $R$ is of type $\rm III_1$, then $\sigma$ and $\rho$ are always approximately unitarily equivalent.
\item
If $R$ is of type $\rm II_\infty$, $\sigma$ and $\rho$ are irreducible and their dimensions agree $d_\sigma = d_\rho$, then $\sigma$ and $\rho$ are approximately unitarily equivalent.
\end{itemize}
\end{lem}

\begin{proof}
Let $\overline{\Int}_r(R)\subset\End(R)$ denote the set of approximately inner endomorphisms of rank~$r$ \cite[Def.~2.2]{1812.04222}.
If we can show that $\rho$ and $\sigma$ are approximately inner of rank $r$ for the same number $r\in \bbR_{>0}$, then we are finished by \cite[Prop.~2.7]{1812.04222}, which states any two approximately inner endomorphisms of rank $r$ are approximately unitarily equivalent.

If $R$ is of type $\rm III_1$, then by \cite[Cor.~3.16(4)]{MR2483716} the set $\overline{\Int}_r(R)$ is independent of $r$, and equal to the set of dualizable endomorphisms $R$.
The result follows.

If $R$ is of type $\rm II_\infty$, fix a faithful normal semifinite trace $\Tr$ on $R$.
By \cite[Cor.~4.4 and Rem.~4.6]{MR1953517}, both $\rho$ and $\sigma$ admit Connes-Takesaki modules $\Mod(\rho), \Mod(\sigma)\in \bbR_{>0}$.
The Connes-Takesaki module is multiplicative \cite[Prop.~4.2(1)]{MR1953517}
so, by the finite depth assumption, $\Mod(\rho)$ and $\Mod(\sigma)$ generate finite subgroups of $\bbR_{>0}$, hence trivial.
Thus, $\rho$ and $\sigma$ both have trivial Connes-Takesaki modules.
Finally, by \cite[Lem~2.15]{MR2483716}, we have $\rho,\sigma \in \overline{\Int}_{r}(R)$ for $r=d_\sigma=d_\rho$.
\end{proof}

With the above preliminaries in place, the following result
is an immediate corollary of Tomatsu's Theorem (Thm.~\ref{thm:Tomatsu}):

\begin{thm}
\label{thm:UniqueRepresentation}
Let $\cC$ be a unitary fusion category, and let
\[
\alpha: \cC \to \Bim(R),
\qquad\,\,\, 
\beta: \cC \to \Bim(S)
\]
be fully faithful $\rm C^*$-representations, where $R$ and $S$ are hyperfinite factors with separable preduals that are either both of type $\rm II$ (either $\rm II_1$ or $\rm II_\infty$), or both of type $\rm III_1$.
Then there is an isomorphism $(\Phi, \phi): \alpha \to \beta$.
\end{thm}
\begin{proof}
The hyperfinite $\rm II_1$ and $\rm II_\infty$ factors are Morita equivalent, so we may assume that $R$ and $S$ are either both of type $\rm II_\infty$, or both of type $\rm III_1$.
By the uniqueness of the hyperfinite $\rm II_\infty$ and $\rm III_1$ factors (\cite{MR0009096,MR0454659,MR880070,1606.03156}), we may furthermore assume without loss of generality that $R=S$.

View $\alpha$ and $\beta$ as representations $\cC\to\End(R)\cup\{0\}$ under the equivalence \eqref{eq:EndomorphismsEquivalentToBimodules}.
By Corollary \ref{cor:UFC centrally free}, $\alpha$ and $\beta$ are centrally free.
By Lemma \ref{lem:ApproxUnitaryEq}, for every $c\in \Irr(\cC)$, $\alpha(c)$ and $\beta(c)$ are approximately unitarily equivalent.
By Tomatsu's Theorem (Thm.~\ref{thm:Tomatsu}), $\alpha$ and $\beta$ are therefore strongly cocycle conjugate.
In particular, 
there exists $\Phi\in \Aut(R)$ and a unitary monoidal equivalence $\phi:\Ad(\Phi)\circ \alpha\Rightarrow \beta$, as in \eqref{eq: equiv or reps into End(M)}.
Applying the equivalence \eqref{eq:EndomorphismsEquivalentToBimodules} once again, we are finished.
\end{proof}

\subsection{Existence of representations of \texorpdfstring{$2\times 2$}{2x2} unitary fusion categories}

In \cite[\S3.1]{MR3663592}, we explained how results of Popa \cite[Thm.~3.1]{MR1334479} (see also \cite[Thm.\,4.1]{MR3028581}) 
can be used to construct a fully faithful representation $\cC\to \Bim(R)$
of a unitary fusion category $\cC$ on a hyperfinite factor $R$ which is not type ${\rm I}$ (or, more generally, any factor which tensorially absorbs the hyperfinite ${\rm II}_1$ factor).
In this section, we extend this to $2\times 2$ unitary multifusion categories.

Let $A\subset B$ be an inclusion of factors.
Recall from \cite[Def.~5.1 and 5.10]{MR3342166} that its \emph{index} $[B: A]$ is the square of the dimension of the bimodule ${}_AL^2(B)_B$. By definition, the latter is specified by the conditions \eqref{eq: d is the number which satisfies...} when ${}_AL^2(B)_B$ is dualizble, and is otherwise infinite.
By \cite[Cor.~7.14]{MR3342166}, unless $A$ and $B$ are finite dimensional, the above index is equal to the index of Longo's \emph{minimal conditional expectation} $E_0 : B \to A$ (defined in \cite[Thm.~5.5]{MR1027496}).

Let us now assume that $A\subset B$ has finite index.
In the diagrammatic calculus, denoting $A$ and $B$ by the regions
$$
\tikz[baseline=.1cm]{\draw[fill=\ROneColor, rounded corners=5, very thin, baseline=1cm] (0,0) rectangle (.5,.5);}=A
\qquad
\qquad
\tikz[baseline=.1cm]{\draw[fill=\RThreeColor, rounded corners=5, very thin, baseline=1cm] (0,0) rectangle (.5,.5);}=B,
$$
the standard solutions of the duality equations (\ref{eq: duality equations (unnormalized)}, \ref{eq: duality equations (balancing condition)})
are denoted by the shaded/unshaded cups and caps:
\begin{align*}
R
&:=\,\,\,
\begin{tikzpicture}[baseline = -.1cm]
	\filldraw[fill=\RThreeColor] (-.3,.2) arc (-180:0:.3cm);
\end{tikzpicture}
\,\,:
L^2(A) \to {}_A L^2(B) \boxtimes_B L^2(B) {}_ A
&
R^*
&:=\,\,\,
\begin{tikzpicture}[baseline = -.1cm, yscale = -1]
	\filldraw[fill=\RThreeColor] (-.3,.2) arc (-180:0:.3cm);
\end{tikzpicture}\\
S
&:=
\begin{tikzpicture}[baseline = -.1cm]
	\fill[\RThreeColor, rounded corners = 5pt] (-.5, -.3) rectangle (.5,.2);
	\fill[fill=\ROneColor] (-.3,.3) -- (-.3,.2) arc (-180:0:.3cm) -- (-.3,.2);
	\draw (-.3,.2) arc (-180:0:.3cm);
\end{tikzpicture}
:
L^2(B) \to {}_B L^2(B) \boxtimes_A L^2(B) {}_ B
&
S^*
&:=
\begin{tikzpicture}[baseline = -.1cm, yscale = -1]
	\fill[\RThreeColor, rounded corners = 5pt] (-.5, -.3) rectangle (.5,.2);
	\fill[fill=\ROneColor] (-.3,.3) -- (-.3,.2) arc (-180:0:.3cm) -- (-.3,.2);
	\draw (-.3,.2) arc (-180:0:.3cm);
\end{tikzpicture}
\end{align*}
The dimension $d=[B:A]^{1/2}$ is then specified by the balancing equations $R^*\circ R = d \id_{L^2(A)}$ and $S^* \circ S = d\id_{L^2(B)}$.

It was shown in \cite[\S5]{MR3178106} that the 
collection $\cP_\bullet=\{\cP_{n,\pm}\}$ of finite dimensional $\rm C^*$-algebras
\begin{align*}
\cP_{2n,+}  &:=  \End_{A\text{-}A}   \left( L^2(B)^{\boxtimes_{\!A} n} \right)             &
\cP_{2n+1,+}  &:=  \End_{A\text{-}B}   \left( L^2(B)^{\boxtimes_{\!A} n+1} \right)     \\
\cP_{2n,-}  &:=  \End_{B\text{-}B}   \left( L^2(B)^{\boxtimes_{\!A} n+1} \right)          &
\cP_{2n+1,-}  &:=  \End_{B\text{-}A}  \left( L^2(B)^{\boxtimes_{\!A} n+1} \right)
\end{align*}
has the structure of a $\rm C^*$-\emph{planar algebra} \cite[Def.~1.37]{math.QA/9909027}.
Note that the planar algebra structure on $\cP_\bullet$ depends on the choice of $R$ and $S$ above.
Since they were chosen to satisfy the balancing condition \eqref{eq: duality equations (balancing condition)}, the planar algebra $\cP_\bullet$ is \emph{spherical}, i.e., it satisfies
\[
\tikzmath{
\fill[\ROneColor, rounded corners=5pt] (-.5,-.8) rectangle (.8,.8);
\filldraw[fill=\RThreeColor] (0,.3) arc (180:0:.3cm) -- (.6,-.3) arc (0:-180:.3cm);
\roundNbox{fill=white}{(0,0)}{.3}{0}{0}{$f$}
}
=
\tikzmath{
\fill[\RThreeColor, rounded corners=5pt] (-.8,-.8) rectangle (.5,.8);
\filldraw[fill=\ROneColor] (0,.3) arc (0:180:.3cm) -- (-.6,-.3) arc (-180:0:.3cm);
\roundNbox{fill=white}{(0,0)}{.3}{0}{0}{$f$}
}
\qquad\qquad
\forall f \in \End_{A\text{-}B}(L^2(B)).
\]

\begin{rem}
\label{rem:DGG planar algebra is Jones for II_1 extremal}
When $A\subset B$ is a finite index $\rm II_1$ subfactor which is \emph{extremal} (i.e., the traces $\tr_{A'}$ and $\tr_B$ agree on the relative commutant $A'\cap B$),
Jones \cite[Thm.~4.2.1]{math.QA/9909027} gave another construction of a spherical $\rm C^*$-planar algebra from $A\subset B$,
commonly referred to as the \emph{standard invariant} of the subfactor.
It was shown in \cite[Proof of Thm.~5.4 and Rem.~5.5]{MR3178106} that, in this setting, the $\rm C^*$-planar algebras defined in \cite[Thm.~4.2.1]{math.QA/9909027} and in \cite[\S5]{MR3178106} are isomorphic.
\end{rem}

The construction in \cite[\S5]{MR3178106} can be applied in greater generality.
Let $\cC$ be a rigid $\rm C^*$-tensor category whose unit decomposes as  a direct sum of two simple objects $1_\cC=1_+\oplus 1_-$, and let $X= 1_+ \otimes X \otimes 1_-$ be an object that generates $\cC$ (i.e., such that every object of $\cC$ is isomorphic to a direct sum of direct summands of alternating tensor powers of $X$ and $\overline{X}$).
We may then set 
$$
\cP_{n,+}:=\Hom_\cC(1_+, (X\otimes \overline{X})^{\otimes n})
\qquad
\text{and}
\qquad
\cP_{n,-}:=\Hom_\cC(1_-, (\overline{X} \otimes X)^{\otimes n}).
$$
Using the balanced solutions of the duality equations (\ref{eq: duality equations (unnormalized)}, \ref{eq: duality equations (balancing condition)}), the graphical calculus for $\cC$
equips $\cP_\bullet = \{\cP_{n,\pm}\}$ with the structure of a spherical $\rm C^*$-planar algebra.

One can recover $\cC$ from the $\rm C^*$-planar algebra $\cP_\bullet$ in the following way.
Consider the non-idempotent complete $\Cstar$ tensor category $\cC^\circ$, whose objects are the symbols
$\scriptstyle X\otimes\overline X\otimes X\otimes\ldots \otimes X$, 
$\scriptstyle X\otimes\overline X\otimes X\otimes\ldots \otimes \overline X$, 
$\scriptstyle \overline X\otimes X\otimes\overline  X\otimes\ldots \otimes X$,
$\scriptstyle \overline X\otimes X\otimes\overline  X\otimes\ldots \otimes \overline X$, 
and whose hom-spaces are given by the $\cP_{n,\pm}$.
Then the idempotent completion of $\cC^\circ$ is $\cC$
(the obvious inclusion functor $\cC^\circ\to \cC$ exhibits $\cC$ as the idempotent completion of $\cC^\circ$).
This construction is known as the \emph{category of projections} of $\cP_\bullet$.
The generating object $X$ is the strand of the planar algebra in $\cP_{1,+}$, and the action of cups and caps from the planar algebra give balanced solutions of the duality equations.

The above two constructions are each other's inverses \cite[Thm.~C and \S4]{1808.00323} (see also~\cite{1607.06041}):

\begin{thm}
\label{thm:PlanarAlgebraTensorCategoryCorrespondence}
There is an equivalence of categories\footnote{Similarly to \cite[Lem.~3.5]{1607.06041}, the collection of pairs $(\cC,X)$ forms a 2-category which is equivalent to a 1-category.}
\[
\left\{\, 
\parbox{5cm}{\rm 
Spherical $\Cstar$-planar algebras $\cP_\bullet$ 
with each $\cP_{n,\pm}$ finite dimensional
and $\cP_{0,\pm} \cong \bbC$
}\,\left\}
\,\,\,\,\cong\,\,
\left\{\,\parbox{8.5cm}{\rm 
Pairs $(\cC, X)$ with $\cC$ a rigid $\rm C^*$ tensor category 
with $1_\cC = 1_+\oplus 1_-$ a simple decomposition
and a generator $X\in \cC$ such that $X = 1_+\otimes X \otimes 1_-$}\,\right\}.
\right.\right.
\]
\end{thm}

To summarize, in \cite{math.QA/9909027}, 
Jones constructed a map from extremal finite index $\rm II_1$ subfactors $A\subset B$ to $\rm C^*$-planar algebras.
In \cite{MR3178106} 
(see also \cite{1808.00323}) a map was constructed from pairs $(\cC,X)$ as above to $\rm C^*$-planar algebras.
As explained in Remark~\ref{rem:DGG planar algebra is Jones for II_1 extremal}, Jones' construction factors through this second map, by taking $\cC \subset \Bim(A \oplus B)$ to be the subcategory generated by ${}_AL^2B{}_B$, and $X={}_AL^2B{}_B$.
So we have a commutative diagram.
\begin{equation}\label{eq: Subfactor--planarAlg--TensorCat}
\begin{tikzcd}
\left\{
\parbox{3cm}{\rm 
Finite index 
subfactors $A\subset B_{\phantom{\bullet}}$
}
\right\}
\arrow[r]
\arrow[dr]
&
\left\{
\parbox{3.5cm}{\rm 
Spherical $\rm C^*$-planar algebras $\cP_\bullet$
}
\right\}
\\
&
\left\{
\parbox{4.7cm}{\rm 
Rigid $\rm C^*$ tensor categories $\cC$ with generator $X$
}
\right\}
\arrow[u, leftrightarrow, swap,"\cong"]
\end{tikzcd}
\end{equation}

\begin{defn}
A $\rm C^*$-planar algebra is called \emph{finite depth} if the corresponding rigid $\rm C^*$ tensor category $\cC$
is multifusion (has finitely many isomorphism classes of simple objects).

A finite index subfactor $A\subset B$ is called \emph{finite depth} if its standard invariant is finite depth, equivalently, if only finitely many isomorphism classes of $A$-$A$ (equivalently $A$-$B$, $B$-$A$, or $B$-$B$) bimodules occur as direct summands of $\bigoplus_{n\in \bbN}L^2(B)^{\boxtimes_A n}$.
\end{defn}

\begin{rem}
A finite index, finite depth $\rm II_1$ subfactor is automatically extremal \cite[3.7.1]{MR1055708}.
\end{rem}

By \cite[Thm.~3.1]{MR1334479}, as explained in \cite[Proof of Thm.~4.3.1]{math.QA/9909027}, the horizontal map in \eqref{eq: Subfactor--planarAlg--TensorCat} is surjective.
Specifically, given a spherical $\rm C^*$-planar algebra $\cP_\bullet$, there exists an extremal ${\rm II}_1$ subfactor $A\subset B$ whose standard invariant is isomorphic to $\cP_\bullet$.
Moreover, when $\cP_\bullet$ has finite depth, then $A$ and $B$ can be taken to be hyperfinite.

\begin{thm}
\label{thm:ExistsRepresenation}
Let $R$ be any hyperfinite non type $\rm I$ factor.
Then every $2\times 2$ unitary multifusion category $\cC$ admits a fully faithful representation into $\Bim(R^{\oplus 2})$.
\end{thm}
\begin{proof}
Let $X \in \cC_{01}$ be any object which generates $\cC$ as an involutive category (i.e., such that every simple object of $\cC$ appears as a direct summand of either 
$\scriptstyle X\otimes\overline X\otimes X\otimes\ldots \otimes X$, 
$\scriptstyle X\otimes\overline X\otimes X\otimes\ldots \otimes \overline X$, 
$\scriptstyle \overline X\otimes X\otimes\overline  X\otimes\ldots \otimes X$, or
$\scriptstyle \overline X\otimes X\otimes\overline  X\otimes\ldots \otimes \overline X$).
For example, we may take $X = \bigoplus_{c\in \Irr(\cC_{01})} c$.
Let $\cP_\bullet$ be the finite depth spherical $\rm C^*$-planar algebra associated to the pair $(\cC,X)$.
By Popa's Theorem (\cite[Thm.~3.1]{MR1334479}, \cite[Proof of Thm.~4.3.1]{math.QA/9909027}),
there exists a hyperfinite $\rm II_1$ subfactor $A\subset B$ whose standard invariant is isomorphic to $\cP_\bullet$.

By the commutativity of \eqref{eq: Subfactor--planarAlg--TensorCat}, the map $X\mapsto {}_AL^2B_B$ extends to a fully faithful representation
$\cC \to \Bim(A\oplus B)\cong \Bim(R_{\rm II_1}^{\oplus 2})$,
where $R_{\rm II_1}$ denotes the hyperfinite $\rm II_1$ factor.
By \cite[Lem.~3.4]{MR3663592}, since $R$ is hyperfinite and not of type $\rm{I}$, $R\,\overline{\otimes}\, R_{\rm II_1} \cong R$.
We get the desired fully faithful representation by tensoring with $R$:
\[
\cC \longrightarrow \Bim(A\oplus B)\cong \Bim(R_{\rm II_1}^{\oplus 2}) \xrightarrow{-\,\overline{\otimes}\, R\,}\Bim(R^{\oplus 2}). \qedhere
\]
\end{proof}

\begin{cor}
\label{cor:ExistsRepresentation}
Let $\cC$ be a unitary fusion category, and let $R$ be a hyperfinite factor which is not of type ${\rm I}$.
Then there exists a fully faithful representation $\cC \to \Bim(R)$.
\end{cor}

\begin{proof}
Apply the previous theorem to the $2\times 2$ multifusion category
$
\Big(\!\!\;\begin{smallmatrix}
\textstyle \cC & \textstyle \cC
\\[.5mm]
\textstyle \cC & \textstyle \cC
\end{smallmatrix}\!\!\;\Big)$.
\end{proof}

It should be possible to extend the above result to the case of $k\times k$ unitary multifusion categories,
but the argument is a bit tricky so, for now, we are content to make the following conjecture, which we leave to a future joint article.

\begin{conj} 
Let $\cC$ be a $k\times k$ unitary multifusion category, and let $R$ be a hyperfinite factor which is not of type ${\rm I}$.
Then there exists a fully faithful representation $\cC \to \Bim(R^{\oplus k})$.
Moreover, if $\alpha: \cC \to \Bim(R^{\oplus k})$ and $\beta: \cC \to \Bim(S^{\oplus k})$ are fully faithful representations, where $R$ and $S$ are either both of type $\rm II$ or both of type $\rm III_1$, then $\alpha$ and $\beta$ are isomorphic in the sense of Definition~\ref{def:Cstar-representation}.
\end{conj}

%% file: Chapters/Biinvolutive.tex


\section{Bi-involutive tensor categories}
\subsection{Bi-involutive tensor categories}
\label{sec: Bi-involutive tensor categories}

An \emph{involutive} tensor category \cite{MR2861112} is a tensor category $\cC$ equipped with an anti-linear anti-tensor functor
\[
\overline{\,\cdot\,}: \cC \to \cC,\,\,\,\,\, \nu_{x,y} : \overline{x}\otimes \overline{y} \to \overline{y\otimes x},\,\,\,\,\, r: 1\to \overline{1}
\]
which squares to the identity in the sense that we are given isomorphisms $\varphi_x: x \to \overline{\overline{x}}$, natural in $x$, satisfying 
$\overline{\varphi_x} = \varphi_{\overline{x}}$,
$\varphi_1=\overline r\circ r$, and 
$\varphi_{x \otimes y}=\overline{\nu_{y,x}}\circ\nu_{\overline x,\overline y}\circ(\varphi_x\otimes\varphi_y)$.
The object $\overline x$ is called the \emph{conjugate} of $x$.

\begin{defn}\label{def of bi-involutive tensor category}
A \emph{bi-involutive} tensor category \cite[\S2.1]{MR3663592} is a dagger tensor category equipped with an involutive structure 
such that the functor $\overline{\,\cdot\,}$ is a dagger functor, and the structure isomorphisms $\nu$, $r$, $\varphi$ are unitary.

A bi-involutive functor $F: \cC \to \cD$ between bi-involutive tensor categories is a dagger tensor functor with unitary isomorphisms $\chi_x : F(\overline{x}) \to \overline{F(x)}$, natural in $x$, satisfying
\begin{align*}
\chi_{\overline x}\; &= \overline{\chi_{x}}^{-1}\circ\varphi_{F(x)}\circ F(\varphi_x)^{-1},\\
\chi_{1_\cC}\, &= \overline i\circ r_\cD\circ i^{-1}\circ F(r_\cC)^{-1},\\
\text{and}\qquad\chi_{x\otimes y}&=
\overline{\mu_{x,y}}\circ\nu_{F(y),F(x)}\circ (\chi_y\otimes\chi_x)\circ\mu_{\overline y,\overline x}^{-1}\circ F(\nu_{y,x})^{-1}.
\end{align*}
\end{defn}

\begin{ex}
The tensor category of Hilbert spaces and bounded linear maps is a bi-involutive tensor category, where the involution $\overline{\,\cdot\,}$ sends a Hilbert space to its complex conjugate.
The subcategories of finite dimensional Hilbert spaces, and of separable Hilbert spaces are similarly bi-involutive tensor categories.
\end{ex}

\begin{ex}\label{ex: bi-involutive structure on multifusion category}
Every unitary multifusion category is a bi-involutive tensor category.
More generally, any semisimple rigid $\Cstar$-tensor category is a bi-involutive tensor category.
The conjugate is described in Lemma \ref{lem : BDH Theorem 4.12 and Theorem 4.22}, and the structure map $\varphi_x$ is given by
\[
\varphi_x:=(\id\,\otimes\, \mathrm{ev}_x)\circ (\mathrm{ev}_{\overline x}^*\otimes \id):x\to \overline{\overline{x}}.
\]
By \cite[Thm.~4.22]{MR3342166}, the isomorphism $\varphi_x$ is unitary;
it is therefore also given by the formula $\varphi_x=(\mathrm{coev}_{x}^*\otimes \id)\circ (\id\,\otimes\, \mathrm{coev}_{\overline x})$.
\end{ex}

In the graphical calculus for rigid tensor categories \cite{MR2767048}, 
objects are commonly depicted by oriented strands which may bend up and down.
The strands could equally well have been chosen \emph{cooriented} given that, in the presence of an ambient orientation (an orientation of the plane in which the string diagrams are drawn), an orientation 
is equivalent to a coorientation.
However, when dealing with unitary (multi)fusion categories, or, more generally, when dealing with
with bi-involutive tensor categories, \emph{it is preferable to use coorientations}.
The involutions $f\mapsto f^*$ and $f\mapsto \overline f$ are then conveniently encoded by the reflections along the coordinate axes:
\begin{equation}\label{eq: graphical for star and bar}
\begin{tikzpicture}[baseline=-.1cm]
	\draw (0,-.9) --  (0,.9);
	\foreach \p in {0,.081,...,.9}{\path (0,-.9) -- node[sloped, rotate=-90, pos=\p, xshift=2, yshift=3]{-} (0,.9);}
	\roundNbox{unshaded}{(0,0)}{.3cm}{0}{0}{$f$};
	\node at (.2,-.92) {$\scriptstyle x$};
	\node at (.2,.92) {$\scriptstyle y$};
\end{tikzpicture}
\,\,\,\mapsto\,\,\,
\begin{tikzpicture}[baseline=-.1cm]
	\draw (0,-.9) --  (0,.9);
	\foreach \p in {0,.081,...,.9}{\path (0,-.9) -- node[sloped, rotate=-90, pos=\p, xshift=2, yshift=3]{-} (0,.9);}
	\roundNbox{unshaded}{(0,0)}{.3cm}{0}{0}{$f^*$};
	\node at (.2,-.92) {$\scriptstyle y$};
	\node at (.2,.92) {$\scriptstyle x$};
\end{tikzpicture}
\qquad\qquad\qquad
\begin{tikzpicture}[baseline=-.1cm]
	\draw (0,-.9) --  (0,.9);
	\foreach \p in {0,.081,...,.9}{\path (0,-.9) -- node[sloped, rotate=-90, pos=\p, xshift=2, yshift=3]{-} (0,.9);}
	\roundNbox{unshaded}{(0,0)}{.3cm}{0}{0}{$f$};
	\node at (.2,-.92) {$\scriptstyle x$};
	\node at (.2,.92) {$\scriptstyle y$};
\end{tikzpicture}
\,\,\,\mapsto\,\,\,
\begin{tikzpicture}[baseline=-.1cm]
	\draw (0,-.9) --  (0,.9);
	\foreach \p in {0,.081,...,.9}{\path (0,-.9) -- node[sloped, rotate=-90, pos=\p, xshift=-1.2, yshift=3]{-} (0,.9);}
	\roundNbox{unshaded}{(0,0)}{.3cm}{0}{0}{$\overline f$};
	\node at (.15,-.92) {$\scriptstyle \overline x$};
	\node at (.15,.92) {$\scriptstyle \overline y$};
\end{tikzpicture}\,\,.\!\!\!
\end{equation}
When 
$x$ and $y$ are dualizable objects in a semisimple rigid $\Cstar$-tensor category and their conjugates $\overline x$ and $\overline y$ are given by 
\eqref{eq: duality equations (unnormalized)} and \eqref{eq: duality equations (balancing condition)},
then, by \cite[(4.17)]{MR3342166}, the two involutions \eqref{eq: graphical for star and bar} are related by
\begin{equation}\label{eq: formula for f bar}
\begin{tikzpicture}[baseline=-.1cm]
	\draw (0,-.9) --  (0,.9);
	\foreach \p in {0,.081,...,.9}{\path (0,-.9) -- node[sloped, rotate=-90, pos=\p, xshift=-1.2, yshift=3]{-} (0,.9);}
	\roundNbox{unshaded}{(0,0)}{.3cm}{0}{0}{$\overline f$};
	\node at (.15,-.92) {$\scriptstyle \overline x$};
	\node at (.15,.92) {$\scriptstyle \overline y$};
\end{tikzpicture}
\,=\,
\begin{tikzpicture}[baseline=-.1cm]
	\draw (-.6,.9) -- (-.6,-.3) 
	\foreach \p in {-.005,0.12,...,.9} {node[pos=\p, xshift=-1.3, yshift=-3]{-}} 
	arc (-180:0:.3) 
	\foreach \p in {.1,.22,...,1} {node[pos=\p, sloped, rotate=-90, xshift=1.8]{-}} 
	-- (0,.3) arc (180:0:.3) 
	\foreach \p in {.1,.27,...,1} {node[pos=\p, sloped, rotate=-90, xshift=1.8]{-}}
	-- (.6,-.9) 
	\foreach \p in {0,0.125,...,.9} {node[pos=\p, xshift=-1.2, yshift=-3]{-}};
	\roundNbox{unshaded}{(0,0)}{.3cm}{0}{0}{$f^*$};
	\node at (.6+.15,-.92) {$\scriptstyle \overline x$};
	\node at (-.6+.15,.92) {$\scriptstyle \overline y$};
\end{tikzpicture}
\!=
\begin{tikzpicture}[baseline=-.1cm]
	\draw (.6,.9) -- (.6,-.3) 
	\foreach \p in {-.005,0.12,...,.9} {node[pos=\p, xshift=-1.2, yshift=-3]{-}} 
	arc (0:-180:.3) 
	\foreach \p in {.09,.25,...,1} {node[pos=\p, sloped, rotate=-90, xshift=-1.2]{-}}
	-- (0,.3) arc (0:180:.3) 
	\foreach \p in {.09,.215,...,1} {node[pos=\p, sloped, rotate=-90, xshift=-1.2]{-}} 
	-- (-.6,-.9) 
	\foreach \p in {0,0.125,...,.9} {node[pos=\p, xshift=-1.2, yshift=-3]{-}};
	\roundNbox{unshaded}{(0,0)}{.3cm}{0}{0}{$f^*$};
	\node at (-.6+.15,-.92) {$\scriptstyle \overline x$};
	\node at (.6+.15,.92) {$\scriptstyle \overline y$};
\end{tikzpicture}
\qquad
\text{and}
\qquad
\begin{tikzpicture}[baseline=-.1cm]
	\draw (0,-.9) --  (0,.9);
	\foreach \p in {0,.081,...,.9}{\path (0,-.9) -- node[sloped, rotate=-90, pos=\p, xshift=2, yshift=3]{-} (0,.9);}
	\roundNbox{unshaded}{(0,0)}{.3cm}{0}{0}{$f^*$};
	\node at (.2,-.92) {$\scriptstyle y$};
	\node at (.2,.92) {$\scriptstyle x$};
\end{tikzpicture}
\,=\,
\begin{tikzpicture}[baseline=-.1cm]
	\draw (-.6,.9) -- (-.6,-.3) 
	\foreach \p in {-.005,0.12,...,.9} {node[pos=\p, xshift=1.8, yshift=-3]{-}} 
	arc (-180:0:.3) 
	\foreach \p in {.09,.25,...,1} {node[pos=\p, sloped, rotate=-90, xshift=-1.2]{-}}
	-- (0,.3) arc (180:0:.3) 
	\foreach \p in {.09,.215,...,1} {node[pos=\p, sloped, rotate=-90, xshift=-1.2]{-}} 
	-- (.6,-.9) 
	\foreach \p in {0,0.125,...,.9} {node[pos=\p, xshift=1.8, yshift=-3]{-}};
	\roundNbox{unshaded}{(0,0)}{.3cm}{0}{0}{$\overline f$};
	\node at (.6+.22,-.92) {$\scriptstyle y$};
	\node at (-.6+.2,.92) {$\scriptstyle x$};
\end{tikzpicture}
\!=
\begin{tikzpicture}[baseline=-.1cm]
	\draw (.6,.9) -- (.6,-.3) 
	\foreach \p in {-.005,0.12,...,.9} {node[pos=\p, xshift=1.8, yshift=-3]{-}} 
	arc (0:-180:.3) 
	\foreach \p in {.1,.22,...,1} {node[pos=\p, sloped, rotate=-90, xshift=1.8]{-}} 
	-- (0,.3) arc (0:180:.3) 
	\foreach \p in {.1,.27,...,1} {node[pos=\p, sloped, rotate=-90, xshift=1.8]{-}}
	-- (-.6,-.9) 
	\foreach \p in {0,0.125,...,.9} {node[pos=\p, xshift=1.8, yshift=-3]{-}};
	\roundNbox{unshaded}{(0,0)}{.3cm}{0}{0}{$\overline f$};
	\node at (.6+.2,.92) {$\scriptstyle x$};
	\node at (-.6+.22,-.92) {$\scriptstyle y$};
\end{tikzpicture}\!.\!
\end{equation}

\subsection{The bi-involutive structure on $\Bim(R)$}
\label{sec: The bi-involutive structure on Bim(R)}

The tensor category of bimodules over a von Neumann algebra $R$ is naturally a bi-involutive tensor category.
The involution $\overline{\,\cdot\,}:\Bim(R)\to \Bim(R)$ sends a bimodule to its complex conjugate, with left and right actions given by $a\overline\xi b:=\overline{b^*\xi a^*}$.

The definition of the coherence
\[
\nu:\overline X\boxtimes_R \overline Y \to \overline{Y\boxtimes_R X}
\]
on $\Bim(R)$, and more generally on the bicategory of von Neumann algebras, involves three instances of the modular conjugation $J$. It is given by
\begin{align}
\notag
\nu\,:\,\,\, \Hom_{\;\!\text{-}R}(L^2R, \overline{X})\otimes L^2R &\otimes \Hom_{R\text{-}}(L^2R ,\overline{Y}) \\ \longrightarrow&\,\,
\overline{\Hom_{\;\!\text{-}R}(L^2R, Y)\otimes L^2R \otimes \Hom_{R\text{-}}(L^2R ,X)}
\notag
\\[1mm]
f \otimes \xi \otimes g
\,\longmapsto&\,\,
\overline{g^\star \otimes J\xi \otimes f^\star}
\label{eq:Bim(R)DefinitionOfNu}
\end{align}
where $g^\star:=\bar g\circ J$.

It was shown in \cite[\S6]{MR3342166} that when $A$ and $B$ are von Neumann algebras with atomic centers,
the dual of a dualizable bimodule ${}_AX_B$ is canonically identified with its complex conjugate ${}_B\overline X_A$.\footnote{
The results in \cite{MR3342166} are phrased in the context of von Neumann algebras with finite dimensional centers but extend verbatim to the case of von Neumann algebras with atomic centers.}
We recall the construction of standard solutions $\ev_X : \overline{X}\boxtimes_B X \to L^2B$ and $\coev_X : L^2A \to X\boxtimes_B \overline{X}$
of the conjugate equations \eqref{eq: duality equations (unnormalized)} and \eqref{eq: duality equations (balancing condition)}.

Let $B'$ denote the commutant of the right $B$-action on $X$.
The bimodule $Y:={}_AL^2(B')_{B'}$ is dualizable;
let $\widetilde{Y}$ be its dual, and let
$\mathrm{r}: L^2A \to Y \boxtimes_{B'} \widetilde{Y}$
and
$\mathrm{s}^*: \widetilde{Y} \boxtimes_A Y \to L^2(B')$
be standard solutions of the conjugate equations.
By \cite[Prop.~3.1]{MR703809}, there exists a canonical $A$-$A$ bimodule isomorphism $X\boxtimes_B \overline{X} \cong L^2(B')$.
The non-normalised minimal conditional expectation \cite[(6.8)]{MR3342166} (see also \cite{MR1172035,MR976765}) 
$$
E: B' \to A
$$ reads
$E(b) := \mathrm{r}^*\circ  (b\boxtimes_{B'} \id_{\widetilde{Y}}) \circ  \mathrm{r}\,\in\Hom_{\;\!\text{-}A}(L^2A,L^2A)=A$,
and the standard evaluation and coevaluation morphisms are given by
\begin{equation}\label{eq: standard evaluation and coevaluation morphisms}
\coev_X := 
\left(\begin{matrix}
L^2A \to L^2(B') \cong X\boxtimes_B \overline{X}\\[1mm] \textstyle\sqrt{\psi}\,\mapsto\;\!\! \sqrt{\psi\circ E}\qquad
\end{matrix}\right)
\quad\qquad \ev_X:=\coev^*_{\overline X},
\end{equation}
where $\psi$ and $\psi\circ E$ are positive elements of the preduals $L^1A$ and $L^1B'$, respectively.

\subsection{Representations of bi-involutive categories}
\label{sec: Representations of tensor categories}

In the presence of bi-involutive structures, the notion of $\Cstar$-representation (Definition~\ref{def:Cstar-representation}) can be enhanced in the following way:

\begin{defn}\label{def: rep of a bi-involutive tensor category}
Let $\cC$ be a bi-involutive tensor category.
A \emph{bi-involutive representation} of $\cC$ is a bi-involutive functor $\cC \to \Bim(R)$,
for some von Neumann algebra $R$.
\end{defn}

Given two bi-involutive representations $\alpha:\cC \to \Bim(R)$ and $\beta:\cC \to \Bim(S)$, and
a morphism $(\Phi,\phi):\alpha\to \beta$ in the sense of Definition~\ref{def:Cstar-representation},
it would be nice if we could formulate a compatibility condition with the bi-involutive structures,
i.e., a compatibility between $(\Phi,\phi)$ and the isomorphisms $\chi^\alpha$ and $\chi^\beta$.
Unfortunately, we do not know how to formulate such a condition.

However, if $\Phi$ is dualizable, then there does exist an additional condition that one can impose (the diagram
in Definition \ref{defn:IsomorphismBetweenRepresentations}), and which is not a consequence of the previous requirements.
The possibility/necessity of this further coherence was missed in our earlier paper \cite{MR3663592}.
We are thus in a position to define a notion of a dualizable morphism from $\alpha$ to $\beta$, which is not the same as that of a morphism which happens to be dualizable.
We call it a \emph{dualizable-morphism} for lack of a better name:

\begin{defn}
\label{defn:IsomorphismBetweenRepresentations}
Let $\alpha: \cC \to \Bim(R)$ and $\beta: \cC\to \Bim(S)$ be bi-involutive representations of a bi-involutive tensor category $\cC$.
A dualizable-morphism from $\alpha$ to $\beta$ is a morphism
\[
(\Phi,\phi):\alpha\to\beta
\]
of the underlying $\Cstar$-representations, where $\Phi$ is dualizable,
such that the following diagram commutes for all $c\in\cC$:
\begin{equation}\label{forgotten coherence}
\begin{tikzcd}
\Phi \boxtimes_R \alpha(\overline c)
\arrow[d, "\id_{\Phi}\boxtimes \chi_c^\alpha"]
\arrow[r, "\phi_{\overline{c}}"]
&
\beta(\overline{c})\boxtimes_S \Phi
\arrow[r, "\chi^\beta_c \boxtimes \id_{\Phi}"]
&
\overline{\beta(c)}\boxtimes_S \Phi
\\
\Phi\boxtimes_R \overline{\alpha(c)}
\arrow[d, "\id\boxtimes \ev^*_\Phi"]
&
&
\Phi\boxtimes_R \overline{\Phi}\boxtimes_S \overline{\beta(c)}\boxtimes_S \Phi
\arrow[u ,"\coev^*_\Phi \boxtimes \id"']
\\
\Phi\boxtimes_R \overline{\alpha(c)}\boxtimes_R \overline{\Phi}\boxtimes_S \Phi
\arrow[r, "\nu_{\alpha(c), \Phi}"]
&
\Phi\boxtimes_R \overline{\Phi\boxtimes_S \alpha(c)}\boxtimes_S \Phi
\arrow[r,"\overline{\phi_c}"]
&
\Phi\boxtimes_R \overline{\beta(c)\boxtimes_S \Phi}\boxtimes_S \Phi.\!
\arrow[u,"\nu^{-1}_{\Phi, \beta(c)}"']
\end{tikzcd}
\end{equation}
Here, $\ev_\Phi$ and $\coev_\Phi$ refer to the standard evaluation and coevaluation morphisms \eqref{eq: standard evaluation and coevaluation morphisms}.

Suppressing the $\nu$'s (and the associators), the coherence \eqref{forgotten coherence} is described by the following diagrammatic equation,
where the crossing on the left is $\phi_{\overline{c}}$ and the crossing on the right is~$\overline{\phi_c}$:
$$
\begin{tikzpicture}[baseline=-.4cm]
  \draw (.4,-1.5) node [right, yshift = .1cm, xshift = -.02cm] {$\scriptstyle\alpha(\overline{c})$} -- (.4,-1.3) .. controls ++(90:.5cm) and ++(270:.5cm) .. (-.4,-.3) node [left, yshift = -.25cm, xshift = .03cm] {$\scriptstyle\beta(\overline{c})$} -- (-.4,.7) node [left, yshift = .05cm, xshift = .03cm] {$\scriptstyle\overline{\beta(c)}$};  
  \draw[super thick, white] (-.4,-1.5) -- (-.4,-1.3) .. controls ++(90:.5cm) and ++(270:.5cm) .. (.4,-.3) -- (.4,.7);  
  \draw (-.4,-1.5) node [left, yshift = .12cm, xshift = .02cm] {$\scriptstyle\Phi$} -- (-.4,-1.3) .. controls ++(90:.5cm) and ++(270:.5cm) .. (.4,-.3) -- (.4,.7);  
  \roundNbox{unshaded}{(-.4,0)}{.27}{0}{0}{$\scriptstyle {}_{\phantom{i}}\chi^\beta_{c\phantom{\beta}}$};
\end{tikzpicture}
=\,
\begin{tikzpicture}[baseline=.4cm]
  \draw (-.4,-.7) node [right, yshift = .1cm, xshift = -.03cm] {$\scriptstyle\alpha(\overline{c})$} -- (-.4,.3) node [left, yshift = .35cm, xshift = .2cm] {$\scriptscriptstyle\overline{\alpha(c)}$} .. controls ++(90:.5cm) and ++(270:.5cm) .. (.4,1.3) -- (.4,1.7) node [left, yshift = -.35cm, xshift = .07cm] {$\scriptscriptstyle\overline{\beta(c)}$};  
  \draw[super thick, white] (1,1.7) -- (1,.3) arc (0:-180:.3cm) -- (.4,.3) .. controls ++(90:.5cm) and ++(270:.5cm) .. (-.4,1.3) arc (0:180:.3cm) -- (-1,-.7);  
  \draw (1,1.7) node [right, yshift = -.12cm, xshift = -.02cm] {$\scriptstyle\Phi$} -- (1,.3) arc (0:-180:.3cm) -- (.4,.3) node [right, yshift=.1cm, xshift=-.05cm] {$\scriptstyle\overline{\Phi}$} .. controls ++(90:.5cm) and ++(270:.5cm) .. (-.4,1.3) arc (0:180:.3cm) -- (-1,-.7) node [left, yshift = .12cm, xshift = .02cm] {$\scriptstyle\Phi$} ;  
  \roundNbox{unshaded}{(-.4,0)}{.27}{0}{0}{$\scriptstyle \chi^\alpha_c$};
\end{tikzpicture}
$$

\noindent
Note that, when $\cC$ is rigid, that condition is equivalent to 
\begin{equation}\label{eq: when C is rigid, that condition is equivalent to...}
\begin{tikzpicture}[baseline=-.4cm]
  \draw (.4,-1.5) node [right, yshift = .1cm, xshift = -.02cm] {$\scriptstyle\alpha(\overline{c})$} -- (.4,-1.3) .. controls ++(90:.5cm) and ++(270:.5cm) .. (-.4,-.3) node [left, yshift = -.25cm, xshift = .03cm] {$\scriptstyle\beta(\overline{c})$} -- (-.4,.7) node [left, yshift = .05cm, xshift = .03cm] {$\scriptstyle\overline{\beta(c)}$};  
  \draw[super thick, white] (-.4,-1.5) -- (-.4,-1.3) .. controls ++(90:.5cm) and ++(270:.5cm) .. (.4,-.3) -- (.4,.7);  
  \draw (-.4,-1.5) node [left, yshift = .12cm, xshift = .02cm] {$\scriptstyle\Phi$} -- (-.4,-1.3) .. controls ++(90:.5cm) and ++(270:.5cm) .. (.4,-.3) -- (.4,.7);  
  \roundNbox{unshaded}{(-.4,0)}{.27}{0}{0}{$\scriptstyle {}_{\phantom{i}}\chi^\beta_{c\phantom{\beta}}$};
\end{tikzpicture}
=\,
\begin{tikzpicture}[baseline=.4cm, xscale=-1]
  \draw (1,1.7) node [left, yshift = -.14cm, xshift = .00cm] {$\scriptstyle\overline{\beta(c)}$} -- (1,.3) arc (0:-180:.3cm) -- (.4,.3)node [left, yshift = .35cm, xshift = .2cm, scale=1.1] {$\scriptscriptstyle \beta(c)$} .. controls ++(90:.5cm) and ++(270:.5cm) .. (-.4,1.3)node [left, yshift = .0cm, xshift = .1cm, scale=1.1] {$\scriptscriptstyle\alpha(c)$} arc (0:180:.3cm) node [right, yshift = -.6cm, xshift = -.02cm] {$\scriptstyle\overline{\alpha(c)}$} -- (-1,-.7) node [right, yshift = .12cm, xshift = -.02cm] {$\scriptstyle\alpha(\overline{c})$} ;  
  \draw[super thick, white] (-.4,-.7) -- (-.4,.3) .. controls ++(90:.5cm) and ++(270:.5cm) .. (.4,1.3) -- (.4,1.7);  
  \draw (-.4,-.7) node [left, yshift = .1cm, xshift = .03cm] {$\scriptstyle\Phi$} -- (-.4,.3) .. controls ++(90:.5cm) and ++(270:.5cm) .. (.4,1.3) -- (.4,1.7);  
  \roundNbox{unshaded}{(-1,0)}{.27}{0}{0}{$\scriptstyle \chi^\alpha_c$};
\end{tikzpicture},
\end{equation}
where the crossing on the right is $\phi_c^{-1}$.
The latter condition makes sense even when $\Phi$ is not dualizable.
\end{defn}

When $\Phi$ is invertible, the maps $\phi$ in Definition~\ref{def:Cstar-representation} can be re-expressed as a unitary monoidal natural isomorphism
\begin{equation}\label{forgotten coherence - invertible case - A}
\begin{matrix}\begin{tikzpicture}
\node (1) at (-2,0) {$\cC$};
\node[inner ysep=2] (2) at (1,.7) {$\Bim(R)$};
\node[inner ysep=2] (3) at (1,-.7) {$\Bim(S)$};
\draw[->] ($(1.east)+(0,.07)$) --node[above]{$\scriptstyle\alpha$} ($(2.west)+(0,-.12)$);
\draw[->] ($(1.east)+(0,-.07)$) --node[below]{$\scriptstyle\beta$} ($(3.west)+(0,.12)$);
\draw[->] (2) --node[right]{$\scriptstyle\Ad(\Phi)$} (3);
\node[rotate=55] at (.3,0) {$\Leftarrow$};
\node at (-.05,0) {$\scriptstyle\phi$};
\end{tikzpicture}\end{matrix},
\end{equation}
in which case the coherence \eqref{forgotten coherence} becomes easier to display:
\begin{equation}\label{forgotten coherence - invertible case}
\begin{tikzcd}
\Phi \boxtimes_R \alpha(\overline c) \boxtimes_R \overline \Phi
\ar[rr, "\phi_{\overline c}"]
\ar[d, "\id \boxtimes\chi_c^{\alpha}\boxtimes \id"]
&&
\beta(\overline c)
\ar[d, "\chi_c^{\beta}"]
\\
\Phi \boxtimes_R \overline {\alpha(c)} \boxtimes_R \overline \Phi
\ar[r]&
\overline{\Phi \boxtimes_R \alpha(c) \boxtimes_R \overline \Phi}
\ar[r, "\overline{ \phi_c}"]
&
\overline {\beta(c)\,.\!\!}
\end{tikzcd}
\end{equation}

%% file: Chapters/Positivity.tex

\section{Positive structures}
\label{sec: Positive structures and bicommutant categories}

As explained in Example~\ref{ex: bi-involutive structure on multifusion category}, every rigid $\Cstar$-tensor category $\cC$ has a canonical bi-involutive structure $x\mapsto \bar x$; 
it is furthermore equipped with distinguished evaluation and coevaluation morphisms $\ev_x:\bar x\otimes x\to 1$ and $\coev_x:1\to x\otimes \bar x$.
However, if we just start with $\cC$ as a bi-involutive tensor category (i.e., if the involution $x\mapsto \overline x$ is provided as part of the data, as opposed to constructed from the $\Cstar$-tensor structure), then
there is no way of knowing which maps $\bar x\otimes x\to 1$ and $1\to x\otimes \bar x$ to call the evaluation and coevaluation morphisms,
even when $\cC$ is unitary fusion.

A \emph{positive structure} on $\cC$ determines those morphisms up to positive scalar (when $x$ is irreducible).
For bi-involutive tensor $\Cstar$-categories which are not rigid, such as $\Bim(R)$, then the evaluation and coevaluation morphisms typically fail to exist. 
But the notion of positive structure is still meaningful.


\subsection{Positive structures on bi-involutive tensor categories}
\label{sec:PositiveStructures}

Let $\cC$ be a bi-involutive $\Cstar$-tensor category.

\begin{defn}\label{def positive structure on bi-involutive tensor category}
A \emph{positive structure} on  $\cC$ is a collection of subsets 
\[
\qquad\quad \cP_{a,b}\,\subset\,\Hom_\cC(a\otimes\bar a, b\otimes\bar b)\qquad \text{for every}\,\, a,b\in\cC
\]
called \emph{cp maps}\footnote{The letters \emph{cp} stands for ``completely positive''.
We warn the reader that positive maps $a\otimes\bar a\to a\otimes\bar a$ (i.e., maps which can be written as $f^*{\circ} f$) are typically not cp.}
that satisfy the following axioms:
\begin{itemize}
\item
$0\in \cP_{a,b}$
and
$\id_{a\otimes\bar a}\in \cP_{a,a}$, 
\item
the $\cP_{a,b}$ are closed under addition, positive scaling, composition, and adjoints, 
\item
(cp maps are real)
every cp map $\theta\in \cP_{a,b}$ satisfies $\bar\theta=\theta$, i.e., the following diagram commutes
\begin{equation}\label{eq: meaning of phi = phi bar}
\begin{tikzcd}
a\otimes \overline{a}
\ar[d, "\tilde\nu"]
\ar[r, "\theta"]
&
b\otimes\overline{b} \ar[d, "\tilde\nu"]
\\
\overline{a\otimes\overline{a}} 
\ar[r, "\overline{\theta}"]
&
\overline{b\otimes\overline b}
\end{tikzcd}
\end{equation}
where $\tilde \nu := \nu\circ (\varphi\boxtimes \id)$, and
\item
(cp maps are closed under amplification)
for every cp map $\theta$ and every morphism $f$ the map $f\otimes\theta\otimes\bar f$ is also cp.
\end{itemize}
\end{defn}

\begin{ex}
\label{ex:PositiveStructureOnHilb}
The category of Hilbert spaces is equipped with a canonical positive structure, by declaring a linear map $\theta:H\otimes \overline H\to K\otimes \overline K$ to be cp if for every Hilbert space $L$, the map
\[
\id_L\otimes\,\theta\,\otimes\,\id_{\bar L}:\cJ_2(L\otimes H)\cong L\otimes H\otimes \overline H\otimes \overline L\rightarrow
L\otimes K\otimes \overline K\otimes \overline L\cong \cJ_2(L\otimes K)
\]
sends positive Hilbert-Schmidt operators to positive Hilbert-Schmidt operators.
Here, $\cJ_2(H) := H\otimes \overline H\subset B(H)$ denotes the ideal of Hilbert-Schmidt operators.
\end{ex}

More generally, we have:

\begin{ex}
\label{ex:PosiviteStructureOnBim(M)}
For every von Neumann algebra $M$, the bi-involutive tensor category $\Bim(M)$ has a canonical positive structure, explained in more generality in Section \ref{ex:PosiviteStructureOnBim} below.
\end{ex}


If $\cC$ is a bi-involutive $\Cstar$-tensor category equipped with a positive structure $\cP$,
then one can form a new category whose objects are in bijection with those of $\cC$, and whose hom-sets are the cones $\cP_{a,b}$.
If $a$ is dualizable with dual $\overline a$, then the cones $\cP_{a,b}$ satisfy a version of Frobenius reciprocity:

\begin{lem}\label{lem: one sided Frobenius reciprocity in C_+}
Let $\cC$ be a bi-involutive $\Cstar$-tensor category equipped with a positive structure,
and let $a,b,c\in\cC$ be objects.
If $a$ is dualizable in $\cC$ and satisfies {\rm$\overline a\cong\! {a\!\!\;\text{\v{}}}$}\!,
then we have an isomorphism
\begin{equation}\label{eq: Frobenius reciprocity for P}
\cP_{a\otimes b,c}\stackrel{\cong}\longrightarrow \cP_{b,\overline{a}\otimes c}
\end{equation}
given by $\theta \mapsto (\id_{\bar a} \otimes\, \theta \otimes \id_a)\circ (\ev_a^* \otimes \id_{b\otimes \overline{b}}\otimes\, \overline{\ev_a^*}\,)$.
\end{lem}
\begin{proof}
The inverse map sends
$\theta\in \cP_{b,\overline{a}\otimes c}$ to
$(\coev^*_a \otimes \id_{c\otimes \overline{c}} \otimes\, \overline{\coev^*_a})\circ
(\id_a \otimes\, \theta \otimes \id_{\overline{a}})$.
\end{proof}

Note that the isomorphism \eqref{eq: Frobenius reciprocity for P} does depend on the choice of identification $\overline a\cong\! {a\!\!\;\text{\v{}}}$.

\subsection{Positive structure on the 2-category of von Neumann algebras}
\label{ex:PosiviteStructureOnBim}

In this section, we describe the canonical positive structure on the bi-involutive 2-category $\vN$ of von Neumann algebras.
One recovers the positive structure mentioned in Example \ref{ex:PosiviteStructureOnBim(M)} by restricting to the full sub 2-category whose only object is $M$
(and, by taking $M=\bbC$, one obtains Example \ref{ex:PositiveStructureOnHilb}).

Let $A$ and $B$ be von Neumann algebras.
For ${}_AX_B$ a bimodule, and $n\in \bbN$, we define the cone
\[
\mathbf{P}_{X,n}\,\subset\,\bbC^n \otimes {}_AX \boxtimes_B \overline{X}_A \otimes \bbC^n=X^{\oplus n}\boxtimes_B \overline X^{\oplus n}
\]
to be the closed positive span of vectors of the form
\[
g \otimes \xi \otimes g^\star
\in
\Hom_{\;\!\text{-}B}(L^2B, X^{\oplus n}) \otimes_B L^2B \otimes_B \Hom_{B\text{-}}(L^2B, \overline{X}^{\oplus n})
\subset X^{\oplus n}\boxtimes_B \overline X^{\oplus n},
\]
where $g\in\Hom_{\;\!\text{-}B}(L^2B, X^{\oplus n})$, $\xi \in L^2_+B$, and $g^\star=\bar g\circ J$.

\begin{defn}\label{def: cp morphisms in Bim(R)}
Let $A$, $B_1$, and $B_2$ be von Neumann algebras, and let ${}_A X_{B_1}$ and ${}_A Y_{B_2}$ be bimodules.
A map
\[
\theta : {}_A X \boxtimes_{B_1} \overline{X}_A \to {}_A Y\boxtimes_{B_2} \overline{Y}_A
\]
is called cp if 
$\id_{\bbC^n} \otimes \theta \otimes \id_{\bbC^n}$ sends $\mathbf{P}_{X,n}$ to $\mathbf{P}_{Y,n}$ for every $n\in \bbN$.
The collection of all cp maps $\theta$ as above
is denoted $\cP_{X,Y}$.
\end{defn}

We claim that the cones $\cP_{X,Y}\subset \Hom_{A\text{-}A}(X\boxtimes_{B_1} \overline{X}, Y\boxtimes_{B_2} \overline{Y})$ equip $\vN$ with a positive structure in the sense of Definition~\ref{def positive structure on bi-involutive tensor category}.
It is clear from the definition that $\id_{X\boxtimes_{B_1} \overline{X}}\in \cP_{X,X}$, and that cp maps are closed under addition and composition.
The fact that $(\cP_{X,Y})^*=\cP_{Y,X}$ is a consequence of the cones $\mathbf{P}_{X,n}$ and $\mathbf{P}_{Y,n}$ being self-dual \cite[Prop.~3.1]{MR703809}.
The remaining two axioms are checked in Propositions~\ref{prop: check positive structure (1)} and~\ref{prop: check positive structure (2)} below.

\begin{prop}\label{prop: check positive structure (1)}
Let ${}_AX_{B_1}, {}_AY_{B_2} \in \vN$ be bimodules.
Then for every cp map $\theta \in \cP_{X,Y}$, we have $\bar \theta = \theta$ (i.e., the diagram \eqref{eq: meaning of phi = phi bar} commutes).
\end{prop}
\begin{proof}
Since $\mathbf{P}_X$ linearly spans ${}_AX\boxtimes_{B_1}\overline{X}_A$ \cite[Prop.~3.1]{MR703809}, and since all morphisms are $\bbC$-linear, it is enough to check that this diagram commutes when applied to vectors $\eta:=f\otimes \xi\otimes f^\star \in\mathbf{P}_X$.
By Definition \eqref{eq:Bim(R)DefinitionOfNu}, $\tilde\nu(\eta)=\overline\eta$.
Hence
\begin{equation*}
\overline\theta(\tilde\nu(\eta))=\overline\theta(\overline\eta)=
\overline{\theta(\eta)}=\tilde\nu(\theta(\eta)).
\qedhere
\end{equation*}
\end{proof}

\begin{lem}
\label{lem:AmplificationPreservesCP}
Let ${}_AX_{B_1}, {}_AY_{B_2} \in \vN$ be bimodules, and let $\theta \in \Hom_{A\text{-}A}(X\boxtimes_{B_1} \overline X, Y\boxtimes_{B_2} \overline Y)$. 
Then
\[
\theta\in \cP_{X,Y}
\qquad
\Leftrightarrow
\qquad
\id_{\ell^2}\otimes\, \theta\otimes \id_{\ell^2}\in\cP_{\ell^2\otimes X,\ell^2\otimes Y}.
\]
\end{lem}
\begin{proof}
Let $\mathbf{P}_{X, \infty}\subset \ell^2 \otimes {}_AX\boxtimes_{B_1}\overline{X}_A \otimes \ell^2$ be the closed positive span of
$$
g \otimes \xi \otimes g^\star
\in
\Hom_{\;\!\text{-}B}(L^2B, X^{\oplus \infty}) \otimes L^2B \otimes \Hom_{B\text{-}}(L^2B, \overline{X}^{\oplus \infty})\,\,\quad \text{for}\,\,\xi \in L^2_+B.
$$
By definition, $\id_{\ell^2}\otimes \theta\otimes \id_{\ell^2}$ is cp if and only if it sends $\mathbf{P}_{X, \infty}$ to $\mathbf{P}_{Y, \infty}$.
Let $p_n\in B(\ell^2)$ be the projection onto the span of the $n$ first basis vectors of $\ell^2=\ell^2(\bbN)$.
Then $\mathbf{P}_{X, \infty}$ is the closure of $\bigcup_{n\in\bbN}\mathbf{P}_{X,n}$ inside $\ell^2 \otimes {}_AX\boxtimes_{B_1}\overline{X}_A \otimes \ell^2$.
If $\id_{\ell^2} \otimes\, \theta \otimes \id_{\ell^2}$ maps $\mathbf{P}_{X, \infty}$ to $\mathbf{P}_{Y, \infty}$, then $p_n \otimes \theta \otimes p_n$ maps $\mathbf{P}_{X,n}$ to $\mathbf{P}_{Y, n}$ for all $n\in \bbN$, so $\theta \in \cP_{X,Y}$.

Conversely, assume $\theta \in \cP_{X,Y}$.
If $\eta \in \mathbf{P}_{X, \infty}$, then $(p_n \otimes \id_{X\boxtimes_{B_1}\overline{X}} \otimes p_n)\eta \in \mathbf{P}_{X,n}$ for all $n\in \bbN$.
Since $\id_{\bbC^n }\otimes\,\theta\otimes \id_{\bbC^n}$ maps $\mathbf{P}_{X,n}$ to $\mathbf{P}_{Y,n}$,
we have $(p_n \otimes \theta \otimes p_n)\eta \in \mathbf{P}_{Y,n}$.
Now $(p_n \otimes \theta \otimes p_n)\eta \to (\id_{\ell^2} \otimes \theta \otimes \id_{\ell^2})\eta$ as $n\to \infty$, and thus $\id_{\ell^2} \otimes \theta \otimes \id_{\ell^2}$ maps $\mathbf{P}_{X, \infty}$ to $\mathbf{P}_{Y, \infty}$.
\end{proof}

\begin{lem}
\label{lem:Z_A Amplification is CP}
Let ${}_AX_{B_1}$, ${}_AY_{B_2}$, and ${}_CZ_A$ be bimodules between von Neumann algebras.
Then, for every cp map $\theta \in \cP_{X,Y}$, we have
\[
\id_Z \boxtimes \theta \boxtimes \id_{\overline{Z}} \,\in\, \cP_{Z\boxtimes_A X, Z\boxtimes_A Y}.
\]
\end{lem}
\begin{proof}
The action of $C$ is irrelevant to the statement of the lemma, so we may treat $Z$ as a mere right $A$-module.
Without loss of generality, we may assume that the right $A$-action is faithful (otherwise, letting $q\in A$ be the central projection onto the support of $Z$, we may replace $X$ and $Y$ by the corresponding summands $qX$ and $qY$, on which $qA$ acts).

Since $Z_A$ is a faithful module, we may identify $\ell^2 \otimes Z_A$ with $(\ell^2 \otimes L^2A)_A$.
By two applications of Lemma \ref{lem:AmplificationPreservesCP}, we then have:
\begin{align*}
\theta
\text{ is cp }
&\Longleftrightarrow
\id_{\ell^2} \otimes \;\!\theta \otimes \id_{\ell^2}
\text{ is cp}
\\&
\Longleftrightarrow
(\id_{\ell^2 \otimes L^2A})\boxtimes_A \theta \boxtimes_A (\id_{L^2A\otimes \ell^2})
\text{ is cp}
\\&
\Longleftrightarrow
(\id_{\ell^2\otimes Z})\boxtimes_A \theta \boxtimes_A (\id_{\overline{Z}\otimes \ell^2})
\text{ is cp}
\\
&\Longleftrightarrow
\id_Z \boxtimes_A \theta \boxtimes_A \id_{\overline{Z}}
\text{ is cp.}
\qedhere
\end{align*}
\end{proof}

\begin{prop}\label{prop: check positive structure (2)}
For every cp map $\theta \in \cP_{X,Y}$ and $f\in \Hom_{C\text{-}A}(W,Z)$, the linear map
$f \boxtimes_A \theta \boxtimes_A \overline{f}$ is cp. 
\end{prop}
\begin{proof}
The relevant map is the composite of
$\id_{Z}\boxtimes_A \theta \boxtimes_A \id_{\overline{Z}}$
and
$f\boxtimes_A \id_{Y\boxtimes_{B_2}\overline{Y}} \boxtimes_A \overline{f}$.
The former is cp by Lemma \ref{lem:Z_A Amplification is CP}.
To see that the latter is cp, simply note that the image of $g\otimes \xi \otimes g^\star \in \mathbf{P}_{Z\boxtimes_A Y, n}$ under
$\id_{\bbC^n} \otimes f\boxtimes_A \id_{Y\boxtimes_{B_2}\overline{Y}} \boxtimes_A \overline{f} \otimes \id_{\bbC^n}$ 
is 
$(f\circ g)\otimes \xi \otimes (f\circ g)^\star$, which is visibly in $\mathbf{P}_{W\boxtimes_A Y, n}$.
\end{proof}

Recall from \eqref{eq: standard evaluation and coevaluation morphisms} that for a bimodule ${}_AX_B$ between von Neumann algebras with atomic centers,
the standard solutions  $\ev_X : \overline{X}\boxtimes_A X \to L^2B$ and $\coev_X : L^2A \to X\boxtimes_B \overline{X}$
of the conjugate equations are given by
\begin{equation*}
\coev_X : {\textstyle\sqrt{\psi}\,\mapsto\;\!\! \sqrt{\psi\circ E}}
\quad\text{and}\quad \ev_X=\coev^*_{\overline X},
\end{equation*}
where $E$ is the non-normalised minimal conditional expectation \cite[(6.8)]{MR3342166}.
Given an inclusion of von Neumann algebras $A\subset B$ and a conditional expectation $E:B\to A$, let us write $L^2E$ for the corresponding linear map $L^2A\to L^2B:\sqrt{\psi}\,\mapsto\;\!\! \sqrt{\psi\circ E}$.

\begin{prop}\label{prop: ev and coev are positive}
Let $A, B$ be von Neumann algebras with atomic centers, and let ${}_AX_B$ be a dualizable bimodule.
Then the standard solutions
\[
\ev_X : \overline{X}\boxtimes_A X\to L^2B \qquad\text{and}\qquad \coev_X: L^2A \to X\boxtimes_B \overline{X}
\]
of the conjugate equations satisfy
$\ev_X \in \cP_{\overline{X},L^2B}$ and $\coev_X\in \cP_{L^2A, X}$. 
\end{prop}
\begin{proof}
We only treat the case of $\coev$ (the proof for $\ev$ follows by taking adjoints).
We have canonical isomorphisms
$L^2A^{\oplus n} \boxtimes_A \overline{L^2A}{}^{\oplus n}\cong L^2(A)\otimes M_n(\bbC)$
and
$$
X^{\oplus n} \boxtimes_B \overline{X}{}^{\oplus n}
\cong 
(X\boxtimes_B \overline{X})\otimes M_n(\bbC)
\cong L^2B' \otimes M_n(\bbC),
$$
where the isomorphism $X\boxtimes_B \overline{X}\cong L^2B'$ is as in \cite[Prop.~3.1]{MR703809}.
Under this identification, the map $\id_{\bbC^n}\otimes \coev_X \otimes \id_{\bbC^n}$ corresponds to $(L^2E) \otimes \id_{M_n(\bbC)}$.
Next, we note that the following square of $A$-$A$ bimodule maps commutes: 
\[
\begin{tikzcd}
L^2A \otimes M_n(\bbC)
\ar[rr, "L^2E \otimes \id_{M_n(\bbC)}"]
\ar[d,"\cong"]
&&
L^2B' \otimes M_n(\bbC)
\ar[d,"\cong"]
\\
L^2(A \otimes M_n(\bbC))
\ar[rr, "L^2(E\otimes \id_{M_n(\bbC)})"]
&&
L^2(B' \otimes M_n(\bbC)).
\end{tikzcd}
\]
Indeed, for an element of the form $\xi \otimes e_{11}\in L^2A \otimes M_n(\bbC)$
with $\xi\in L^2_+A$, the two maps visibly agree.
The commutativity follows as all four maps are $M_n(A)$-$M_n(A)$-bilinear.

Finally, since the bottom arrow $L^2(E\otimes \id_{M_n(\bbC)})$ maps $L^2_+(A\otimes M_n(\bbC))$ to $L^2_+(B'\otimes M_n(\bbC))$,
the map $\id_{\bbC^n}\otimes \coev_X \otimes \id_{\bbC^n}=\id_{\bbC^n}\otimes L^2E \otimes \id_{\bbC^n}=L^2E \otimes \id_{M_n(\bbC)}$ sends $\mathbf{P}_{L^2A,n}$ to $\mathbf{P}_{X,n}$.
\end{proof}

\subsection{Positive structure on rigid $\Cstar$-tensor categories}
\label{ex:PosiviteStructureOnMultifusion}

We show that every unitary multifusion category, indeed any semisimple rigid $\Cstar$-tensor category, admits canonical bi-involutive and positive structures.
The notion of cp map presented in this section is originally due to Selinger \cite[Cor.~4.13 \& \S4.4]{10.1016/j.entcs.2006.12.018}.

Let $\cC$ be a semisimple rigid $\Cstar$-tensor category.
Recall from \cite{1808.00323} that a \emph{unitary dual functor} is an assignment $c\mapsto (\overline c,\ev_c,\coev_c)$ of a dual object plus duality data to every object $c\in\cC$,
in such a way that the canonical isomorphisms
\[
\nu_{a,b}:=\big((\ev_a\circ(\id_{\overline a}\otimes\ev_b\otimes\id_a))\otimes\id_{\overline{b\otimes a}}\big)\circ\big(\id_{\overline a\otimes \overline b}\otimes\coev_{b\otimes a}\big)
: \, \overline a\otimes \overline b\to\overline{b\otimes a}
\]
are unitary, and such that for every $f:a\to b$ the conjugate morphism
\[
\overline f:=(\id_{\overline b}\otimes \coev_a^*)\circ(\id_{\overline b}\otimes f^*\otimes\id_{\overline a})\circ(\ev_b^*\otimes\id_{\overline a})\,=\,
\begin{tikzpicture}[baseline=-.1cm]
	\draw (-.6,.9) -- (-.6,-.3) 
	\foreach \p in {-.005,0.12,...,.9} {node[pos=\p, xshift=-1.3, yshift=-3]{-}} 
	arc (-180:0:.3) 
	\foreach \p in {.1,.22,...,1} {node[pos=\p, sloped, rotate=-90, xshift=1.8]{-}} 
	-- (0,.3) arc (180:0:.3) 
	\foreach \p in {.1,.27,...,1} {node[pos=\p, sloped, rotate=-90, xshift=1.8]{-}}
	-- (.6,-.9) 
	\foreach \p in {0,0.125,...,.9} {node[pos=\p, xshift=-1.2, yshift=-3]{-}};
	\roundNbox{unshaded}{(0,0)}{.3cm}{0}{0}{$f^*$};
	\node at (.6+.17,-.85) {$\scriptstyle \overline a$};
	\node at (-.6+.14,.85) {$\scriptstyle \overline b$};
\end{tikzpicture}
\]
satisfies
$\overline f=(\ev_a\otimes\id_{\overline b})\circ(\id_{\overline a}\otimes f^*\otimes\id_{\overline b})\circ(\id_{\overline a}\otimes \coev_b)$
(ensuring the consistency of the graphical calculus \eqref{eq: formula for f bar}).
A unitary dual functor induces a bi-involutive structure on $\cC$ via the structure maps
\begin{equation}\label{eq: bi-involutive structure from unitary dual functor}
\overline{\,\cdot\,}:\cC\to \cC,
\qquad
\nu_{a,b}: \, \overline b\otimes \overline a\to\overline{a\otimes b},
\qquad
r:1\to \overline 1,
\qquad
\varphi_c:c\to\overline{\overline c},
\end{equation}
where $\nu_{a,b}$ is as above, $r$ is $\coev_1$ (followed by a left unitor isomorphism), and 
\[
\varphi_c
:=
(\coev_c^*\otimes \id_{\overline{\overline c}})\circ(\id_{c}\otimes \coev_{\overline{c}})
=
(\id_{\overline{\overline c}}\otimes \ev_c)\circ(\ev_{\overline c}^*\otimes\id_{c}),
\]
where the second equality is \cite[Cor.~3.10]{1808.00323}.

\begin{rem}
As explained in Example~\ref{ex: bi-involutive structure on multifusion category}, any semisimple rigid $\Cstar$-tensor category (for instance a unitary multifusion category) admits a distinguished unitary dual functor, characterised by the conjugate equations \eqref{eq: duality equations (unnormalized)} and \eqref{eq: duality equations (balancing condition)}.
A typical unitary dual functor does not, however, satisfy the balancing condition~\eqref{eq: duality equations (balancing condition)}.
\end{rem}

By \cite[Cor.~B]{1808.00323}, any two unitary dual functors yield canonically equivalent bi-involutive structures.
Specifically, if $(\overline{\,\cdot\,}^1,\ev^1,\coev^1)$ and $(\overline{\,\cdot\,}^2,\ev^2,\coev^2)$ are unitary dual functors,
letting
\begin{equation}\label{eq: chi_c:overline c^1to overline c^2}
\chi_c\,:\,\overline{\;\!c\;\!}^1\,\to\, \overline{\;\!c\;\!}^2
\end{equation}
be the unitary in the polar decomposition of 
$\,\tilde\chi_c :=\, (\ev_{c}^1\otimes \id_{\overline{\;\!c\;\!}^2})\circ (\id_{\overline{\;\!c\;\!}^1}\otimes \coev_c^2)$,
then ${(\mu\equiv \id,i\equiv \id,\chi)}$ is an equivalence of bi-involutive tensor categories.

Given a unitary dual functor on $\cC$, let
\begin{equation}\label{eq: positive cones for multifusion}
\cP_{a,b}:=\big\{\,\theta_f:a\otimes \overline{a}\to b\otimes \overline{b}\;\big|\; c\in \cC,\,\,f:a\otimes c\to b\,\big\},
\end{equation}
\begin{equation}\label{eq: two pics for cp maps in a rigid category}
\hspace{-3.55cm}\text{where}\hspace{.2cm}
\theta_f\,:= (f\otimes \bar f)\circ(\id_a\otimes \coev_c\otimes \id_{\bar a})=\!
\,\begin{tikzpicture}[baseline=-.1cm]
	\draw (-1,-.3) -- (-1,.9)
	\foreach \p in {-.005,0.12,...,.9} {node[pos=\p, xshift=1.8, yshift=2.5]{-}};
	\draw (0,-.3) -- (0,.9)
	\foreach \p in {-.005,0.12,...,.9} {node[pos=\p, xshift=-1.3, yshift=2.5]{-}};
	\draw (-.85,-.3) arc (-180:0:.35) 
	\foreach \p in {.1,.26,...,1} {node[pos=\p, sloped, rotate=-90, xshift=-1.3]{-}};
	\draw (.15,-.9) -- (.15,.3) 
	\foreach \p in {0,0.125,...,.9} {node[pos=\p, xshift=-1.3, yshift=2.2]{-}};
	\draw (-1.15,-.9) -- (-1.15,.3) 
	\foreach \p in {0,0.125,...,.9} {node[pos=\p, xshift=1.8, yshift=2.2]{-}};
	\roundNbox{unshaded}{(0,0)}{.3cm}{.07}{.07}{$\overline f$};
	\roundNbox{unshaded}{(-1,0)}{.3cm}{.07}{.07}{$f$};
	\node at (-1.35,-.7) {$\scriptstyle a$};
	\node at (-.82,-.7) {$\scriptstyle c$};
	\node at (-1.18,.7) {$\scriptstyle b$};
	\node at (.35,-.7) {$\scriptstyle \overline{a}$};
	\node at (.18,.7) {$\scriptstyle \overline{b}$};
\end{tikzpicture}
\;\!=\!\!
\begin{tikzpicture}[baseline=.6cm]
	\roundNbox{unshaded}{(-.03,0)}{.3cm}{.07}{.07}{$f^*$};
	\roundNbox{unshaded}{(-.45,1.5)}{.3cm}{.07}{.07}{$f$};
	\draw (0,-.3) arc (-180:0:.3) 
	\foreach \p in {.13,.29,...,1} {node[pos=\p, sloped, rotate=-90, xshift=-1.2]{-}}
	(.6,-.3) -- (.6,2.13)
	\foreach \p in {0,0.058,...,.98} {node[pos=\p, xshift=-1.2, yshift=2.5]{-}};
	\draw (-.2,.3) arc (0:180:.2)
	\foreach \p in {.13,.29,...,1} {node[pos=\p, sloped, rotate=-90, xshift=-1.2]{-}}
	(-.6,.3) -- (-.6,-.6)
	\foreach \p in {.2,0.36,...,1.1} {node[pos=\p, xshift=-1.3, yshift=2.8]{-}};
	\draw (-.45,1.8) -- (-.45,2.13)
	\foreach \p in {.35,.75} {node[pos=\p, xshift=1.7, yshift=0]{-}};
	\draw (-.3,1.2) arc (-180:-135:.63)
	\foreach \p in {.16,.42,...,1} {node[pos=\p, sloped, rotate=-90, xshift=-1.3]{-}}
	arc (45:0:.63)
	\foreach \p in {.15,.37,...,1} {node[pos=\p, sloped, rotate=-90, xshift=-1.3]{-}};
	\draw (-.6,1.2) arc (0:-45:.63)
	\foreach \p in {.15,.38,...,1} {node[pos=\p, sloped, rotate=-90, xshift=1.7]{-}}
	arc (135:180:.63)
	\foreach \p in {.12,.395,...,1} {node[pos=\p, sloped, rotate=-90, xshift=1.7]{-}};
	\draw (-.969,.31) -- (-.969,-.6)
	\foreach \p in {.2,0.36,...,1.1} {node[pos=\p, xshift=1.7, yshift=2.3]{-}};
	\node at (-1.18,-.51) {$\scriptstyle a$};
	\node at (-.41,-.49) {$\scriptstyle \overline{a}$};
	\node at (-.64,2.05) {$\scriptstyle b$};
	\node at (.77,2.05) {$\scriptstyle \overline{b}$};
	\node at (.08,.9) {$\scriptstyle c$};
\end{tikzpicture}\!\;
\end{equation}
We call the elements of $\cP_{a,b}$ cp maps.
Our next goal is to prove that this is a positive structure on~$\cC$, and that it is independent of the choice of unitary dual functor.

\begin{lem}
\label{lem:RotationOfPositiveOperatorIsCP}
A map $\theta\in \cC(a\otimes \overline{a}, b\otimes \overline{b})$ is cp as in \eqref{eq: positive cones for multifusion} if and only if its one-click rotation
\begin{equation*} 
\rho(\theta)\,
:=\!
\begin{tikzpicture}[baseline=-.1cm]
	\draw (-.6,.9) -- (-.6,-.3) 
	\foreach \p in {-.005,0.12,...,.9} {node[pos=\p, xshift=-1.2, yshift=-3]{\rm -}} 
	arc (-180:0:.2) 
	\foreach \p in {.09,.25,...,1} {node[pos=\p, sloped, rotate=-90, xshift=1.7]{\rm -}}
	(-.2,-.3) -- (-.2,.9)
	\foreach \p in {-.005,0.12,...,.9} {node[pos=\p, xshift=1.7, yshift=2.5]{\rm -}};
	\draw (.2,-.9) -- (.2,.3) 
	\foreach \p in {0,0.125,...,.9} {node[pos=\p, xshift=-1.2, yshift=2.2]{\rm -}}
	arc (180:0:.2) 
	\foreach \p in {.1,.26,...,1} {node[pos=\p, sloped, rotate=-90, xshift=-1.2]{\rm -}} 
	-- (.6,-.9) 
	\foreach \p in {0,0.125,...,.9} {node[pos=\p, xshift=1.8, yshift=-1.5]{\rm -}};
	\roundNbox{unshaded}{(0,0)}{.3cm}{.1}{.1}{$\theta$};
	\node at (.6-.22,-.7) {$\scriptstyle \overline{a}$};
	\node at (.6+.22,-.7) {$\scriptstyle b$};
	\node at (-.6-.26,.7) {$\scriptstyle \overline{a}$};
	\node at (-.6+.22,.7) {$\scriptstyle b$};
\end{tikzpicture}
=
(\id_{\overline{a}\otimes b}\otimes \ev_b)
\circ
(\id_{\overline{a}}\otimes\, \theta \otimes \id_b)
\circ
(\ev_a^*\otimes \id_{\overline{a}\otimes b})
\end{equation*}
is a positive operator in the $\Cstar$-algebra $\cC(\overline{a}\otimes b , \overline{a}\otimes b)$.
Similarly, $\theta$ 
is cp if and only if $\rho^{-1}(\theta):=
(\ev_b\otimes \id_{\overline{b}\otimes a})
\circ
(\id_{\overline{b}}\otimes\, \theta \otimes \id_a)
\circ
(\id_{\overline{b}\otimes a}\otimes \ev^*_a)
$ is positive.
\end{lem}
\begin{proof}
If $\theta = \theta_f$ for some $f: a\otimes c\to b$, then
$
\rho(\theta)
= (\id_{\overline{a}}\otimes f) \circ ((\ev_a^*\circ \ev_a)\otimes \id_b) \circ (\id_{\overline{a}}\otimes f^*) \geq 0.
$
Conversely, if $\theta\in \cC(a\otimes \overline{a}, b\otimes \overline{b})$ is such that $\rho(\theta)\geq 0$, 
then we may write $\rho(\theta)$ as $g\circ g^*$ for some $g:c\to \bar a\otimes b$.
Setting $f:= (\coev_a^* \otimes \id_b) \circ (\id_a \otimes g)$,
we then have $\theta = \theta_f \in \cP_{a,b}$.
The second statement is similar.
\end{proof}

The subsets $\cP_{a,b}\subset \cC(a\otimes \overline{a}, b\otimes \overline{b})$ 
defined in
\eqref{eq: positive cones for multifusion} form a positive structure in the sense of Definition~\ref{def positive structure on bi-involutive tensor category}.
It is straightforward to verify that the $\cP_{a,b}$ contain identities, are closed under composition, closed under adjoints,
and that $f\otimes \theta \otimes \overline{f}$ is cp whenever $\theta$ is cp.
To see that $\cP_{a,b}$ is closed under addition we note that,
by Lemma~\ref{lem:RotationOfPositiveOperatorIsCP}, if $\theta, \psi \in \cP_{a,b}$, then $\rho(\theta)$ and $\rho(\psi)$ are positive, and thus so is
$\rho(\theta)+\rho(\psi)= \rho(\theta+\psi)$.
Hence $\theta + \psi \in \cP_{a,b}$.

Our next task is to show that this positive structure is independent of the choice of unitary dual functor.
Let $(\overline{\,\cdot\,}^1,\ev^1,\coev^1)$ and $(\overline{\,\cdot\,}^2,\ev^2,\coev^2)$ be two unitary dual functors on $\cC$.
Then the equivalence \eqref{eq: chi_c:overline c^1to overline c^2} between the corresponding bi-involutive structures sends cp maps to cp maps:

\begin{prop}\label{prop: independent of the choice of unitary dual functor}
Let $\cC$ be a semisimple rigid $\Cstar$-tensor category equipped with two unitary dual functors.
Let $\cP^1$ and $\cP^2$ be the corresponding sets of cp maps, as defined in \eqref{eq: positive cones for multifusion}.
Then
\[
\theta \in \cP^1_{a,b}
\quad\Longleftrightarrow\quad
(\id_b \otimes \chi_b)\circ \theta \circ (\id_a \otimes \chi^{-1}_a) \in \cP^2_{a,b},
\]
where $\chi_a$ and $\chi_b$ are as in \eqref{eq: chi_c:overline c^1to overline c^2}.
\end{prop}

\begin{proof}
We only prove the ``$\Rightarrow$'' implication
(the other one follows by exchanging the roles of $\cP^1$ and $\cP^2$).
Given $f:a\otimes c\to b$, let
$\theta^\varepsilon_f:=(f\otimes \bar f)\circ(\id_a\otimes \coev^\varepsilon_c\otimes \id_{\bar a})\in\cP^i_{a,b}$, for $\varepsilon=1,2$.
We need to show that
\[
\forall f\;\!{:}\;\!a\otimes c\to b\qquad
(\id_b \otimes \chi_b)\circ \theta^1_f \circ (\id_a \otimes \chi^{-1}_a) \in \cP^2_{a,b}\;\!.
\]
Pick orthogonal direct sum decompositions into simples
$a=\bigoplus a_i$, $b=\bigoplus b_j$, $c=\bigoplus c_k$.
Letting $f_{ijk}:a_i\otimes c_k\to b_j$ be the matrix elements of $f$, so that $f=\sum f_{ijk}$,
we then have $\theta^\varepsilon_f=\sum_{ijk} \theta^\varepsilon_{f_{ijk}}$.
Similarly,  
$\chi_a=\sum_i \chi_{a_i}$ and $\chi_b=\sum_j \chi_{b_j}$.
It follows that 
\begin{alignat*}{1}
(\id_b \otimes \chi_b)\circ \theta^1_f \circ (\id_a \otimes \chi^{-1}_a)&\,=\,
\sum_{ijk} \;(\id_{b_j} \otimes \chi_{b_j})\circ \theta^1_{f_{ijk}} \circ (\id_{a_i} \otimes \chi^{-1}_{a_i})\\
&\,=\, \sum_{ijk} \;[\begin{matrix}\text{\tiny positive}\\[-3mm]\text{\tiny number}\\[.8mm]\end{matrix}]\cdot (\id_{b_j} \otimes \tilde\chi_{b_j})\circ \theta^1_{f_{ijk}} \circ (\id_{a_i} \otimes \tilde\chi^{-1}_{a_i})\\
&\,=\, \sum_{ijk} \;[\begin{matrix}\text{\tiny positive}\\[-3mm]\text{\tiny number}\\[.8mm]\end{matrix}]\,{\cdot}\, \theta^2_{f_{ijk}}\in\, \cP^2_{a,b}
\end{alignat*}\vspace{-6mm}

\noindent
where the last equality is most easily checked by using the definition
$\theta_f=\begin{tikzpicture}[baseline=-.3cm]\node[scale=.5]{\begin{tikzpicture}
	\roundNbox{unshaded}{(-.03,0)}{.3cm}{.07}{.07}{$f^*$};
	\roundNbox{unshaded}{(-.45,1.5)}{.3cm}{.07}{.07}{$f$};
	\draw (0,-.3) arc (-180:0:.3) 
	\foreach \p in {.13,.29,...,1} {node[pos=\p, sloped, rotate=-90, xshift=-1.2]{-}}
	(.6,-.3) -- (.6,2.13)
	\foreach \p in {0,0.058,...,.98} {node[pos=\p, xshift=-1.2, yshift=2.5]{-}};
	\draw (-.2,.3) arc (0:180:.2)
	\foreach \p in {.13,.29,...,1} {node[pos=\p, sloped, rotate=-90, xshift=-1.2]{-}}
	(-.6,.3) -- (-.6,-.6)
	\foreach \p in {.2,0.36,...,1.1} {node[pos=\p, xshift=-1.3, yshift=2.8]{-}};
	\draw (-.45,1.8) -- (-.45,2.13)
	\foreach \p in {.35,.75} {node[pos=\p, xshift=1.7, yshift=0]{-}};
	\draw (-.3,1.2) arc (-180:-135:.63)
	\foreach \p in {.16,.42,...,1} {node[pos=\p, sloped, rotate=-90, xshift=-1.3]{-}}
	arc (45:0:.63)
	\foreach \p in {.15,.37,...,1} {node[pos=\p, sloped, rotate=-90, xshift=-1.3]{-}};
	\draw (-.6,1.2) arc (0:-45:.63)
	\foreach \p in {.15,.38,...,1} {node[pos=\p, sloped, rotate=-90, xshift=1.7]{-}}
	arc (135:180:.63)
	\foreach \p in {.12,.395,...,1} {node[pos=\p, sloped, rotate=-90, xshift=1.7]{-}};
	\draw (-.969,.31) -- (-.969,-.6)
	\foreach \p in {.2,0.36,...,1.1} {node[pos=\p, xshift=1.7, yshift=2.3]{-}};
\end{tikzpicture}};\end{tikzpicture}
$
of $\theta_f$.
\end{proof}

The cone $\cP_{1,a}\subset \cC(1, a\otimes\overline a)$ is self-dual in the following sense.
Let $\cC$ be a semisimple rigid $\Cstar$-tensor category, and let $\varphi$ be a faithful state on the finite dimensional abelian $\Cstar$-algebra $\End_\cC(1)$.

\begin{lem}
\label{lem:SelfDualityOfCPCone}
Let $\theta : 1 \to a\otimes\overline a $ be any morphism.
Assume that for all $\theta' \in \cP_{a,1}$ we have $\varphi(\theta' \circ \theta) \ge 0$.  Then $\theta \in \cP_{1,a}$.
\end{lem}
\begin{proof}
The map $\mathrm{Tr}:x,y\mapsto\varphi(\tr_\cC(x\circ y))$ is a faithful trace on $\End(a)$.
By Lemma~\ref{lem:RotationOfPositiveOperatorIsCP}, the one-click rotation $\rho^{-1}:\cC(a\otimes\overline a, 1)\to \End(a)$
induces a bijection between $\cP_{a,1}$ and the set of positive elements of $\End(a)$.
By assumption,
\[
\varphi(\theta' \circ \theta) = \mathrm{Tr}\big( \rho(\theta) \circ \rho^{-1}(\theta') \big) \ge 0
\]
for all $\theta'\in\cP_{a,1}$, i.e., $\mathrm{Tr}( \rho(\theta) \circ x)\ge 0$ for all positive $x\in \End(a)$.
It follows that $\rho(\theta)$ is positive in $\End(a)$.
Hence, by Lemma~\ref{lem:RotationOfPositiveOperatorIsCP}, $\theta$ is cp.
\end{proof}

\begin{rem}
The above results are all formulated in the context of semisimple rigid $\Cstar$-tensor categories.
However, they apply verbatim to any ``unitary 2-category'' (the many-object version of a semisimple rigid $\Cstar$-tensor category).
\end{rem}

Fix a von Neumann algebra $A$ with atomic center.
If we apply the construction described in this section to the dualizable subcategory $\Bim_d(A)\subset\Bim(A)$, then the resulting positive structure agrees with the one inherited from $\Bim(A)$:

\begin{prop}
\label{prop:PositiveStructureOnDualizableSubcategoryOfBim(R)}
The positive structure \eqref{eq: positive cones for multifusion} on the dualizable subcategory
$\Bim_d(A)\subset \Bim(A)$ agrees with the restriction of the positive structure on $\Bim(A)$ from Definition~\ref{def: cp morphisms in Bim(R)}.
\end{prop}
\begin{proof}
Let $\cP^{1}$ denote the positive structure \eqref{eq: positive cones for multifusion}, and $\cP^{2}$ the one from Definition~\ref{def: cp morphisms in Bim(R)}.
The former is generated by the maps $\coev_X$ in the sense that it is the smallest positive structure which contains those maps.
By Proposition~\ref{prop: ev and coev are positive}, $\coev_X\in\cP^{2}$. Hence $\cP^{1} \subseteq \cP^{2}$.

By Lemma~\ref{lem: one sided Frobenius reciprocity in C_+}, we have
$\cP^{i}_{X,Y}\cong \cP^{i}_{L^2A,\overline{X}\otimes Y}$ for both $i=1,2$, and the isomorphism is provided by the same map.
It is therefore enough to show that
\[
\cP^{2}_{L^2A,X}\subset \cP^{1}_{L^2A,X}
\]
for every $X\in \Bim_d(A)$.

Fix $\theta \in \cP^{2}_{L^2A,X}$.
By Lemma \ref{lem:SelfDualityOfCPCone}, it suffices to show that $\forall\theta' \in \cP^{1}_{X,L^2A}$
the composite $\theta' \circ \theta$ is positive in $\End(L^2A)=Z(A)$.
We'll show that the inequality
\[
\theta' \circ \theta \geq 0
\]
holds true for every $\theta \in \cP^{2}_{L^2A,X}$ and $\theta' \in \cP^{2}_{X,L^2A}$.
Since $(\theta' \circ \theta)L^2_+A\subset L^2_+A$, we have
\(
\langle (\theta' \circ \theta)\xi,\xi\rangle\ge 0
\)
for every $\xi \in L^2_+A$.
By varying the central support of $\xi$, it follows that $\theta' \circ \theta \ge 0$.
\end{proof}

We finish the section by noting that
our notion \eqref{eq: positive cones for multifusion} of cp morphism agrees with the 
notion of cp map between $\rm C^*$-algebra objects/Q-systems introduced in 
\cite[Def.~4.20]{MR3687214}, generalizing the cp multipliers of~\cite{MR3406647}.
Recall that a map $\theta : a\otimes \overline{a} \to b \otimes \overline{b}$ is cp in the sense of \cite{MR3687214} if for every $d\in \cC$ and every positive map $g\in \cC(d\otimes a, d\otimes a)$, the morphism
\begin{equation}
\label{3rd def of cp maps}
\begin{tikzpicture}[baseline=.4cm]
	\draw (-.2,-.7) -- (-.2,1.7)
	\foreach \p in {0,0.06,...,1} {node[pos=\p, xshift=1.8, yshift=1.5]{-}} ;
	\draw (.2,1.7) -- (.2,-.4)
	\foreach \p in {0,0.06,...,1} {node[pos=\p, xshift=1.7, yshift=-1.7]{-}}
	arc (-180:0:.2cm) 
	\foreach \p in {.19,.42,...,1} {node[pos=\p, sloped, rotate=-90, xshift=-1.2]{-}}
	-- (.6,1.4) 
	\foreach \p in {0,0.06,...,1} {node[pos=\p, xshift=-1.2, yshift=1.5]{-}}
	arc (180:0:.2cm) 
	\foreach \p in {.12,.28,...,1} {node[pos=\p, sloped, rotate=-90, xshift=-1.2]{-}}
	-- (1,-.7)
	\foreach \p in {0,0.06,...,1} {node[pos=\p, xshift=1.8, yshift=-1.5]{-}};
	\roundNbox{unshaded}{(0,0)}{.3cm}{.1}{.1}{$g$};
	\roundNbox{unshaded}{(.4,1)}{.3cm}{.1}{.1}{$\theta$};
	\node at (-.35,-.5) {$\scriptstyle d$};
	\node at (.05,-.5) {$\scriptstyle a$};
	\node at (.85,-.5) {$\scriptstyle b$};
	\node at (.05,.5) {$\scriptstyle a$};
	\node at (.75,.5) {$\scriptstyle \overline{a}$};
	\node at (-.35,1.5) {$\scriptstyle d$};
	\node at (.05,1.5) {$\scriptstyle b$};
\end{tikzpicture}
=
(\id_{d\otimes b} \otimes \ev_b)
\circ
(\id_d \otimes\;\! \theta \otimes \id_b)
\circ
(g\otimes \id_{\overline{a}\otimes b})
\circ
(\id_d \otimes \coev_a \otimes \id_b)
\end{equation}
is positive.
If $\theta = \theta_f 
\in \cP_{a,b}$,
then for every $g\ge 0$ the morphism \eqref{3rd def of cp maps} can be written as 
$
(\id_d\otimes f) \circ (\sqrt g\otimes \id_c)
\circ
(\sqrt g\otimes \id_c)
\circ
(\id_d \otimes f^*)
$,
and is thus clearly positive.
Conversely, if $\theta \in \cC(a\otimes \overline{a}, b\otimes \overline{b})$ is cp in the sense of \cite{MR3687214}, 
then setting $d=\overline{a}$ and $g = \ev_a^* \circ \ev_a$ in \eqref{3rd def of cp maps}, we get $\rho(\theta)\geq 0$, thus $\theta\in \cP_{a,b}$ by Lemma \ref{lem:RotationOfPositiveOperatorIsCP}.

\subsection{Positive representations}
\label{sec:PositiveRepresentations}

Let $\cC$ and $\cD$ be bi-involutive tensor categories equipped with positive structures.
A bi-involutive functor $F:\cC\to \cD$ is said to \emph{respect the positive structures} if for every cp map $\theta:a\otimes\bar a\to b\otimes\bar b$ in $\cC$,
the following composite is cp in $\cD$:
\begin{equation}
\label{eq: respect the positive structures}
F(a)\otimes \overline{F(a)} \xrightarrow{\!\id\otimes \chi_a^{-1}\!\!}
F(a)\otimes F(\bar a) \xrightarrow{\!\mu_{a,\bar a}\!}
F(a\otimes \bar a) \xrightarrow{\!F(\theta)\!}
F(b\otimes \bar b) \xrightarrow{\!\mu^{-1}_{b,\bar b}\!}
F(b)\otimes F(\bar b) \xrightarrow{\!\id\otimes \chi_b}
F(b)\otimes \overline{F(b)}.
\end{equation}

\begin{defn}\label{def: positive representation}
Let $\cC$ be as above.
A \emph{positive representation} of $\cC$ is a bi-involutive functor $\cC\to \Bim(R)$ (as in Definition~\ref{def: rep of a bi-involutive tensor category})
which respects the positive structures.
Here, $\Bim(R)$ is equipped with the positive structure described in Section~\ref{ex:PosiviteStructureOnBim}.
\end{defn}

By the results of the previous section, a semisimple rigid $\Cstar$-tensor category $\cC$ admits canonical bi-involutive and positive structures \eqref{eq: bi-involutive structure from unitary dual functor} and \eqref{eq: positive cones for multifusion}
(remember that, by \eqref{eq: chi_c:overline c^1to overline c^2} and Proposition~\ref{prop: independent of the choice of unitary dual functor}, these are independent of the choice of unitary dual functor on $\cC$).
As we have seen, given a tensor functor $\alpha:\cC\to \Bim(R)$, there are many layers of structure that one may or may not require this functor to preserve.
Specifically, one could require $\alpha$ to be:
\begin{itemize}
\item[(i)] a $\Cstar$-representation (Definition~\ref{def:Cstar-representation}),
\item[(ii)] a bi-involutive representation (Definition~\ref{def: rep of a bi-involutive tensor category}),
\item[(iii)] or, finally, a positive representation (Definition~\ref{def: positive representation}).
\end{itemize}

Surprisingly, at least when $\alpha$ is fully faithful, options (i) and (iii) yield equivalent notions.
In contrast, option (ii) yields a non-equivalent, and less well-behaved notion.
The equivalence between (i) and (iii) is formalised in the following theorem:
\begin{thm}[Thm.~\ref{THM-A}]
\label{thm: (i) and (iii) are equivalent} 
Let $\cC$ be a semisimple rigid $\Cstar$-tensor category.
Then every fully faithful $\Cstar$-representation $\alpha:\cC\to \Bim(R)$ extends uniquely to a positive representation.

Given fully faithful $\Cstar$-representation $\alpha_i:\cC\to \Bim(R_i)$, $i=1,2$, then every isomorphism $\alpha_1\cong \alpha_2$ of $\Cstar$-representations
(an invertible bimodule ${}_{R_2}\Phi_{R_1}$ together with a unitary monoidal natural isomorphism $\phi$ as in \eqref{forgotten coherence - invertible case - A})
is an isomorphism of positive representations. I.e., the coherence~\eqref{forgotten coherence - invertible case} is automatically satisfied.
\end{thm}
\begin{proof}
This is the content of Lemmata~\ref{lem:UniquePositiveCoherenceForTensorDaggerFunctor} and~\ref{lem: coherence automatically satisfied} below.
\end{proof}

It is natural to ask whether the statement of Theorem~\ref{thm: (i) and (iii) are equivalent} also holds true without the requirement of $\alpha$ being fully faithful.
We leave this as an open question (we do not know of any counterexamples).

\begin{rem}
Recall that an isomorphism of bi-involutive representations is a pair $(\Phi, \phi)$ as in Definition~\ref{def:Cstar-representation},
where $\Phi$ is an invertible 
bimodule, and $\phi$ satisfies the coherences \eqref{eq: half-braiding condition} and \eqref{forgotten coherence},
equivalently \eqref{forgotten coherence - invertible case - A} and~\eqref{forgotten coherence - invertible case}. 
One easly checks that if $\alpha: \cC \to \Bim(R)$ and $\beta: \cC\to \Bim(S)$ are isomorphic representations,
then $\alpha$ is positive if and only if $\beta$ is positive (the proof relies on the coherence~\eqref{forgotten coherence - invertible case}).
\end{rem}

Let now $\cC$ be a semisimple rigid $\Cstar$-tensor category (for example a unitary multifusion category), equipped with its canonical bi-involutive and positive structures \eqref{eq: bi-involutive structure from unitary dual functor} and \eqref{eq: positive cones for multifusion}.

\begin{lem}\label{lem:UniquePositiveCoherenceForTensorDaggerFunctor}
Let $\alpha:\cC\to\Bim(R)$ be a fully faithful $\Cstar$-representation. 
Then
\begin{alignat}{1}
\label{eq:FormulaForChi}
\chi_c &:= 
\big((\alpha(\ev^\cC_c)\circ \mu_{\overline{c},c})\otimes\id_{\overline{\alpha(c)}}\big)
\circ
\big(\id_{\alpha(\overline{c})}\otimes \coev^{\Bim(R)}_{\alpha(c)}\big)
\\
\label{eq:EquivalentFormulaForChi}
&\phantom{:}=\big(\id_{\overline{\alpha(c)}} \otimes (\alpha(\coev^\cC_c)^*\circ \mu_{c, \overline{c}})\big)
\circ
\big((\ev_{\alpha(c)}^{\Bim(R)})^*\otimes \id_{\alpha(\overline{c})}\big),
\end{alignat}
and these equip $\alpha$ with the structure of a positive representation.

Conversely, if $\alpha$ is a positive representation which is furthermore fully faithful,
then $\chi_c:\alpha(\overline c)\to\overline{\alpha(c)}$ is given by \eqref{eq:FormulaForChi}, equivalently~\eqref{eq:EquivalentFormulaForChi}.
\end{lem}

\begin{proof}
Let $\chi$ be as in~\eqref{eq:FormulaForChi}.
By Lemma~\ref{lem:FullyFaithfulPreservesStandardPairings}, it is unitary and agrees with~\eqref{eq:EquivalentFormulaForChi}.
We show that $\alpha$ respects positive structures, i.e., that the morphism \eqref{eq: respect the positive structures} is cp.
Indeed, for any $f: a\otimes c \to b$ in $\cC$,
$$
(\id_{\alpha(b)}\otimes \chi_b)
\circ
\mu_{b, \overline{b}}^{-1}
\circ
\,\alpha\left(
\begin{tikzpicture}[baseline=.6cm]
\useasboundingbox (-1.3,-.6) rectangle (.9,2.1);
	\roundNbox{unshaded}{(-.03,0)}{.3cm}{.07}{.07}{$f^*$};
	\roundNbox{unshaded}{(-.45,1.5)}{.3cm}{.07}{.07}{$f$};
	\draw (0,-.3) arc (-180:0:.3) 
	\foreach \p in {.13,.29,...,1} {node[pos=\p, sloped, rotate=-90, xshift=-1.2]{-}}
	(.6,-.3) -- (.6,2.13)
	\foreach \p in {0,0.058,...,.98} {node[pos=\p, xshift=-1.2, yshift=2.5]{-}};
	\draw (-.2,.3) arc (0:180:.2)
	\foreach \p in {.13,.29,...,1} {node[pos=\p, sloped, rotate=-90, xshift=-1.2]{-}}
	(-.6,.3) -- (-.6,-.6)
	\foreach \p in {.2,0.36,...,1.1} {node[pos=\p, xshift=-1.3, yshift=2.8]{-}};
	\draw (-.45,1.8) -- (-.45,2.13)
	\foreach \p in {.35,.75} {node[pos=\p, xshift=1.7, yshift=0]{-}};
	\draw (-.3,1.2) arc (-180:-135:.63)
	\foreach \p in {.16,.42,...,1} {node[pos=\p, sloped, rotate=-90, xshift=-1.3]{-}}
	arc (45:0:.63)
	\foreach \p in {.15,.37,...,1} {node[pos=\p, sloped, rotate=-90, xshift=-1.3]{-}};
	\draw (-.6,1.2) arc (0:-45:.63)
	\foreach \p in {.15,.38,...,1} {node[pos=\p, sloped, rotate=-90, xshift=1.7]{-}}
	arc (135:180:.63)
	\foreach \p in {.12,.395,...,1} {node[pos=\p, sloped, rotate=-90, xshift=1.7]{-}};
	\draw (-.969,.31) -- (-.969,-.6)
	\foreach \p in {.2,0.36,...,1.1} {node[pos=\p, xshift=1.7, yshift=2.3]{-}};
	\node at (-1.18,-.51) {$\scriptstyle a$};
	\node at (-.41,-.49) {$\scriptstyle \overline{a}$};
	\node at (-.64,2.05) {$\scriptstyle b$};
	\node at (.75,2.05) {$\scriptstyle \overline{b}$};
	\node at (.08,.9) {$\scriptstyle c$};
\end{tikzpicture}
\right)
\circ \mu_{a,\overline{a}}
\circ
(\id_{\alpha(a)}\otimes \chi_a^{-1})
\,=
\begin{tikzpicture}[baseline=.6cm]
\useasboundingbox (-1.3,-.6) rectangle (.9,2.1);
	\roundNbox{unshaded}{(-.03,0)}{.3cm}{.11}{.11}{$\hspace{.3mm}\scriptstyle{\alpha(f)^*}$};
	\roundNbox{unshaded}{(-.45,1.5)}{.3cm}{.1}{.1}{$\scriptstyle{\alpha(f)}$};
	\draw (0,-.3) arc (-180:0:.3) 
	\foreach \p in {.13,.29,...,1} {node[pos=\p, sloped, rotate=-90, xshift=-1.2]{-}}
	(.6,-.3) -- (.6,2.13)
	\foreach \p in {0,0.058,...,.98} {node[pos=\p, xshift=-1.2, yshift=2.5]{-}};
	\draw (-.2,.3) arc (0:180:.2)
	\foreach \p in {.13,.29,...,1} {node[pos=\p, sloped, rotate=-90, xshift=-1.2]{-}}
	(-.6,.3) -- (-.6,-.6)
	\foreach \p in {.2,0.36,...,1.1} {node[pos=\p, xshift=-1.3, yshift=2.8]{-}};
	\draw (-.45,1.8) -- (-.45,2.13)
	\foreach \p in {.35,.75} {node[pos=\p, xshift=1.7, yshift=0]{-}};
	\draw (-.3,1.2) arc (-180:-135:.63)
	\foreach \p in {.16,.42,...,1} {node[pos=\p, sloped, rotate=-90, xshift=-1.3]{-}}
	arc (45:0:.63)
	\foreach \p in {.15,.37,...,1} {node[pos=\p, sloped, rotate=-90, xshift=-1.3]{-}};
	\draw (-.6,1.2) arc (0:-45:.63)
	\foreach \p in {.15,.38,...,1} {node[pos=\p, sloped, rotate=-90, xshift=1.7]{-}}
	arc (135:180:.63)
	\foreach \p in {.12,.395,...,1} {node[pos=\p, sloped, rotate=-90, xshift=1.7]{-}};
	\draw (-.969,.31) -- (-.969,-.6)
	\foreach \p in {.2,0.36,...,1.1} {node[pos=\p, xshift=1.7, yshift=2.3]{-}};
	\node at (-1.15,-.82) {$\scriptstyle \alpha(a)$};
	\node at (-.45,-.8) {$\scriptstyle \overline{\alpha(a)}$};
	\node at (-.78,2.05) {$\scriptstyle \alpha(b)$};
	\node at (.95,2.05) {$\scriptstyle \overline{\alpha(b)}$};
	\node at (.18,.95) {$\scriptstyle \alpha(c)$};
\end{tikzpicture}\,\,\in\;\! \cP^{\Bim(R)}_{\!\alpha(a),\alpha(b)}
$$
by Proposition \ref{prop:PositiveStructureOnDualizableSubcategoryOfBim(R)}.

Suppose now that $\alpha:\cC\to\Bim(R)$ is a fully faithful positive representation.
Since $\alpha$ respects the positive structures, for every object $c\in\cC$, we have
\[
\tilde\coev_{\alpha(c)}:=(\id_{\alpha(c)}\otimes \chi_c) \circ \mu^{-1}_{c,\overline{c}} \circ \alpha(\coev_c) \;\in\, \cP^{\Bim(R)}_{1,\alpha(c)}.
\]
By Lemma~\ref{lem:FullyFaithfulPreservesStandardPairings}, since $\chi_c$ is unitary, $\tilde\coev_{\alpha(c)}$ is one half of a standard duality pairing (one half of a balanced solution of the duality equations).

Let us assume for a moment that $c$ is simple.
Since $\alpha$ is fully faithful, $\alpha(c)$ is then also simple.
So the only morphisms $1\to \alpha(c)\otimes\overline{\alpha(c)}$
which fit into a standard duality pairing are those of the form $\lambda\cdot\coev_{\alpha(c)}$ for $\lambda\in U(1)$.
Exactly one of them is cp. So $\tilde\coev_{\alpha(c)}=\coev_{\alpha(c)}$.
Now, both $\tilde\coev$ and $\coev$ are compatible with direct sums.
So this last equation holds true for every $c\in\cC$, not just the simples.
Finally, using that $\tilde\coev_{\alpha(c)}=\coev_{\alpha(c)}$, we compute
\[
\id=(\ev_{\alpha(c)}\otimes \id) \circ (\id\otimes \tilde\coev_{\alpha(c)})=
\chi_c \circ \Big[\big(\ev_{\alpha(c)}\otimes \id\big) \circ \big(\id\otimes(\mu^{-1}_{c,\overline{c}} \circ \alpha(\coev_c))\big)\Big].
\]
It follows that $\chi_c=\eqref{eq:FormulaForChi}=\eqref{eq:EquivalentFormulaForChi}$.
\end{proof}

Let $\cC$ be a semisimple rigid $\Cstar$-tensor category.
Once again, we equip $\cC$ with its canonical bi-involutive and positive structures, described in Section~\ref{ex:PosiviteStructureOnMultifusion}.

\begin{lem}\label{lem: coherence automatically satisfied}
Let $\alpha : \cC \to \Bim(R)$ and $\beta: \cC \to \Bim(S)$ be positive fully faithful representations.
And let $\Phi$ be a morphism between the underlying $\Cstar$-representations (a bimodule ${}_S\Phi_R$ together with 
unitary natural isomorphisms $\phi_c: \Phi\boxtimes_R \alpha(c) \to \beta(c)\boxtimes_S \Phi$ satisfying the half-braiding condition \eqref{eq: half-braiding condition}).
Then the coherence \eqref{eq: when C is rigid, that condition is equivalent to...} is automatically satisfied.

In particular, under the above assumptions, when ${}_S\Phi_R$ is dualizable, the coherence \eqref{forgotten coherence} is automatically satisfied.
\end{lem}

\begin{proof}
Let $\tilde\ev^\alpha:=\alpha(\coev_c)^*\circ\mu^\alpha_{c,\overline{c}}$, and $\tilde\ev^\beta:=\beta(\coev_c)^*\circ\mu^\beta_{c,\overline{c}}$. Then we have
\[
\;\!\!\!\begin{tikzpicture}[baseline=.4cm, xscale=-1, scale=.95]
  \draw (1,1.7) node [left, yshift = -.14cm, xshift = .00cm] {$\scriptstyle\overline{\beta(c)}$} -- (1,.3) arc (0:-180:.3cm) -- (.4,.3)node [left, yshift = .35cm, xshift = .22cm, scale=1.1] {$\scriptscriptstyle \beta(c)$} .. controls ++(90:.5cm) and ++(270:.5cm) .. (-.4,1.3)node [left, yshift = .0cm, xshift = .1cm, scale=1.1] {$\scriptscriptstyle\alpha(c)$} arc (0:180:.3cm) node [right, yshift = -.4cm, xshift = -.02cm, scale=.95] {$\scriptstyle\overline{\alpha(c)}$} -- (-1,-.7) node [right, yshift = .12cm, xshift = -.02cm] {$\scriptstyle\alpha(\overline{c})$} ;  
  \draw[super thick, white] (-.4,-.7) -- (-.4,.3) .. controls ++(90:.5cm) and ++(270:.5cm) .. (.4,1.3) -- (.4,1.7);  
  \draw (-.4,-.7) node [left, yshift = .1cm, xshift = .03cm] {$\scriptstyle\Phi$} -- (-.4,.3) .. controls ++(90:.5cm) and ++(270:.5cm) .. (.4,1.3) -- (.4,1.7);  
  \roundNbox{unshaded}{(-1,.15)}{.27}{0}{0}{$\scriptstyle \chi^\alpha_c$};
\end{tikzpicture}
=\!\!\!\!
\begin{tikzpicture}[baseline=.5cm, xscale=-1, scale=.95]
  \draw (1,1.8) node [left, yshift = -.14cm, xshift = -.00cm] {$\scriptstyle\overline{\beta(c)}$} -- (1,.3) arc (0:-180:.3cm) -- (.4,.3)node [left, yshift = .35cm, xshift = .22cm, scale=1.1] {$\scriptscriptstyle \beta(c)$} .. controls ++(90:.5cm) and ++(270:.5cm) .. (-.5,1.3)node [left, yshift = -.1cm, xshift = .05cm, scale=1.1] {$\scriptscriptstyle\alpha(c)$} arc (0:180:.2cm)
  -- (-.9,-.5) node [right, yshift = .12cm, xshift = -.02cm] {$\scriptstyle\alpha(\overline{c})$} ;  
  \draw[super thick, white] (-.4,-.7) -- (-.4,.3) .. controls ++(90:.5cm) and ++(270:.5cm) .. (.4,1.3) -- (.4,1.8);  
  \draw (-.4,-.5) node [left, yshift = .1cm, xshift = .03cm] {$\scriptstyle\Phi$} -- (-.4,.3) .. controls ++(90:.5cm) and ++(270:.5cm) .. (.4,1.3) -- (.4,1.8);  
  \roundNbox{unshaded}{(-.7,1.6)}{.27}{.07}{.07}{$\scriptstyle \tilde\ev^\alpha$};
\end{tikzpicture}
\!\!\!\!=\!\!\!\!
\begin{tikzpicture}[baseline=.4cm, xscale=-1, scale=.95]
  \draw (1,1.7) node [left, yshift = -.14cm, xshift = .00cm] {$\scriptstyle\overline{\beta(c)}$} -- (1,.3) arc (0:-180:.3cm) -- (.4,.3)node [left, yshift = .35cm, xshift = .22cm, scale=1.1] {$\scriptscriptstyle \beta(c)$} .. controls ++(90:.5cm) and ++(270:.4cm) .. (-.5,1.3)node [left, yshift = -.1cm, xshift = .05cm, scale=1.1] {$\scriptscriptstyle\alpha(c)$} arc (0:180:.2cm)
  .. controls ++(270:.6cm) and ++(90:.4cm) ..
  (-.02,.2) node [left, yshift = -.35cm, xshift = .22cm, scale=1.1] {$\scriptscriptstyle \beta(\overline c)$} .. controls ++(270:.4cm) and ++(90:.3cm) .. (-.8,-.7) node [right, yshift = .12cm, xshift = -.02cm] {$\scriptstyle\alpha(\overline{c})$} ;  
  \draw[super thick, white] (-.4,-.7) -- (-.4,.3) .. controls ++(90:.5cm) and ++(270:.5cm) .. (.4,1.3) -- (.4,1.7);  
  \draw (-.4,-.7) node [left, yshift = .1cm, xshift = .03cm] {$\scriptstyle\Phi$} -- (-.4,.3) .. controls ++(90:.5cm) and ++(270:.5cm) .. (.4,1.3) -- (.4,1.7);  
  \roundNbox{unshaded}{(-.7,1.6)}{.27}{.07}{.07}{$\scriptstyle \tilde\ev^\alpha$};
\end{tikzpicture}
\!\!\!\!=\!\!\!
\begin{tikzpicture}[baseline=.4cm, xscale=-1]
  \draw (1,1.5) node [left, yshift = -.14cm, xshift = .00cm] {$\scriptstyle\overline{\beta(c)}$} -- (1,.3) arc (0:-180:.3cm) -- (.4,.3)node [left, yshift = .05cm, xshift = .12cm, scale=.95] {$\scriptscriptstyle \beta(c)$} -- ++(0,.3) arc (0:180:.2cm) -- ++(0,-.3) node [left, yshift = -.42cm, xshift = .28cm, scale=1.05] {$\scriptscriptstyle \beta(\overline c)$} .. controls ++(270:.4cm) and ++(90:.3cm) .. (-.8,-.7) node [right, yshift = .12cm, xshift = -.02cm] {$\scriptstyle\alpha(\overline{c})$} ;  
  \draw[super thick, white] (-.4,-.7) -- (-.4,.3) -- (-.4,1.5);  
  \draw (-.4,-.7) node [left, yshift = .1cm, xshift = .03cm] {$\scriptstyle\Phi$} -- (-.4,.3) -- (-.4,1.5);  
  \roundNbox{unshaded}{(.2,.8)}{.27}{.07}{.07}{$\scriptstyle \tilde\ev^\beta$};
\end{tikzpicture}
\!\!\!\!\!\!=
\begin{tikzpicture}[baseline=-.4cm]
  \draw (.4,-1.5) node [right, yshift = .1cm, xshift = -.02cm] {$\scriptstyle\alpha(\overline{c})$} -- (.4,-1.3) .. controls ++(90:.5cm) and ++(270:.5cm) .. (-.4,-.3) node [left, yshift = -.25cm, xshift = .05cm] {$\scriptstyle\beta(\overline c)$} -- (-.4,.7) node [left, yshift = -.1cm, xshift = .00cm] {$\scriptstyle\overline{\beta(c)}$};  
  \draw[super thick, white] (-.4,-1.5) -- (-.4,-1.3) .. controls ++(90:.5cm) and ++(270:.5cm) .. (.4,-.3) -- (.4,.7);  
  \draw (-.4,-1.5) node [left, yshift = .12cm, xshift = .02cm] {$\scriptstyle\Phi$} -- (-.4,-1.3) .. controls ++(90:.5cm) and ++(270:.5cm) .. (.4,-.3) -- (.4,.7);  
  \roundNbox{unshaded}{(-.4,0)}{.27}{0}{0}{$\scriptstyle {}_{\phantom{i}}\chi^\beta_{c\phantom{\beta}}$};
\end{tikzpicture}
\]
where the first and last equalities hold by the second formula for $\chi$ in Lemma~\ref{lem:UniquePositiveCoherenceForTensorDaggerFunctor}, and
the third equality holds by the half-braiding condition \eqref{eq: half-braiding condition} followed by the naturality of $\phi$.
\end{proof}

From now on, given a semisimple rigid $\Cstar$-tensor category
$\cC$, by a \emph{fully faithful representation} of $\cC$, we shall always mean a fully faithful $\Cstar$-representation,
equivalently a fully faithful positive representation.
By Theorem~\ref{thm: (i) and (iii) are equivalent}, two fully faithful representations are equivalent as $\Cstar$-representations if and only if the are equivalent as positive representation.

%% file: Chapters/Commutants.tex

\section{Commutant categories}

The present section is devoted to proof of Theorem \ref{thm:MoritaEquivalentCommutantsEquivalent}.
The latter states that if $\cC_0$ and $\cC_1$ are Morita equivalent unitary fusion categories
equipped with fully faithful representations $\cC_0\to \Bim(R_0)$ and $\cC_1\to \Bim(R_1)$, 
where $R_0$ and $R_1$ are hyperfinite factors which are either both of type ${\rm II}$ or both of type ${\rm III}_1$,
then the commutant categories $\cC_0'$ and $\cC_1'$ are equivalent as bi-involutive tensor categories with positive structures.

\subsection{Bicommutant categories}
\label{sec: Bicommutant categories}

We start by recalling and correcting the notion of bicommutant category from \cite{MR3747830} and \cite[\S3]{MR3663592} (the correction does not affect any of the previous results).
The new feature, which was not present in \cite{MR3747830, MR3663592}, is that of a positive structure (Def.~\ref{def positive structure on bi-involutive tensor category}).

Let $\cC$ be a bi-involutive tensor category equipped with a positive structure, and let $\alpha:\cC \to \Bim(R)$ be a positive representation.
The \emph{commutant category} $\cC'$ 
is the category whose objects are pairs $(X, e_X)$ with $X\in \Bim(R)$ and $e_X= (e_{X,c})_{c\in\cC}$ a unitary half-braiding
$$
\begin{tikzpicture}[baseline=-.1cm]
	\coordinate (a) at (-.4,-.4);
	\coordinate (b) at (.4,-.4);
	\coordinate (c) at (.4,.4);
	\coordinate (d) at (-.4,.4);
	\draw (b) .. controls ++(90:.2cm) and ++(-30:.1cm) .. (.1,-.05);
	\draw (d) .. controls ++(270:.2cm) and ++(150:.1cm) .. (-.1,.05);
	\draw[very thick] (a) .. controls ++(90:.4cm) and ++(270:.4cm) .. (c);
	\node at ($ (a) + (0,-.2) $) {\scriptsize{$X$}};
	\node at ($ (b) + (0,-.2) $) {\scriptsize{$c$}};
	\node at ($ (d) + (0,.2) $) {\scriptsize{$c$}};
	\node at ($ (c) + (0,.2) $) {\scriptsize{$X$}};
\end{tikzpicture}
\qquad\,\,\,
e_{X,c}:X\boxtimes \alpha(c) \to \alpha(c)\boxtimes X,
$$
natural in $c$, and subject to the `hexagon' axiom:
\[
\qquad
\begin{tikzpicture}[baseline=-.1cm]
	\coordinate (a) at (-.4,-.4);
	\coordinate (b) at (.4,-.4);
	\coordinate (c) at (.4,.4);
	\coordinate (d) at (-.4,.4);
	\draw (b) .. controls ++(90:.2cm) and ++(-30:.1cm) .. (.1,-.05);
	\draw (d) .. controls ++(270:.2cm) and ++(150:.1cm) .. (-.1,.05);
	\draw[very thick] (a) .. controls ++(90:.4cm) and ++(270:.4cm) .. (c);
	\node at ($ (a) + (0,-.2) $) {\scriptsize{$X$}};
	\node at ($ (b) + (0,-.2) $) {\scriptsize{$c\otimes d$}};
	\node at ($ (d) + (0,.2) $) {\scriptsize{$c\otimes d$}};
	\node at ($ (c) + (0,.2) $) {\scriptsize{$X$}};
\end{tikzpicture}
=
\begin{tikzpicture}[baseline=-.1cm]
	\coordinate (a1) at (-.8,-.8);
	\coordinate (a2) at (0,-.8);
	\coordinate (a3) at (.8,-.8);
	\coordinate (b1) at ($ (a1) + (0,.8) $);
	\coordinate (b2) at ($ (a2) + (0,.8) $);
	\coordinate (b3) at ($ (a3) + (0,.8) $);
	\coordinate (c1) at ($ (b1) + (0,.8) $);
	\coordinate (c2) at ($ (b2) + (0,.8) $);
	\coordinate (c3) at ($ (b3) + (0,.8) $);
	\draw (a3) -- (b3) .. controls ++(90:.2cm) and ++(-30:.1cm) .. (.5,.35);
	\draw (c2) .. controls ++(270:.2cm) and ++(150:.1cm) .. (.3,.45);
	\draw (c1) -- (b1) .. controls ++(270:.2cm) and ++(150:.1cm) .. (-.5,-.35);
	\draw (a2) .. controls ++(90:.2cm) and ++(-30:.1cm) .. (-.3,-.45);
	\draw[very thick] (a1) .. controls ++(90:.4cm) and ++(270:.4cm) .. (b2) .. controls ++(90:.4cm) and ++(270:.4cm) ..  (c3);
	\node at ($ (a1) + (0,-.2) $) {\scriptsize{$X$}};
	\node at ($ (a2) + (0,-.2) $) {\scriptsize{$c$}};
	\node at ($ (a3) + (0,-.2) $) {\scriptsize{$d$}};
	\node at ($ (c1) + (0,.2) $) {\scriptsize{$c$}};
	\node at ($ (c2) + (0,.2) $) {\scriptsize{$d$}};
	\node at ($ (c3) + (0,.2) $) {\scriptsize{$X$}};
\end{tikzpicture}
\qquad\,\,\,\,
\begin{matrix}\phantom{\Big|}
((\mu^\alpha_{c,d})^{-1}\boxtimes\id_X)\circ e_{X,c\otimes d}\circ(\id_X\boxtimes\mu^\alpha_{c,d})
\\
=(\id_{\alpha(c)}\boxtimes e_{X,d})\circ(e_{X,c}\boxtimes\id_{\alpha(d)})
\end{matrix}
\]
Morphisms $(X, e_X) \to (Y, e_Y)$ are morphisms $f: X\to Y$ in $\Bim(R)$ that satisfy
$$
\begin{tikzpicture}[baseline=.3cm]
	\coordinate (a) at (-.4,-.4);
	\coordinate (b) at (.4,-.4);
	\coordinate (c) at (.4,.4);
	\coordinate (d) at (-.4,.4);
	\coordinate (e) at (.4,1.2);
	\coordinate (f) at (-.4,1.2);
	\draw (b) .. controls ++(90:.2cm) and ++(-30:.1cm) .. (.1,-.05);
	\draw (f) -- (d) .. controls ++(270:.2cm) and ++(150:.1cm) .. (-.1,.05);
	\draw[very thick] (a) .. controls ++(90:.4cm) and ++(270:.4cm) .. (c) -- (e);
	\roundNbox{unshaded}{($ (c) + (0,.3) $)}{.3cm}{0}{0}{$f$};
	\node at ($ (a) + (0,-.2) $) {\scriptsize{$X$}};
	\node at ($ (b) + (0,-.2) $) {\scriptsize{$c$}};
	\node at ($ (f) + (0,.2) $) {\scriptsize{$c$}};
	\node at ($ (e) + (0,.2) $) {\scriptsize{$Y$}};
\end{tikzpicture}
\,\,=\,
\begin{tikzpicture}[baseline=-.5cm, xscale=-1, yscale=-1]
	\coordinate (a) at (-.4,-.4);
	\coordinate (b) at (.4,-.4);
	\coordinate (c) at (.4,.4);
	\coordinate (d) at (-.4,.4);
	\coordinate (e) at (.4,1.2);
	\coordinate (f) at (-.4,1.2);
	\draw (b) .. controls ++(90:.2cm) and ++(-30:.1cm) .. (.1,-.05);
	\draw (f) -- (d) .. controls ++(270:.2cm) and ++(150:.1cm) .. (-.1,.05);
	\draw[very thick] (a) .. controls ++(90:.4cm) and ++(270:.4cm) .. (c) -- (e);
	\roundNbox{unshaded}{($ (c) + (0,.3) $)}{.3cm}{0}{0}{$f$};
	\node at ($ (a) + (0,-.2) $) {\scriptsize{$Y$}};
	\node at ($ (b) + (0,-.2) $) {\scriptsize{$c$}};
	\node at ($ (f) + (0,.2) $) {\scriptsize{$c$}};
	\node at ($ (e) + (0,.2) $) {\scriptsize{$X$}};
\end{tikzpicture}
.
$$
The commutant $\cC'$ is a tensor category with $(X,e_X)\otimes (Y, e_Y)=(X\boxtimes Y, e_{X\boxtimes Y})$, where the half-braiding $e_{X\boxtimes Y, c}$ is given by
$$
e_{X\boxtimes Y, c} :=
\begin{tikzpicture}[baseline=-.1cm]
	\coordinate (a1) at (-.8,-.8);
	\coordinate (a2) at (0,-.8);
	\coordinate (a3) at (.8,-.8);
	\coordinate (b1) at ($ (a1) + (0,.8) $);
	\coordinate (b2) at ($ (a2) + (0,.8) $);
	\coordinate (b3) at ($ (a3) + (0,.8) $);
	\coordinate (c1) at ($ (b1) + (0,.8) $);
	\coordinate (c2) at ($ (b2) + (0,.8) $);
	\coordinate (c3) at ($ (b3) + (0,.8) $);
	\draw (a3) .. controls ++(90:.2cm) and ++(-30:.1cm) .. (.5,-.45);
	\draw (c1) .. controls ++(270:.2cm) and ++(150:.1cm) .. (-.5,.45);
	\draw (b2) .. controls ++(90:.2cm) and ++(-30:.1cm) .. (-.3,.35);
	\draw (b2) .. controls ++(270:.2cm) and ++(150:.1cm) .. (.3,-.35);
	\draw[very thick] (a1) -- (b1) .. controls ++(90:.4cm) and ++(270:.4cm) .. (c2);
	\draw[very thick] (a2) .. controls ++(90:.4cm) and ++(270:.4cm) .. (b3) -- (c3);
	\node at ($ (a1) + (0,-.2) $) {\scriptsize{$X$}};
	\node at ($ (a2) + (0,-.2) $) {\scriptsize{$Y$}};
	\node at ($ (a3) + (0,-.2) $) {\scriptsize{$c$}};
	\node at ($ (c1) + (0,.2) $) {\scriptsize{$c$}};
	\node at ($ (c2) + (0,.2) $) {\scriptsize{$X$}};
	\node at ($ (c3) + (0,.2) $) {\scriptsize{$Y$}};
\end{tikzpicture}
\,\,:\,X\boxtimes Y\boxtimes\alpha(c)\to \alpha(c)\boxtimes X\boxtimes Y.
$$ 

The commutant category is naturally a bi-involutive category.
The dagger structure inherited from $\Bim(R)$, and the conjugate of $(X, e_{X})$ given by
$\overline{X}\in \Bim(R)$ along with the unitary half-braiding
$$
e_{\overline X,c}:
\xymatrix@C=1.2cm{
\overline X \otimes c
\ar[r]^{\id \otimes \varphi_c} &
\overline X \otimes\overline{\overline{c}}
\ar[r]^{\nu_{X,\overline{c}}} &
\overline{\overline{c} \otimes X}
\ar[r]^{\overline{e_{X,\overline{c}}}^{-1}} &
\overline {X \otimes \overline{c}}
\ar[r]^{\nu_{\overline c,X}^{-1}} &
\overline{\overline{c}}\otimes \overline X
\ar[r]^{\varphi_c^{-1}\otimes \id} &
c\otimes \overline X
}
$$
(which is an abbreviation for
$$
\xymatrix@C=1.2cm@R=.3cm{
\overline X \otimes \alpha(c)
\ar[r]^{\id \otimes \alpha(\varphi_c)} &
\overline X \otimes \alpha(\overline{\overline c})
\ar[r]^{\id\otimes \chi_{\overline c}} &
\overline X \otimes\overline{\alpha(\overline c)}
\ar[r]^{\nu_{X,\alpha(\overline{c})}} &
\overline{\alpha(\overline{c}) \otimes X}
\ar[r]^{\overline{e_{X,\overline{c}}}^{-1}}&
\qquad\qquad
\\
\qquad\qquad
\ar[r] &
\overline {X \otimes \alpha(\overline{c})}
\ar[r]^{\nu_{\alpha(\overline c),X}^{-1}} &
\overline{\alpha(\overline{c})}\otimes \overline X
\ar[r]^{\overline{\chi_{c}}\otimes \id} &
\overline{\overline{\alpha(c)}}\otimes \overline X
\ar[r]^{\varphi_{\alpha(c)}^{-1}\otimes \id} &
c\otimes \overline X\hspace{3mm}).\!\!\hspace{-3mm}
}
$$
It is moreover equipped with a positive structure, once again inherited from $\Bim(R)$:
\[
\cP^{\cC'}_{(X, e_{X}),(Y, e_{Y})}\,:=\,\cP^{\Bim(R)}_{X,Y}\,\cap\, \Hom_{\cC'}\!\big((X, e_{X}),(Y, e_{Y})\big).
\]

The commutant category admits an evident positive representation $\alpha':\cC'\to\Bim(R)$ given by forgetting the half-braiding: $(X,e_X)\mapsto X$.
So we may iterate the commutant construction to obtain $\cC'' = (\cC')'$. 
Note that there is a natural inclusion functor $\iota : \cC \to \cC''$ given by $c\mapsto (\alpha(c), e_c)$, where for $\underline{X}= (X, e_X)\in \cC'$ the map $e_{c, \underline{X}}:\alpha(c)\boxtimes X\to X\boxtimes\alpha(c)$ is given by $e_{c, \underline{X}} := e_{X, c}^{-1}$.
Naturality and the hexagon axiom are easily verified, as is the fact that morphisms in $\cC$ give morphisms in $\cC''$.

The following is a slight modification of \cite[Def.~3.2]{MR3663592}:
\begin{defn}\label{def:bicommutant category}
A \emph{bicommutant category} is a bi-involutive tensor category with positive structure $\cC$ such that there exists a hyperfinite von Neumann algebra $R$ and a positive representation $\cC\to \Bim(R)$ for which the inclusion functor $\iota: \cC \hookrightarrow \cC''$ is an equivalence.

An equivalence of bicommutant categories is an equivalence of the underlying bi-involutive tensor categories, which respects the positive structures.
\end{defn}

By \cite[Lem~6.1 and Thm~A]{MR3663592}, for any fully faithful representation $\cC\to\Bim(R)$ of a unitary fusion category $\cC$, the commutant category $\cC'$ is a bicommutant category.

\begin{rem}
In \cite{MR3663592}, the definition of a bicommutant category demanded $R$ to be a hyperfinite factor.
In principle, we could allow $R$ to be any von Neumann algebra.
\end{rem}

\subsection{Description of the commutant}

Let $\cC$ be a unitary fusion category equipped with a fully faithful representation $\alpha:\cC\rightarrow \Bim(R)$, where $R$ is a factor.
We recall from \cite[\S4.1]{MR3663592} the definition of the functor $\Delta: \Bim(R) \to \cC'$:
\begin{equation}\label{The functor Delta}
\Delta(X) := \bigoplus_{c\in \Irr(\cC)} c\boxtimes X \boxtimes \overline{c},
\qquad\quad
e_{\Delta(X), a} 
\,:=\,
\sum_{b,c\in \Irr(\cC)} \!
\sqrt{d_a^{-1}}\,
\begin{tikzpicture}[baseline=-.1cm]
	\pgfmathsetmacro{\width}{.3}
	\pgfmathsetmacro{\height}{.6}
	\coordinate (a) at (-\width,.2);
	\coordinate (b) at (\width,-.2);
	\draw[very thick] (0,-\height) -- (0,\height);
	\draw (-\width,-\height) -- (-\width,\height);
	\draw (\width,-\height) -- (\width,\height);
	\draw ($ 3*(-\width, 0) + (0,\height) $) -- (a);
	\draw ($ 3*(\width, 0) + (0,-\height) $) -- (b);
	\draw[fill=\betacolor] (a) circle (.05cm);
	\draw[fill=\betacolor] (b) circle (.05cm);
	\node at ($ (0,-\height) + (0,-.2) $) {\scriptsize{$X$}};
	\node at ($ (-\width,-\height) + (0,-.2) $) {\scriptsize{$c$}};
	\node at ($ (\width,-\height) + (0,-.2) $) {\scriptsize{$\overline c$}};
	\node at  ($ 3*(\width,0) + (0,-\height) + (0,-.2) $) {\scriptsize{$a$}};
	\node at ($ (0,\height) + (0,.2) $) {\scriptsize{$X$}};
	\node at ($ (-\width,\height) + (0,.2) $) {\scriptsize{$b$}};
	\node at ($ (\width,\height) + (0,.2) $) {\scriptsize{$\overline b$}};
	\node at  ($ 3*(-\width,0) + (0,\height) + (0,.2) $) {\scriptsize{$a$}};
\end{tikzpicture}
\end{equation}
where we have used the graphical convention \eqref{eq:pair of shaded nodes} for pairs of colored nodes.
The functor \eqref{The functor Delta} is is a categorical version of a (non-normalised) conditional expectation.
Let $D=D(\cC)$ be the global dimension of $\cC$.

\begin{lem}
\label{lem:Projector}
If $(X, e_X)\in \cC'$, then the map $u_X : X \to \Delta(X)$ given by
\begin{align*}
u_X 
&:=
\frac{1}{\sqrt{D}}
\sum_{c\in \Irr(\cC)}
\sqrt{d_c}\,
\begin{tikzpicture}[baseline=.3cm, yscale=-1]
	\pgfmathsetmacro{\width}{.4}
	\pgfmathsetmacro{\height}{.6}
	\draw (-\width,-\height) arc (180:0:\width);
	\draw[white, super thick] (0,-\height) -- (0,.2);
	\draw[very thick] (0,-\height) -- (0,.2);
	\node at ($ (0,-\height) + (0,-.2) $) {\scriptsize{$X$}};
	\node at ($ (-\width,-\height) + (0,-.185) $) {\scriptsize{$c$}};
	\node at ($ (\width,-\height) + (0,-.2) $) {\scriptsize{$\overline c$}};
\end{tikzpicture}
\\
&\phantom{:}=
\frac{1}{\sqrt{D}}
\sum_{c\in \Irr(\cC)}
\textstyle \sqrt{d_c}\cdot
(e_{X,c}\otimes\id_{\overline{c}})\circ (\id_X\otimes \coev_c)
\end{align*}
is an isometry, and is a morphism in $\cC'$.
The projector $p_X:=u_Xu^*_X\in\End_{\cC'}(\Delta(X))$ is given by
$$
p_{X} 
=
\frac{1}{D}
\sum_{a,x,y\in \Irr(\cC)}
\sqrt{d_a}\,
\begin{tikzpicture}[baseline=-.1cm]
	\pgfmathsetmacro{\width}{.4}
	\pgfmathsetmacro{\height}{.6}
	\coordinate (a) at (-\width,.3);
	\coordinate (b) at (\width,-.3);
	\draw (-\width,-\height) -- (-\width,\height);
	\draw (\width,-\height) -- (\width,\height);
	\draw (a) -- (b);
	\draw[white, super thick] (0,-\height) -- (0,\height);
	\draw[very thick] (0,-\height) -- (0,\height);
	\draw[fill=\betacolor] (a) circle (.05cm);
	\draw[fill=\betacolor] (b) circle (.05cm);
	\node at ($ (0,-\height) + (0,-.2) $) {\scriptsize{$X$}};
	\node at ($ (-\width,-\height) + (0,-.2) $) {\scriptsize{$x$}};
	\node at ($ (\width,-\height) + (0,-.2) $) {\scriptsize{$\overline x$}};
	\node at ($ (0,\height) + (0,.2) $) {\scriptsize{$X$}};
	\node at ($ (-\width,\height) + (0,.2) $) {\scriptsize{$y$}};
	\node at ($ (\width,\height) + (0,.2) $) {\scriptsize{$\overline y$}};
	\node at ($ (b) + (-.15,.35) $) {\scriptsize{$a$}};
	\node at ($ (a) + (.2,.05) $) {\scriptsize{$a$}};
\end{tikzpicture}.
$$
\end{lem}

\begin{proof}
We first check that $u_X$ is an isometry:
\[
u_X^*u_X 
=
\frac{1}{D}
\sum_{c\in \Irr(\cC)}
d_c\,
\begin{tikzpicture}[baseline=-.1cm]
	\pgfmathsetmacro{\width}{.25}
	\pgfmathsetmacro{\height}{.5}
	\draw (0,0) circle (\width);
	\draw[white, super thick] (0,-\height) -- (0,\height);
	\draw[very thick] (0,-\height) -- (0,\height);
	\node at ($ (0,-\height) + (0,-.2) $) {\scriptsize{$X$}};
	\node at ($ (-\width,0) + (-.2,0) $) {\scriptsize{$c$}};
\end{tikzpicture}
=
\frac{1}{D} \left( \sum_{c\in \Irr(\cC)} d_c^2\right) \id_X = \id_X.
\]
To see that $u_X$ is a morphism in $\cC'$, i.e., that it commutes with the half-braiding, we compute
\[
e_{\Delta(X),a} (u_X\otimes\id)  = 
\frac{1}{\sqrt{D}}
\sum_{x,y\in \Irr(\cC)}
\sqrt{\frac{d_y}{d_a}}
\begin{tikzpicture}[baseline=-.1cm, scale=-1]
	\pgfmathsetmacro{\width}{.4}
	\pgfmathsetmacro{\height}{.6}
	\coordinate (a) at ($ (-\width,0) + 1/2*(0,-\height) $);
	\coordinate (b) at ($ (\width,0) + 1/2*(0,-\height) $);
	\draw (-\width,-\height) -- ($ (-\width,0) + 1/2*(0,-\height) $) arc (180:0:\width) -- (\width,-\height);
	\draw ($ 2*(\width,0) + (0,-\height) $) -- (b);
	\draw ($ 2*(-\width,0) + (0,\height) $) -- (a);
	\draw[white, super thick] (0,-\height) -- (0,\height);
	\draw[very thick] (0,-\height) -- (0,\height);
	\draw[fill=\betacolor] (a) circle (.05cm);
	\draw[fill=\betacolor] (b) circle (.05cm);
	\node at ($ (0,-\height) + (0,-.2) $) {\scriptsize{$X$}};
	\node at ($ (\width,-\height) + (0,-.185) $) {\scriptsize{$x$}};
	\node at ($ (-\width,-\height) + (0,-.2) $) {\scriptsize{$\overline x$}};
	\node at ($ 2*(\width,0) + (0,-\height) + (0,-.2) $) {\scriptsize{$a$}};
	\node at ($ 2*(-\width,0) + (0,\height) + (0,.2) $) {\scriptsize{$a$}};
	\node at ($ (b) + (-.15,.55) $) {\scriptsize{$y$}};
	\node at ($ (a) + (.15,.55) $) {\scriptsize{$\overline y$}};
\end{tikzpicture}
\!=
\frac{1}{\sqrt{D}}
\sum_{x,y\in \Irr(\cC)}
\sqrt{d_x}
\begin{tikzpicture}[baseline=-.1cm, scale=-1]
	\pgfmathsetmacro{\width}{.4}
	\pgfmathsetmacro{\height}{.6}
	\draw ($ 2*(\width,0) + (0,-\height) $) -- ($ (-\width,0) + (0,\height) $);
	\draw (-\width,-\height) arc (180:0:\width);
	\draw[white, super thick] (0,-\height) -- (0,\height);
	\draw[very thick] (0,-\height) -- (0,\height);
	\node at ($ (0,-\height) + (0,-.2) $) {\scriptsize{$X$}};
	\node at ($ (\width,-\height) + (0,-.185) $) {\scriptsize{$x$}};
	\node at ($ (-\width,-\height) + (0,-.2) $) {\scriptsize{$\overline x$}};
	\node at ($ (0,\height) + (0,.2) $) {\scriptsize{$X$}};
	\node at ($ 2*(\width,0) + (0,-\height) + (0,-.185) $) {\scriptsize{$a$}};
	\node at ($ (-\width,0) + (0,\height) + (0,.2) $) {\scriptsize{$a$}};
\end{tikzpicture}
=
(\id\otimes u_X) e_{X,a},
\]
where we have used Lemma~\ref{lem: graph calc with pairs of nodes} (Fusion) for the second equality.
Finally, we check that
\[
u_Xu_X^*
=\,
\frac1D
\sum_{x,y\in \Irr(\cC)}
\sqrt{d_x d_y}\,
\begin{tikzpicture}[baseline=-.1cm]
	\pgfmathsetmacro{\width}{.4}
	\pgfmathsetmacro{\height}{.6}
	\draw (-\width,\height) arc (-180:0:\width);
	\draw (-\width,-\height) arc (180:0:\width);
	\draw[white, super thick] (0,-\height) -- (0,\height);
	\draw[very thick] (0,-\height) -- (0,\height);
	\node at ($ (0,-\height) + (0,-.2) $) {\scriptsize{$X$}};
	\node at ($ (-\width,-\height) + (0,-.2) $) {\scriptsize{$x$}};
	\node at ($ (\width,-\height) + (0,-.2) $) {\scriptsize{$\overline x$}};
	\node at ($ (0,\height) + (0,.2) $) {\scriptsize{$X$}};
	\node at ($ (-\width,\height) + (0,.2) $) {\scriptsize{$y$}};
	\node at ($ (\width,\height) + (0,.2) $) {\scriptsize{$\overline y$}};
\end{tikzpicture}
=\,
\frac1D
\sum_{\substack{x,y,a\\\in \Irr(\cC)}}
\sqrt{d_a}\,
\begin{tikzpicture}[baseline=-.1cm]
	\pgfmathsetmacro{\width}{.3}
	\pgfmathsetmacro{\height}{.8}
	\coordinate (f) at ($ (-\width,0) $);
	\coordinate (e) at ($ 2*(-\width,0) + (0,\width) $);
	\draw ($ 2*(-\width,0) + (0, -\height) $) -- ($ 2*(-\width,0) + (0, \height) $);
	\draw (\width,-\height) -- (f) -- (\width,\height);
	\draw (e) -- (f);
	\draw[white, super thick] (0,-\height) -- (0,\height);
	\draw[very thick] (0,-\height) -- (0,\height);
	\draw[fill=\betacolor] (e) circle (.05cm);
	\draw[fill=\betacolor] (f) circle (.05cm);
	\node at ($ (0,-\height) + (0,-.2) $) {\scriptsize{$X$}};
	\node at ($ 2*(-\width,0) + (0,-\height) + (0,-.2) $) {\scriptsize{$x$}};
	\node at ($ (\width,-\height) + (0,-.2) $) {\scriptsize{$\overline x$}};
	\node at ($ (0,\height) + (0,.2) $) {\scriptsize{$X$}};
	\node at ($ 2*(-\width,0) + (0,\height) + (0,.17) $) {\scriptsize{$y$}};
	\node at ($ (\width,\height) + (0,.2) $) {\scriptsize{$\overline y$}};
	\node at ($ 1/2*(e) + 1/2*(f) + (.05,.2) $) {\scriptsize{$a$}};
\end{tikzpicture}
=\,
\frac1D
\sum_{\substack{x,y,a\\\in \Irr(\cC)}}
\sqrt{d_a}\,
\begin{tikzpicture}[baseline=-.1cm]
	\pgfmathsetmacro{\width}{.4}
	\pgfmathsetmacro{\height}{.6}
	\coordinate (a) at (-\width,.3);
	\coordinate (b) at (\width,-.3);
	\draw (-\width,-\height) -- (-\width,\height);
	\draw (\width,-\height) -- (\width,\height);
	\draw (a) -- (b);
	\draw[white, super thick] (0,-\height) -- (0,\height);
	\draw[very thick] (0,-\height) -- (0,\height);
	\draw[fill=\betacolor] (a) circle (.05cm);
	\draw[fill=\betacolor] (b) circle (.05cm);
	\node at ($ (0,-\height) + (0,-.2) $) {\scriptsize{$X$}};
	\node at ($ (-\width,-\height) + (0,-.2) $) {\scriptsize{$x$}};
	\node at ($ (\width,-\height) + (0,-.2) $) {\scriptsize{$\overline x$}};
	\node at ($ (0,\height) + (0,.2) $) {\scriptsize{$X$}};
	\node at ($ (-\width,\height) + (0,.17) $) {\scriptsize{$y$}};
	\node at ($ (\width,\height) + (0,.2) $) {\scriptsize{$\overline y$}};
	\node at ($ (b) + (-.15,.35) $) {\scriptsize{$a$}};
	\node at ($ (a) + (.2,.05) $) {\scriptsize{$a$}};
\end{tikzpicture}
=
p_X.
\]
Once again, we have used Lemma~\ref{lem: graph calc with pairs of nodes} (Fusion) for the second equality.

The relation $p_X^2=p_X$ follows from the previous computations.
We present a second proof of this relation for the benefit of the reader, and for later reference:
\begin{align} \label{eq: p_X^2=p_X}
p_X^2\,
&= 
\frac{1}{D^2}
\sum_{\substack{x,y,z,a,b\\ \in\Irr(\cC)}}
\sqrt{d_a d_b}\,
\begin{tikzpicture}[baseline=-.1cm]
	\pgfmathsetmacro{\width}{.4}
	\pgfmathsetmacro{\height}{1}
	\coordinate (a) at (-\width,.6);
	\coordinate (b) at (\width,0);
	\coordinate (c) at (-\width,0);
	\coordinate (d) at (\width,-.6);
	\draw (-\width,-\height) -- (-\width,\height);
	\draw (\width,-\height) -- (\width,\height);
	\draw (a) -- (b);
	\draw (c) -- (d);
	\draw[white, super thick] (0,-\height) -- (0,\height);
	\draw[very thick] (0,-\height) -- (0,\height);
	\draw[fill=\alphacolor] (a) circle (.05cm);
	\draw[fill=\alphacolor] (b) circle (.05cm);
	\draw[fill=\betacolor] (c) circle (.05cm);
	\draw[fill=\betacolor] (d) circle (.05cm);
	\node at ($ (0,-\height) + (0,-.2) $) {\scriptsize{$X$}};
	\node at ($ (-\width,-\height) + (0,-.2) $) {\scriptsize{$x$}};
	\node at ($ (\width,-\height) + (0,-.2) $) {\scriptsize{$\overline x$}};
	\node at ($ (0,\height) + (0,.2) $) {\scriptsize{$X$}};
	\node at ($ (-\width,\height) + (0,.17) $) {\scriptsize{$z$}};
	\node at ($ (\width,\height) + (0,.2) $) {\scriptsize{$\overline z$}};
	\node at ($ (b) + (-.15,.35) $) {\scriptsize{$b$}};
	\node at ($ (a) + (.2,.05) $) {\scriptsize{$b$}};
	\node at ($ (d) + (-.15,.35) $) {\scriptsize{$a$}};
	\node at ($ (c) + (.2,.05) $) {\scriptsize{$a$}};
	\node at ($ (c) + (-.15,.35) $) {\scriptsize{$y$}};
	\node at ($ (b) + (.15,-.35) $) {\scriptsize{$\overline y$}};
\end{tikzpicture}
=
\frac{1}{D^2}
\sum_{\substack{x,y,z,a,b,c\\ \in\Irr(\cC)}}
\sqrt{d_c}\,
\begin{tikzpicture}[baseline=-.1cm]
	\pgfmathsetmacro{\width}{.6}
	\pgfmathsetmacro{\height}{1}
	\coordinate (a) at ($ 2*(-\width,0) + (0,.7) $);
	\coordinate (b) at (\width,-.3);
	\coordinate (c) at ($ 2*(-\width,0) + (0,.3) $);
	\coordinate (d) at (\width,-.7);
	\coordinate (e) at ($ (-\width,0) + (-.2, .2) $);
	\coordinate (f) at ($ (-\width,0) + (.2, 0) $);
	\draw ($ 2*(-\width,0) + (0, -\height) $) -- ($ 2*(-\width,0) + (0, \height) $);
	\draw (\width,-\height) -- (\width,\height);
	\draw (a) -- (e) -- (c);
	\draw (b) -- (f) -- (d);
	\draw (e) -- (f);
	\draw[white, super thick] (0,-\height) -- (0,\height);
	\draw[very thick] (0,-\height) -- (0,\height);
	\draw[fill=\alphacolor] (a) circle (.05cm);
	\draw[fill=\alphacolor] (b) circle (.05cm);
	\draw[fill=\betacolor] (c) circle (.05cm);
	\draw[fill=\betacolor] (d) circle (.05cm);
	\draw[fill=\gammacolor] (e) circle (.05cm);
	\draw[fill=\gammacolor] (f) circle (.05cm);
	\node at ($ (0,-\height) + (0,-.2) $) {\scriptsize{$X$}};
	\node at ($ 2*(-\width,0) + (0,-\height) + (0,-.2) $) {\scriptsize{$x$}};
	\node at ($ (\width,-\height) + (0,-.2) $) {\scriptsize{$\overline x$}};
	\node at ($ (0,\height) + (0,.2) $) {\scriptsize{$X$}};
	\node at ($ 2*(-\width,0) + (0,\height) + (0,.17) $) {\scriptsize{$z$}};
	\node at ($ (\width,\height) + (0,.2) $) {\scriptsize{$\overline z$}};
	\node at ($ (b) + (-.15,.25) $) {\scriptsize{$b$}};
	\node at ($ (a) + (.2,.05) $) {\scriptsize{$b$}};
	\node at ($ (d) + (-.2,-.05) $) {\scriptsize{$a$}};
	\node at ($ (c) + (.15,-.25) $) {\scriptsize{$a$}};
	\node at ($ (c) + (-.15,.2) $) {\scriptsize{$y$}};
	\node at ($ (b) + (.15,-.2) $) {\scriptsize{$\overline y$}};
	\node at ($ 1/2*(e) + 1/2*(f) + (.05,.2) $) {\scriptsize{$c$}};
\end{tikzpicture}
\\
\notag
&=
\frac{1}{D^2}
\sum_{\substack{x,y,z,a,b,c\\ \in\Irr(\cC)}}
\sqrt{d_c}\,
\begin{tikzpicture}[baseline=-.1cm]
	\pgfmathsetmacro{\width}{.8}
	\pgfmathsetmacro{\height}{1}
	\coordinate (a) at ($ (-\width,0) + (0,.7) $);
	\coordinate (b) at (\width,-.3);
	\coordinate (c) at ($ (-\width,0) + (0,.3) $);
	\coordinate (d) at (\width,-.7);
	\coordinate (e) at ($ 1/2*(-\width,0) + (0, .2) $);
	\coordinate (f) at ($ 1/2*(\width,0) + (0, -.2) $);
	\draw ($ (-\width,0) + (0, -\height) $) -- ($ (-\width,0) + (0, \height) $);
	\draw (\width,-\height) -- (\width,\height);
	\draw (a) -- (e) -- (c);
	\draw (b) -- (f) -- (d);
	\draw (e) -- (f);
	\draw[white, super thick] (0,-\height) -- (0,\height);
	\draw[very thick] (0,-\height) -- (0,\height);
	\draw[fill=\alphacolor] (a) circle (.05cm);
	\draw[fill=\alphacolor] (b) circle (.05cm);
	\draw[fill=\betacolor] (c) circle (.05cm);
	\draw[fill=\betacolor] (d) circle (.05cm);
	\draw[fill=\gammacolor] (e) circle (.05cm);
	\draw[fill=\gammacolor] (f) circle (.05cm);
	\node at ($ (0,-\height) + (0,-.2) $) {\scriptsize{$X$}};
	\node at ($ (-\width,0) + (0,-\height) + (0,-.2) $) {\scriptsize{$x$}};
	\node at ($ (\width,-\height) + (0,-.2) $) {\scriptsize{$\overline x$}};
	\node at ($ (0,\height) + (0,.2) $) {\scriptsize{$X$}};
	\node at ($ (-\width,0) + (0,\height) + (0,.17) $) {\scriptsize{$z$}};
	\node at ($ (\width,\height) + (0,.2) $) {\scriptsize{$\overline z$}};
	\node at ($ (b) + (-.15,.25) $) {\scriptsize{$b$}};
	\node at ($ (a) + (.2,.05) $) {\scriptsize{$b$}};
	\node at ($ (d) + (-.2,-.05) $) {\scriptsize{$a$}};
	\node at ($ (c) + (.15,-.25) $) {\scriptsize{$a$}};
	\node at ($ (c) + (-.15,.2) $) {\scriptsize{$y$}};
	\node at ($ (b) + (.15,-.2) $) {\scriptsize{$\overline y$}};
	\node at ($ (e) + (.2,.05) $) {\scriptsize{$c$}};
	\node at ($ (f) + (-.2,.25) $) {\scriptsize{$c$}};
\end{tikzpicture}
=
\frac{1}{D^2}
\sum_{\substack{x,y,z,a,b,c\\ \in\Irr(\cC)}}
\sqrt{d_c}\,
\begin{tikzpicture}[baseline=-.1cm]
	\pgfmathsetmacro{\width}{.5}
	\pgfmathsetmacro{\height}{1}
	\coordinate (a) at ($ (-\width,0) + (0,.8) $);
	\coordinate (b) at (\width,.2);
	\coordinate (c) at ($ (-\width,0) + (0,.2) $);
	\coordinate (d) at (\width,.8);
	\coordinate (e) at ($ (-\width,0) + (0, -.2) $);
	\coordinate (f) at ($ (\width,0) + (0, -.6) $);
	\draw ($ (-\width,0) + (0, -\height) $) -- (c);
	\draw (a) -- ($ (-\width,0) + (0, \height) $);
	\draw (\width,-\height) -- (b);
	\draw (d) -- (\width,\height);
	\draw (c) .. controls ++(45:.2cm) and ++(-45:.2cm) .. (a);
	\draw (c) .. controls ++(135:.2cm) and ++(225:.2cm) .. (a);
	\draw (b) .. controls ++(45:.2cm) and ++(-45:.2cm) .. (d);
	\draw (b) .. controls ++(135:.2cm) and ++(225:.2cm) .. (d);
	\draw (e) -- (f);
	\draw[white, super thick] (0,-\height) -- (0,\height);
	\draw[very thick] (0,-\height) -- (0,\height);
	\draw[fill=\alphacolor] (a) circle (.05cm);
	\draw[fill=\alphacolor] (d) circle (.05cm);
	\draw[fill=\betacolor] (c) circle (.05cm);
	\draw[fill=\betacolor] (b) circle (.05cm);
	\draw[fill=\gammacolor] (e) circle (.05cm);
	\draw[fill=\gammacolor] (f) circle (.05cm);
	\node at ($ (0,-\height) + (0,-.2) $) {\scriptsize{$X$}};
	\node at ($ (-\width,0) + (0,-\height) + (0,-.2) $) {\scriptsize{$x$}};
	\node at ($ (\width,-\height) + (0,-.2) $) {\scriptsize{$\overline x$}};
	\node at ($ (0,\height) + (0,.2) $) {\scriptsize{$X$}};
	\node at ($ (-\width,0) + (0,\height) + (0,.17) $) {\scriptsize{$z$}};
	\node at ($ (\width,\height) + (0,.2) $) {\scriptsize{$\overline z$}};
	\node at ($ (a) + (.25,-.3) $) {\scriptsize{$b$}};
	\node at ($ (d) + (-.25,-.3) $) {\scriptsize{$\overline b$}};
	\node at ($ (a) + (-.25,-.3) $) {\scriptsize{$y$}};
	\node at ($ (d) + (.25,-.3) $) {\scriptsize{$\overline y$}};
	\node at ($ (c) + (-.15,-.2) $) {\scriptsize{$a$}};
	\node at ($ (b) + (.15,-.4) $) {\scriptsize{$\overline a$}};
	\node at ($ (e) + (.2,-.25) $) {\scriptsize{$c$}};
	\node at ($ (f) + (-.25,-.05) $) {\scriptsize{$c$}};
\end{tikzpicture}
=\,p_X.
\end{align}
Here, we have used Lemma~\ref{lem: graph calc with pairs of nodes} (Fusion) for the second equality,
Lemma~\ref{lem: graph calc with pairs of nodes} (I=H) for the four equality,
and Lemma~\ref{lem:Sumdxdy} for the last equality.
\end{proof}

Recall that a functor $F: \cS\to \cT$ is called \emph{dominant} if every object of $\cT$ is isomorphic to a subobject of an object of the form $F(s)$, for some $s\in\cS$.
\begin{cor}
The functor $\Delta : \Bim(R) \to \cC'$ is dominant.
\end{cor}

\begin{proof}
If $\underline X=(X, e_X)\in \cC'$, then $\underline X$ is a direct summand of $\Delta(X)$.
Indeed, more is true. If $\underline X \in \cC'$, then $\Delta(X)\cong \underline X\otimes \Delta(1)$.
\end{proof}

\subsection{Commutants of multifusion categories}
\label{sec:CommutantsOfMultifusion}

Let $\cC$ be a $k\times k$ unitary multifusion category.
As in Section~\ref{sec: Unitary multifusion categories}, we write $\cC_{ij}$ for $1_i\otimes \cC \otimes 1_j$, and $\cC_{i}$ for $\cC_{ii}$.
Let $R_1,\ldots,R_k$ be factors, let $\cC\to \Bim(R_1\oplus \cdots \oplus R_k)$ be a fully faithful representation, and
let $\cC'$ be the corresponding commutant category.

Letting $z_i\in R_1\oplus \cdots \oplus R_k$ denote the $i$-th central projection, any $(R_1\oplus \cdots \oplus R_k)$-bimodule $X$ can be decomposed as
\[
X=\bigoplus_{i,j\in\{1,\ldots,k\}} X_{ij}
\]
with $X_{ij}=z_iX z_j$.
In that way, we may think of a bimodule $X\in \Bim(R_1\oplus \cdots \oplus R_k)$ as a matrix of Hilbert spaces $X=(X_{ij})$, where each $X_{ij}$ is an $R_i$-$R_j$ bimodule.

\begin{lem}
\label{lem:Diagonal}
Let $\cC$ be as above, and let $(X, e_X)\in \cC'$ be in of commutant category.
Then $X_{ij} = 0$ for all $i\neq j$.
\end{lem}

\begin{proof}
Let $A:=R_1\oplus \cdots \oplus R_k$, and write $1=\bigoplus_{i=1}^k 1_i\in\Bim(A)$, where $1_i=L^2R_i$.
Since $1_i\in\cC$ commutes with $X$, we have
$X_{ij}=1_i\boxtimes_AX_{ij}\boxtimes_A1_j\cong X_{ij}\boxtimes_A 1_i\boxtimes_A1_j = 0$.
\end{proof}

In the diagrams that follow, the shading of regions will correspond to the various $R_i$:
$$
\tikz[baseline=.1cm]{\draw[fill=\RTwoColor, rounded corners=5, very thin, baseline=1cm] (0,0) rectangle (.5,.5);}=R_i\,\,
\qquad
\tikz[baseline=.1cm]{\draw[fill=\RThreeColor, rounded corners=5, very thin, baseline=1cm] (0,0) rectangle (.5,.5);}=R_j\,\,
\qquad
\tikz[baseline=.1cm]{\draw[fill=\RFourColor, rounded corners=5, very thin, baseline=1cm] (0,0) rectangle (.5,.5);}=R_\ell\,\,
\qquad
\text{etc.}
$$
By the previous lemma, for any $(X, e_X)\in\cC'$, 
we can decompose $X$ as
\begin{equation}\label{eq: X = (+)X_i}
X= \bigoplus_{i\in\{1,\ldots,k\}} X_i,
\end{equation}
with $X_i = z_iX = Xz_i$.
Similarly, the half-braiding $e_X$ decomposes as a family of isomorphisms
\begin{equation}\label{eq: e_X_j,c}
e_{X,c}=\begin{tikzpicture}[baseline=-.1cm]
	\coordinate (a) at (-.4,-.4);
	\coordinate (b) at (.4,-.4);
	\coordinate (c) at (.4,.4);
	\coordinate (d) at (-.4,.4);
	\fill[\RThreeColor] ($ (b) + (.2,0) $) -- (b) .. controls ++(90:.4cm) and ++(270:.4cm) .. (d) -- ($ (c) + (.2,0) $);
	\fill[\RTwoColor] ($ (a) + (-.2,0) $) -- (b) .. controls ++(90:.4cm) and ++(270:.4cm) .. (d) -- ($ (d) + (-.2,0) $);
	\draw (b) .. controls ++(90:.2cm) and ++(-30:.1cm) .. (.1,-.05);
	\draw (d) .. controls ++(270:.2cm) and ++(150:.1cm) .. (-.1,.05);
	\draw[very thick] (a) .. controls ++(90:.4cm) and ++(270:.4cm) .. (c);
	\node at ($ (a) + (0,-.2) $) {\scriptsize{$X_i$}};
	\node at ($ (b) + (0,-.2) $) {\scriptsize{$c$}};
	\node at ($ (d) + (0,.2) $) {\scriptsize{$c$}};
	\node at ($ (c) + (0,.2) $) {\scriptsize{$X_j$}};
\end{tikzpicture}
:\,\, X_i \boxtimes_{R_i} c \,\to\, c\boxtimes_{R_j} X_j
\end{equation}
for every $c\in \cC_{ij}$, and $1\le i,j\le k$.
Here, the shadings $\tikz[baseline=0.07cm]{\draw[fill=\RTwoColor, rounded corners=5, very thin, baseline=1cm] (0,0) rectangle (.5,.5);}$
and $\tikz[baseline=0.07cm]{\draw[fill=\RThreeColor, rounded corners=5, very thin, baseline=1cm] (0,0) rectangle (.5,.5);}$
on the two sides of the strand $c$ represent $R_i$ and $R_j$, respectively.

Let $\cC_i'$ denote the commutant of $\cC_i$ inside $\Bim(R_i)$.
For every $i\in \{1,\dots, k\}$, there is an obvious forgetful/projection functor
\begin{equation}\label{eq: obvious forgetful functor}
\begin{split}
\Pi_i\,\,:\,\,\cC'\,\,&\to\,\,\,\cC_i'\\
(X,e_X) &\mapsto (X_i, e_{X_i}),
\end{split}
\end{equation}
where $X_i$ denotes the $i$-th summand in the decomposition \eqref{eq: X = (+)X_i} of $X$, and
\[
e_{X_i}=\big\{e_{X_i,c}:X_i \boxtimes_{R_i} c \to c\boxtimes_{R_i} X_i\,\big\}_{c\in\cC_i}
\]
is as in \eqref{eq: e_X_j,c}.
The functor $\Pi_i$ 
is bi-involutive (in particular it is a tensor functor), and respects the positive structures.

\begin{thm}[Thm.~\ref{thm:CommutantsEquivalent}]
\label{thm: C' --> C_i' is an equivalence of categories}
For every $i\in\{1,\dots,k\}$,
the functor 
$\Pi_i:\cC'\to\cC_i'$
is an equivalence of categories.
\end{thm}

\begin{proof}
We assume without loss of generality that $i=1$, and reserve the shading\;\!
$\tikz[baseline=0.07cm]{\draw[fill=\RTwoColor, rounded corners=5, very thin, baseline=1cm] (0,0) rectangle (.5,.5);}$\;\! for $R_1$.
Consider the following functor (which is a modification of the functor $\Delta$ from the previous section), given by
\begin{gather*}
\Delta_1: \Bim(R_1) \to \cC'\\
\Delta_1(X) 
:=\!
\bigoplus_{
j\in\{1,\dots, k\}
}\,
\bigoplus_{
c\in \Irr(\cC_{j1})
} 
c\boxtimes X \boxtimes \overline{c}
\end{gather*}
with half-braiding
\begin{equation}\label{eq: with half-braiding}
\,\,\,\,
e_{\Delta_1(X), a} 
\,:=\,
\sum_{\substack{
b\in \Irr(\cC_{i,1})
\\
c\in \Irr(\cC_{j1})
}} \!
\sqrt{d_a^{-1}}\;
\begin{tikzpicture}[baseline=-.1cm]
	\pgfmathsetmacro{\width}{.3}
	\pgfmathsetmacro{\height}{.6}
	\coordinate (a) at (-\width,.2);
	\coordinate (b) at (\width,-.2);
	\fill[fill=\RTwoColor] (-\width,-\height) rectangle (\width,\height);
	\fill[fill=\RThreeColor] ($ 7/2*(-\width, 0) + (0,\height) $) -- ($ 3*(-\width, 0) + (0,\height) $) -- (a) -- (-\width,-\height) -- ($ 7/2*(-\width,0) + (0,-\height) $); 
	\fill[fill=\RThreeColor] (\width, -\height) -- ($ 3*(\width, 0) + (0,-\height) $) -- (b);
	\fill[fill=\RFourColor] ($ 7/2*(\width, 0) + (0,-\height) $) -- ($ 3*(\width, 0) + (0,-\height) $) -- (b) -- (\width,\height) -- ($ 7/2*(\width,0) + (0,\height) $); 
	\fill[fill=\RFourColor] (-\width, \height) -- ($ 3*(-\width, 0) + (0,\height) $) -- (a);
	\draw[very thick] (0,-\height) -- (0,\height);
	\draw (-\width,-\height) -- (-\width,\height);
	\draw (\width,-\height) -- (\width,\height);
	\draw ($ 3*(-\width, 0) + (0,\height) $) -- (a);
	\draw ($ 3*(\width, 0) + (0,-\height) $) -- (b);
	\draw[fill=\betacolor] (a) circle (.05cm);
	\draw[fill=\betacolor] (b) circle (.05cm);
	\node at ($ (0,-\height) + (0,-.2) $) {\scriptsize{$X$}};
	\node at ($ (-\width,-\height) + (0,-.2) $) {\scriptsize{$b$}};
	\node at ($ (\width,-\height) + (0,-.17) $) {\scriptsize{$\overline b$}};
	\node at  ($ 3*(\width,0) + (0,-\height) + (0,-.22) $) {\scriptsize{$a$}};
	\node at ($ (0,\height) + (0,.2) $) {\scriptsize{$X$}};
	\node at ($ (-\width,\height) + (0,.2) $) {\scriptsize{$c$}};
	\node at ($ (\width,\height) + (0,.22) $) {\scriptsize{$\overline c$}};
	\node at  ($ 3*(-\width,0) + (0,\height) + (0,.2) $) {\scriptsize{$a$}};
\end{tikzpicture}
\,\,\,\,\,\,\text{for}\,\,\,\,
a\in \cC_{ij},
\end{equation}
where\;\!
$\tikz[baseline=0.07cm]{\draw[fill=\RThreeColor, rounded corners=5, very thin, baseline=1cm] (0,0) rectangle (.5,.5);}\,=R_i$
and\;\!
$\tikz[baseline=0.07cm]{\draw[fill=\RFourColor, rounded corners=5, very thin, baseline=1cm] (0,0) rectangle (.5,.5);}\,=R_j$.

Recall 
that $D:=\dim(\cC_i) = \dim(\cC_j)$ for all $1\le i, j\le k$.
Let $(X, e_X)$ be an object of $\cC_1'$.
We claim that, similarly to Lemma~\ref{lem:Projector}, the morphism
\begin{equation}
\label{eq:ProjectorC1}
p_X 
:=
\frac{1}{D}
\sum_{\substack{
j\in\{1,\dots, k\}
\\
\tikz[baseline=.07cm]{\draw[fill=\RThreeColor, rounded corners=3, very thin, baseline=1cm] (0,0) rectangle (.3,.3);}\,=R_j
}}
\sum_{\substack{
a\in \Irr(\cC_{1})
\\
x,y\in \Irr(\cC_{j1})
}}
\!\sqrt{d_a}\,\,
\begin{tikzpicture}[baseline=-.1cm]
	\pgfmathsetmacro{\width}{.4}
	\pgfmathsetmacro{\height}{.6}
	\coordinate (a) at (-\width,.3);
	\coordinate (b) at (\width,-.3);
	\fill[fill=\RTwoColor] (-\width,-\height) rectangle (\width,\height);
	\fill[fill=\RThreeColor] ($ 2*(-\width, 0) + (0,-\height) $) rectangle ($ (-\width, 0) + (0,\height) $);
	\fill[fill=\RThreeColor] ($ 2*(\width, 0) + (0,\height) $) rectangle ($ (\width, 0) + (0,-\height) $);
	\draw (-\width,-\height) -- (-\width,\height);
	\draw (\width,-\height) -- (\width,\height);
	\draw (a) -- ($ .58*(a) + .42*(b) $);
	\draw (b) -- ($ .42*(a) + .58*(b) $);
	\draw[very thick] (0,-\height) -- (0,\height);
	\draw[fill=\betacolor] (a) circle (.05cm);
	\draw[fill=\betacolor] (b) circle (.05cm);
	\node at ($ (0,-\height) + (0,-.2) $) {\scriptsize{$X$}};
	\node at ($ (-\width,-\height) + (0,-.2) $) {\scriptsize{$x$}};
	\node at ($ (\width,-\height) + (0,-.2) $) {\scriptsize{$\overline x$}};
	\node at ($ (0,\height) + (0,.2) $) {\scriptsize{$X$}};
	\node at ($ (-\width,\height) + (0,.2) $) {\scriptsize{$y$}};
	\node at ($ (\width,\height) + (0,.2) $) {\scriptsize{$\overline y$}};
	\node at ($ (b) + (-.15,.35) $) {\scriptsize{$a$}};
	\node at ($ (a) + (.2,.05) $) {\scriptsize{$a$}};
\end{tikzpicture}
\,=\,
\frac{1}{D}
\sum_{
\tikz[baseline=.07cm]{\draw[fill=\RThreeColor, rounded corners=3, very thin, baseline=1cm] (0,0) rectangle (.3,.3);}\,=R_j
}
\sum_{\substack{
a\in \Irr(\cC_{1})
\\
x,y\in \Irr(\cC_{j1})
}}
\!\sqrt{d_a}\,\,
\begin{tikzpicture}[baseline=-.1cm]
	\pgfmathsetmacro{\width}{.8}
	\pgfmathsetmacro{\height}{1}
	\coordinate (a) at (-\width,.6);
	\coordinate (b) at (\width,-.6);
	\fill[fill=\RTwoColor] (-\width,-\height) rectangle (\width,\height);
	\fill[fill=\RThreeColor] ($ 1.58*(-\width, 0) + (0,-\height) $) rectangle ($ (-\width, 0) + (0,\height) $);
	\fill[fill=\RThreeColor] ($ 1.58*(\width, 0) + (0,\height) $) rectangle ($ (\width, 0) + (0,-\height) $);
	\draw (-\width,-\height) -- (-\width,\height);
	\draw (\width,-\height) -- (\width,\height);
	\draw (a) -- (b);
	\draw[\RTwoColor, super thick] (0,-\height) -- (0,\height);
	\draw[very thick] (0,-\height) -- (0,\height);
	\draw[fill=\betacolor] (a) circle (.05cm);
	\draw[fill=\betacolor] (b) circle (.05cm);
	\roundNbox{unshaded}{(0,0)}{.4}{0}{0}{$e_{X,a}$}
	\node at ($ (0,-\height) + (0,-.2) $) {\scriptsize{$X$}};
	\node at ($ (-\width,-\height) + (0,-.2) $) {\scriptsize{$x$}};
	\node at ($ (\width,-\height) + (0,-.2) $) {\scriptsize{$\overline x$}};
	\node at ($ (0,\height) + (0,.2) $) {\scriptsize{$X$}};
	\node at ($ (-\width,\height) + (0,.2) $) {\scriptsize{$y$}};
	\node at ($ (\width,\height) + (0,.2) $) {\scriptsize{$\overline y$}};
	\node at ($ (b) + (-.15,.35) $) {\scriptsize{$a$}};
	\node at ($ (a) + (.2,.05) $) {\scriptsize{$a$}};
\end{tikzpicture}
\end{equation}
is a projector, and is an element of $\End_{\cC'}(\Delta_1(X))$.
The relation $p_X^*=p_X$ follows from the unitarity of the half-braiding, and is left to the reader as as exercise.
The relation $p_X^2=p_X$ is proven along the same lines as~\eqref{eq: p_X^2=p_X}, using Lemmata~\ref{lem: graph calc with pairs of nodes} and~\ref{lem:Sumdxdy}:
\begin{gather*} 
p_X^2\,
= 
\frac{1}{D^2}\,
\sum_{
\tikz[baseline=.07cm]{\draw[fill=\RThreeColor, rounded corners=3, very thin, baseline=1cm] (0,0) rectangle (.3,.3);}\,=R_j
}
\sum_{\substack{
a,b\in \Irr(\cC_{1})
\\
x,y,z\in \Irr(\cC_{j1})
}}
\sqrt{d_a d_b}\,\,
\begin{tikzpicture}[baseline=-.1cm]
	\pgfmathsetmacro{\width}{.4}
	\pgfmathsetmacro{\height}{1}
	\coordinate (a) at (-\width,.6);
	\coordinate (b) at (\width,0);
	\coordinate (c) at (-\width,0);
	\coordinate (d) at (\width,-.6);
	\fill[fill=\RTwoColor] (-\width,-\height) rectangle (\width,\height);
	\fill[fill=\RThreeColor] ($ 2.1*(-\width, 0) + (0,-\height) $) rectangle ($ (-\width, 0) + (0,\height) $);
	\fill[fill=\RThreeColor] ($ 2.1*(\width, 0) + (0,\height) $) rectangle ($ (\width, 0) + (0,-\height) $);
	\draw (-\width,-\height) -- (-\width,\height);
	\draw (\width,-\height) -- (\width,\height);
	\draw (a) -- (b);
	\draw (c) -- (d);
	\draw[\RTwoColor, super thick] (0,-\height) -- (0,\height);
	\draw[very thick] (0,-\height) -- (0,\height);
	\draw[fill=\alphacolor] (a) circle (.05cm);
	\draw[fill=\alphacolor] (b) circle (.05cm);
	\draw[fill=\betacolor] (c) circle (.05cm);
	\draw[fill=\betacolor] (d) circle (.05cm);
	\node at ($ (0,-\height) + (0,-.2) $) {\scriptsize{$X$}};
	\node at ($ (-\width,-\height) + (0,-.2) $) {\scriptsize{$x$}};
	\node at ($ (\width,-\height) + (0,-.2) $) {\scriptsize{$\overline x$}};
	\node at ($ (0,\height) + (0,.2) $) {\scriptsize{$X$}};
	\node at ($ (-\width,\height) + (0,.17) $) {\scriptsize{$z$}};
	\node at ($ (\width,\height) + (0,.2) $) {\scriptsize{$\overline z$}};
	\node at ($ (b) + (-.15,.35) $) {\scriptsize{$b$}};
	\node at ($ (a) + (.2,.05) $) {\scriptsize{$b$}};
	\node at ($ (d) + (-.15,.35) $) {\scriptsize{$a$}};
	\node at ($ (c) + (.2,.05) $) {\scriptsize{$a$}};
	\node at ($ (c) + (-.15,.35) $) {\scriptsize{$y$}};
	\node at ($ (b) + (.15,-.35) $) {\scriptsize{$\overline y$}};
\end{tikzpicture}
=
\frac{1}{D^2}\,
\sum_{
\tikz[baseline=.07cm]{\draw[fill=\RThreeColor, rounded corners=3, very thin, baseline=1cm] (0,0) rectangle (.3,.3);}\,=R_j
}
\sum_{\substack{
a,b,c\in \Irr(\cC_{1})
\\
x,y,z\in \Irr(\cC_{j1})
}}
\sqrt{d_c}\,\,
\begin{tikzpicture}[baseline=-.1cm]
	\pgfmathsetmacro{\width}{.8}
	\pgfmathsetmacro{\height}{1}
	\coordinate (a) at ($ (-\width,0) + (0,.7) $);
	\coordinate (b) at (\width,-.3);
	\coordinate (c) at ($ (-\width,0) + (0,.3) $);
	\coordinate (d) at (\width,-.7);
	\coordinate (e) at ($ 1/2*(-\width,0) + (0, .2) $);
	\coordinate (f) at ($ 1/2*(\width,0) + (0, -.2) $);
	\fill[fill=\RTwoColor] (-\width,-\height) rectangle (\width,\height);
	\fill[fill=\RThreeColor] ($ 1.58*(-\width, 0) + (0,-\height) $) rectangle ($ (-\width, 0) + (0,\height) $);
	\fill[fill=\RThreeColor] ($ 1.58*(\width, 0) + (0,\height) $) rectangle ($ (\width, 0) + (0,-\height) $);
	\draw ($ (-\width,0) + (0, -\height) $) -- ($ (-\width,0) + (0, \height) $);
	\draw (\width,-\height) -- (\width,\height);
	\draw (a) -- (e) -- (c);
	\draw (b) -- (f) -- (d);
	\draw (e) -- (f);
	\draw[\RTwoColor, super thick] (0,-\height) -- (0,\height);
	\draw[very thick] (0,-\height) -- (0,\height);
	\draw[fill=\alphacolor] (a) circle (.05cm);
	\draw[fill=\alphacolor] (b) circle (.05cm);
	\draw[fill=\betacolor] (c) circle (.05cm);
	\draw[fill=\betacolor] (d) circle (.05cm);
	\draw[fill=\gammacolor] (e) circle (.05cm);
	\draw[fill=\gammacolor] (f) circle (.05cm);
	\node at ($ (0,-\height) + (0,-.2) $) {\scriptsize{$X$}};
	\node at ($ (-\width,0) + (0,-\height) + (0,-.2) $) {\scriptsize{$x$}};
	\node at ($ (\width,-\height) + (0,-.2) $) {\scriptsize{$\overline x$}};
	\node at ($ (0,\height) + (0,.2) $) {\scriptsize{$X$}};
	\node at ($ (-\width,0) + (0,\height) + (0,.17) $) {\scriptsize{$z$}};
	\node at ($ (\width,\height) + (0,.2) $) {\scriptsize{$\overline z$}};
	\node at ($ (b) + (-.15,.25) $) {\scriptsize{$b$}};
	\node at ($ (a) + (.2,.05) $) {\scriptsize{$b$}};
	\node at ($ (d) + (-.2,-.05) $) {\scriptsize{$a$}};
	\node at ($ (c) + (.15,-.25) $) {\scriptsize{$a$}};
	\node at ($ (c) + (-.15,.2) $) {\scriptsize{$y$}};
	\node at ($ (b) + (.15,-.2) $) {\scriptsize{$\overline y$}};
	\node at ($ (e) + (.2,.05) $) {\scriptsize{$c$}};
	\node at ($ (f) + (-.2,.25) $) {\scriptsize{$c$}};
\end{tikzpicture}
\\
=
\frac{1}{D^2}\,
\sum_{
\tikz[baseline=.07cm]{\draw[fill=\RThreeColor, rounded corners=3, very thin, baseline=1cm] (0,0) rectangle (.3,.3);}\,=R_j
}
\sum_{\substack{
b,c\in \Irr(\cC_{1})
\\
a,x,y,z\in \Irr(\cC_{j1})
}}
\sqrt{d_c}\,\,
\begin{tikzpicture}[baseline=-.1cm]
	\pgfmathsetmacro{\width}{.5}
	\pgfmathsetmacro{\height}{1}
	\coordinate (a) at ($ (-\width,0) + (0,.8) $);
	\coordinate (b) at (\width,.2);
	\coordinate (c) at ($ (-\width,0) + (0,.2) $);
	\coordinate (d) at (\width,.8);
	\coordinate (e) at ($ (-\width,0) + (0, -.2) $);
	\coordinate (f) at ($ (\width,0) + (0, -.6) $);
	\fill[fill=\RTwoColor] (-\width,-\height) rectangle (\width,\height);
	\fill[fill=\RThreeColor] ($ 2*(-\width, 0) + (0,-\height) $) rectangle ($ (-\width, 0) + (0,\height) $);
	\fill[fill=\RThreeColor] ($ 2*(\width, 0) + (0,\height) $) rectangle ($ (\width, 0) + (0,-\height) $);
	\fill[fill=\RTwoColor] (c) .. controls ++(45:.2cm) and ++(-45:.2cm) .. (a)
	.. controls ++(225:.2cm) and ++(135:.2cm) .. (c);
	\fill[fill=\RTwoColor] (b) .. controls ++(45:.2cm) and ++(-45:.2cm) .. (d)
	.. controls ++(225:.2cm) and ++(135:.2cm) .. (b);
	\draw ($ (-\width,0) + (0, -\height) $) -- (c);
	\draw (a) -- ($ (-\width,0) + (0, \height) $);
	\draw (\width,-\height) -- (b);
	\draw (d) -- (\width,\height);
	\draw (c) .. controls ++(45:.2cm) and ++(-45:.2cm) .. (a);
	\draw (c) .. controls ++(135:.2cm) and ++(225:.2cm) .. (a);
	\draw (b) .. controls ++(45:.2cm) and ++(-45:.2cm) .. (d);
	\draw (b) .. controls ++(135:.2cm) and ++(225:.2cm) .. (d);
	\draw (e) -- (f);
	\draw[\RTwoColor, super thick] (0,-\height) -- (0,\height);
	\draw[very thick] (0,-\height) -- (0,\height);
	\draw[fill=\alphacolor] (a) circle (.05cm);
	\draw[fill=\alphacolor] (d) circle (.05cm);
	\draw[fill=\betacolor] (c) circle (.05cm);
	\draw[fill=\betacolor] (b) circle (.05cm);
	\draw[fill=\gammacolor] (e) circle (.05cm);
	\draw[fill=\gammacolor] (f) circle (.05cm);
	\node at ($ (0,-\height) + (0,-.2) $) {\scriptsize{$X$}};
	\node at ($ (-\width,0) + (0,-\height) + (0,-.2) $) {\scriptsize{$x$}};
	\node at ($ (\width,-\height) + (0,-.2) $) {\scriptsize{$\overline x$}};
	\node at ($ (0,\height) + (0,.2) $) {\scriptsize{$X$}};
	\node at ($ (-\width,0) + (0,\height) + (0,.17) $) {\scriptsize{$z$}};
	\node at ($ (\width,\height) + (0,.2) $) {\scriptsize{$\overline z$}};
	\node at ($ (a) + (.25,-.3) $) {\scriptsize{$b$}};
	\node at ($ (d) + (-.25,-.3) $) {\scriptsize{$\overline b$}};
	\node at ($ (a) + (-.25,-.3) $) {\scriptsize{$y$}};
	\node at ($ (d) + (.25,-.3) $) {\scriptsize{$\overline y$}};
	\node at ($ (c) + (-.15,-.2) $) {\scriptsize{$a$}};
	\node at ($ (b) + (.15,-.4) $) {\scriptsize{$\overline a$}};
	\node at ($ (e) + (.2,-.25) $) {\scriptsize{$c$}};
	\node at ($ (f) + (-.25,-.05) $) {\scriptsize{$c$}};
\end{tikzpicture}
\,=\,p_X.
\end{gather*}
Finally, to see that $p_X$ is a morphism of $\cC'$, i.e. that it commutes with the half-braiding \eqref{eq: with half-braiding}, we compute
\begin{align*}
e_{\Delta_1(X),a}\, (p_X\boxtimes \id_a)
\,&=\,
\frac{1}{D}
\sum_{\substack{
b\in \Irr(\cC_{1})
\\
x,y\in \Irr(\cC_{i1})
\\
z\in \Irr(\cC_{j1})
}}
\sqrt{
\frac{d_b}{d_a}
}\,\,\,\,
\begin{tikzpicture}[baseline=-.1cm]
	\pgfmathsetmacro{\width}{.4}
	\pgfmathsetmacro{\height}{1}
	\coordinate (c) at (-\width,.6);
	\coordinate (d) at (\width,-.1);
	\coordinate (a) at (-\width,.1);
	\coordinate (b) at (\width,-.7);
	\fill[fill=\RTwoColor] (-\width,-\height) rectangle (\width,\height);
	\fill[fill=\RThreeColor] ($ 7/2*(-\width, 0) + (0,\height) $) -- ($ 2*(-\width, 0) + (0,\height) $) -- (c) -- (-\width,-\height) -- ($ 7/2*(-\width,0) + (0,-\height) $); 
	\fill[fill=\RThreeColor] (\width, -\height) -- ($ 3*(\width, 0) + (0,-\height) $) -- (d);
	\fill[fill=\RFourColor] ($ 7/2*(\width, 0) + (0,-\height) $) -- ($ 3*(\width, 0) + (0,-\height) $) -- (d) -- (\width,\height) -- ($ 7/2*(\width,0) + (0,\height) $); 
	\fill[fill=\RFourColor] (-\width, \height) -- ($ 2*(-\width, 0) + (0,\height) $) -- (c);
	\draw (-\width,-\height) -- (-\width,\height);
	\draw (\width,-\height) -- (\width,\height);
	\draw (a) -- (b);
	\draw[\RTwoColor, super thick] (0,-\height) -- (0,\height);
	\draw[very thick] (0,-\height) -- (0,\height);
	\draw (d) -- ($ 3*(\width, 0) + (0, -\height) $);
	\draw (c) -- ($ 2*(-\width, 0) + (0,\height) $);
	\draw[fill=\betacolor] (a) circle (.05cm);
	\draw[fill=\betacolor] (b) circle (.05cm);
	\draw[fill=\alphacolor] (c) circle (.05cm);
	\draw[fill=\alphacolor] (d) circle (.05cm);
	\node at ($ (0,-\height) + (0,-.2) $) {\scriptsize{$X$}};
	\node at ($ (-\width,-\height) + (0,-.2) $) {\scriptsize{$x$}};
	\node at ($ (\width,-\height) + (0,-.2) $) {\scriptsize{$\overline x$}};
	\node at ($ (0,\height) + (0,.2) $) {\scriptsize{$X$}};
	\node at ($ (-\width,\height) + (0,.17) $) {\scriptsize{$z$}};
	\node at ($ (\width,\height) + (0,.2) $) {\scriptsize{$\overline z$}};
	\node at ($ (a) + (-.15,.2) $) {\scriptsize{$y$}};
	\node at ($ (b) + (.15,.2) $) {\scriptsize{$\overline y$}};
	\node at ($ (b) + (-.15,.35) $) {\scriptsize{$b$}};
	\node at ($ (a) + (.2,.05) $) {\scriptsize{$b$}};
	\node at ($ 2*(-\width,0) + (0,\height) + (0,.2) $) {\scriptsize{$a$}};
	\node at ($ 3*(\width,0) + (0,-\height) + (0,-.2) $) {\scriptsize{$a$}};
\end{tikzpicture}
\\
&=\,
\frac{1}{D}
\sum_{\substack{
b\in \Irr(\cC_{1})
\\
x\in \Irr(\cC_{i1})
\\
y,z\in \Irr(\cC_{j1})
}}
\sqrt{
\frac{d_b}{d_a}
}\,\,\,\,
\begin{tikzpicture}[baseline=-.1cm]
	\pgfmathsetmacro{\width}{.4}
	\pgfmathsetmacro{\height}{1}
	\coordinate (a) at (-\width,.7);
	\coordinate (b) at (\width,-.1);
	\coordinate (c) at (-\width,.1);
	\coordinate (d) at (\width,-.6);
	\fill[fill=\RTwoColor] (-\width,-\height) rectangle (\width,\height);
	\fill[fill=\RThreeColor] ($ 7/2*(-\width, 0) + (0,\height) $) -- ($ 3*(-\width, 0) + (0,\height) $) -- (c) -- (-\width,-\height) -- ($ 7/2*(-\width,0) + (0,-\height) $); 
	\fill[fill=\RThreeColor] (\width, -\height) -- ($ 2*(\width, 0) + (0,-\height) $) -- (d);
	\fill[fill=\RFourColor] ($ 7/2*(\width, 0) + (0,-\height) $) -- ($ 2*(\width, 0) + (0,-\height) $) -- (d) -- (\width,\height) -- ($ 7/2*(\width,0) + (0,\height) $); 
	\fill[fill=\RFourColor] (-\width, \height) -- ($ 3*(-\width, 0) + (0,\height) $) -- (c);
	\draw (-\width,-\height) -- (-\width,\height);
	\draw (\width,-\height) -- (\width,\height);
	\draw (a) -- (b);
	\draw[\RTwoColor, super thick] (0,-\height) -- (0,\height);
	\draw[very thick] (0,-\height) -- (0,\height);
	\draw (c) -- ($ 3*(-\width, 0) + (0, \height) $);
	\draw (d) -- ($ 2*(\width, 0) + (0,-\height) $);
	\draw[fill=\betacolor] (a) circle (.05cm);
	\draw[fill=\betacolor] (b) circle (.05cm);
	\draw[fill=\alphacolor] (c) circle (.05cm);
	\draw[fill=\alphacolor] (d) circle (.05cm);
	\node at ($ (0,-\height) + (0,-.2) $) {\scriptsize{$X$}};
	\node at ($ (-\width,-\height) + (0,-.2) $) {\scriptsize{$x$}};
	\node at ($ (\width,-\height) + (0,-.2) $) {\scriptsize{$\overline x$}};
	\node at ($ (0,\height) + (0,.2) $) {\scriptsize{$X$}};
	\node at ($ (-\width,\height) + (0,.17) $) {\scriptsize{$z$}};
	\node at ($ (\width,\height) + (0,.2) $) {\scriptsize{$\overline z$}};
	\node at ($ (a) + (-.15,-.2) $) {\scriptsize{$y$}};
	\node at ($ (b) + (.15,-.2) $) {\scriptsize{$\overline y$}};
	\node at ($ (b) + (-.15,.35) $) {\scriptsize{$b$}};
	\node at ($ (a) + (.2,.05) $) {\scriptsize{$b$}};
	\node at ($ 2*(\width,0) + (0,-\height) + (0,-.2) $) {\scriptsize{$a$}};
	\node at ($ 3*(-\width,0) + (0,\height) + (0,.2) $) {\scriptsize{$a$}};
\end{tikzpicture}
\,=\,
(\id_a\boxtimes p_X)\,e_{\Delta_1(X),a}
\end{align*}
for\;\!
$\tikz[baseline=0.07cm]{\draw[fill=\RThreeColor, rounded corners=5, very thin, baseline=1cm] (0,0) rectangle (.5,.5);}\,=R_i$,\;\!
$\tikz[baseline=0.07cm]{\draw[fill=\RFourColor, rounded corners=5, very thin, baseline=1cm] (0,0) rectangle (.5,.5);}\,=R_j$
and $a\in\cC_{ij}$,
where the middle equality holds by Lemma~\ref{lem: graph calc with pairs of nodes}~(I=H).

Since $\cC'$ is idempotent complete,
the above computations allow us to define, for every $\underline X=(X, e_X)\in \cC_1'$,
a new object $\Phi_1(\underline X):=p_X(\Delta_1(X))\in \cC'$, as the image of the projector~\eqref{eq:ProjectorC1}.
We claim that the resulting functor
\[
\begin{split}
\Phi_1 \,\,:\,\,\,\cC_1'\,\,\,\, &\!\longrightarrow\,\,\,\, \cC'\\
(X,e_X) &\mapsto p_X(\Delta_1(X))
\end{split}
\]
is an inverse of the functor $\Pi_1:\cC'\to \cC_1'$ defined in~\eqref{eq: obvious forgetful functor}.

The equation $\Pi_1\circ\Phi_1\simeq \id_{\cC_1}$ is a direct consequence of Lemma~\ref{lem:Projector}.
More precisely, for any object $\underline X=(X,e_X)$ in $\cC_1$ the isometry $u_X$ induces a unitary between $\underline X$ and $\Pi_1\circ\Phi_1(\underline X)$.
These morphisms assemble to a unitary natural transformation $\id_{\cC_1}\stackrel{\scriptscriptstyle\simeq\,}\to\Pi_1\circ\Phi_1 $.

It remains to show that $\Phi_1\circ \Pi_1\simeq \id_\cC$.
For $(X,e_X)\in \cC'$, let $X_1$ be as in \eqref{eq: obvious forgetful functor}.
Then, as in Lemma~\ref{lem:Projector}, the map
\[
u_X :=
\frac{1}{\sqrt{D}}
\sum_{
\tikz[baseline=.07cm]{\draw[fill=\RThreeColor, rounded corners=3, very thin, baseline=1cm] (0,0) rectangle (.3,.3);}\,=R_j
}
\sum_{\substack{
x\in \Irr(\cC_{j1})
}}
\sqrt{d_x}\,\,
\begin{tikzpicture}[baseline=.08cm, yscale=-1]
	\pgfmathsetmacro{\width}{.4}
	\pgfmathsetmacro{\height}{.6}
	\fill[\RTwoColor] (-\width,-\height) arc (180:0:\width);
	\fill[\RThreeColor] ($ 3/2*(-\width, 0) + (0,.2) $)  -- ($ 3/2*(-\width, 0) + (0,-\height) $) -- (-\width,-\height) arc (180:0:\width) -- ($ 3/2*(\width, 0) + (0,-\height) $) -- ($ 3/2*(\width, 0) + (0,.2) $);
	\draw (-\width,-\height) arc (180:98:\width);
	\draw (\width,-\height) arc (0:82:\width);
	\draw[very thick] (0,-\height) -- (0,.2);
	\node at ($ (0,-\height) + (0,-.2) $) {\scriptsize{$X_1$}};
	\node at ($ (0,.2) + (0,.2) $) {\scriptsize{$X_j$}};
	\node at ($ (-\width,-\height) + (0,-.19) $) {\scriptsize{$x$}};
	\node at ($ (\width,-\height) + (0,-.2) $) {\scriptsize{$\overline x$}};
\end{tikzpicture}\,: X\to \Delta_1(X_1)
\]
is an isometry, and an morphism in $\cC'$.
It is an isometry because
\[
u_X^* u_X
=
\frac{1}{D}
\sum_{
\tikz[baseline=.07cm]{\draw[fill=\RThreeColor, rounded corners=3, very thin, baseline=1cm] (0,0) rectangle (.3,.3);}\,=R_j
}
\sum_{\substack{
x\in \Irr(\cC_{j1})
}}\!
d_x\,
\begin{tikzpicture}[baseline=-.1cm]
	\pgfmathsetmacro{\width}{.5}
	\pgfmathsetmacro{\height}{.8}
	\fill[\RThreeColor] ($ 2*(-\width, 0) + (0, -\height) $) rectangle ($ 2*(\width, 0) + (0, \height) $);
	\fill[\RTwoColor] (0,0) circle (\width);
	\draw (-82:\width) arc (-82:82:\width);
	\draw (98:\width) arc (98:262:\width);
	\draw[very thick] (0,-\height) -- (0,\height);
	\node at ($ (0,\height) + (0,.2) $) {\scriptsize{$X_j$}};
	\node at ($ (0,-\height) + (0,-.2) $) {\scriptsize{$X_j$}};
	\node at ($ (-\width,0) + (-.12,0) $) {\scriptsize{$x$}};
	\node at ($ (0,0) + (.25,0) $) {\scriptsize{$X_1$}};
\end{tikzpicture}
\,=
\frac{1}{D} \sum_{j=1}^k \Bigg(\sum_{x\in \Irr(\cC_{j1})}\! d_x^2\Bigg) \id_{X_j} 
=
\frac{1}{D} \sum_{j=1}^k D\cdot \id_{X_j}
= 
\id_X.
\]
And it is a morphism in $\cC'$ because
\begin{align*}
e_{\Delta_1(X_1),a}(u_X\otimes\id)
&= 
\frac{1}{\sqrt{D}}
\sum_{\substack{
x\in \Irr(\cC_{j1})
\\
y\in \Irr(\cC_{i1})
}}\sqrt{\frac{d_y}{d_a}}\,\,
\begin{tikzpicture}[baseline=-.1cm, scale=-1]
	\pgfmathsetmacro{\width}{.4}
	\pgfmathsetmacro{\height}{.6}
	\coordinate (a) at ($ (-\width,0) + 1/2*(0,-\height) $);
	\coordinate (b) at ($ (\width,0) + 1/2*(0,-\height) $);
	\fill[\RTwoColor] ($ (\width,0) + (0,-\height) $) -- (b) arc (0:180:\width) -- ($ (-\width,0) + (0,-\height) $);
	\fill[\RThreeColor] ($ 2*(-\width,0) + (0,\height) $) -- ($ 3*(\width,0) + (0,\height) $) -- ($ 3*(\width,0) + (0,-\height) $) -- ($ 2*(\width,0) + (0,-\height) $) -- (b) arc (0:180:\width) -- ($ 2*(-\width,0) + (0,\height) $);
	\fill[\RFourColor] ($ 3*(-\width,0) + (0,\height) $) -- ($ 2*(-\width,0) + (0,\height) $) -- (a) -- ($ (-\width,0) + (0,-\height) $) -- ($ 3*(-\width,0) + (0,-\height) $);
	\fill[\RFourColor] ($ (\width,0) + (0,-\height) $) -- (b) -- ($ 2*(\width,0) + (0,-\height) $);
	\draw (-\width,-\height) -- ($ (-\width,0) + 1/2*(0,-\height) $);
	\draw ($ (-\width,0) + 1/2*(0,-\height) $) arc (180:98:\width);
	\draw ($ (\width,0) + 1/2*(0,-\height) $) arc (0:82:\width);
	\draw ($ (\width,0) + 1/2*(0,-\height) $) -- (\width,-\height);
	\draw ($ 2*(\width,0) + (0,-\height) $) -- (b);
	\draw ($ 2*(-\width,0) + (0,\height) $) -- (a);
	\draw[very thick] (0,-\height) -- (0,\height);
	\draw[fill=\betacolor] (a) circle (.05cm);
	\draw[fill=\betacolor] (b) circle (.05cm);
	\node at ($ (0,-\height) + (0,-.2) $) {\scriptsize{$X_1$}};
	\node at ($ (-\width,-\height) + (0,-.215) $) {\scriptsize{$\overline x$}};
	\node at ($ (\width,-\height) + (0,-.2) $) {\scriptsize{$x$}};
	\node at ($ (0,\height) + (0,.2) $) {\scriptsize{$X_i$}};
	\node at ($ 2*(\width,0) + (0,-\height) + (.1,-.2) $) {\scriptsize{$a$}};
	\node at ($ 2*(-\width,0) + (0,\height) + (-.1,.2) $) {\scriptsize{$a$}};
	\node at ($ (b) + (-.15,.5) $) {\scriptsize{$y$}};
	\node at ($ (a) + (.15,.5) $) {\scriptsize{$\overline y$}};
\end{tikzpicture}
\\&=
\frac{1}{\sqrt{D}}
\sum_{
x\in \Irr(\cC_{j1})
}
\sqrt{d_x}\;
\begin{tikzpicture}[baseline=-.1cm, scale=-1]
	\pgfmathsetmacro{\width}{.4}
	\pgfmathsetmacro{\height}{.6}
	\pgfmathsetmacro{\convOne}{.26}	
	\pgfmathsetmacro{\convTwo}{.4}	
	\fill[\RTwoColor] ($ (\width,0) + (0,-\height) $) arc (0:180:\width);
	\fill[\RFourColor] ($ 2.5*(-\width,0) + (0,\height) $) -- ($ (-\width,0) + (0,\height) $) -- ($ 2*(\width,0) + (0,-\height) $) -- ($ (\width,0) + (0,-\height) $) arc (0:180:\width) -- ($ 2.5*(-\width,0) + (0,-\height) $);
	\fill[\RThreeColor] ($ (-\width,0) + (0,\height) $) -- ($ 2.5*(\width,0) + (0,\height) $) -- ($ 2.5*(\width,0) + (0,-\height) $) -- ($ 2*(\width,0) + (0,-\height) $);
	\draw (-\width,0) + (0,\height) -- ($ \convOne*(\width,0) + \convOne*(\width,0) + \convOne*(0,-\height) + (-\width,0) - \convOne*(-\width,0) + (0,\height) - \convOne*(0,\height) $);
	\draw ($ 2*(\width,0) + (0,-\height) $) -- ($ \convTwo*(\width,0) + \convTwo*(\width,0) + \convTwo*(0,-\height) + (-\width,0) - \convTwo*(-\width,0) + (0,\height) - \convTwo*(0,\height) $);
	\draw (-\width,-\height) arc (180:98:\width);
	\draw (\width,-\height) arc (0:82:\width);
	\draw[very thick] (0,-\height) -- (0,\height);
	\node at ($ (0,-\height) + (0,-.2) $) {\scriptsize{$X_1$}};
	\node at ($ (-\width,-\height) + (0,-.2) $) {\scriptsize{$x$}};
	\node at ($ (\width,-\height) + (0,-.215) $) {\scriptsize{$\overline x$}};
	\node at ($ (0,\height) + (0,.2) $) {\scriptsize{$X_i$}};
	\node at ($ 2*(\width,0) + (0,-\height) + (.05,-.2) $) {\scriptsize{$a$}};
	\node at ($ (-\width,0) + (0,\height) + (-.1,.2) $) {\scriptsize{$a$}};
	\node at (-.25,0) {\scriptsize{$X_j$}};
\end{tikzpicture}\,
=
(\id\otimes u_X)e_{X,a}
\end{align*}
for all $a\in \cC_{ij}$ and $1\le i,j\le k$, where we have used Lemma~\ref{lem: graph calc with pairs of nodes} (Fusion) for the second equality.

Moreover, as in Lemma~\ref{lem:Projector}, one readily checks that $u_Xu_X^*=p_{X_1}$:
\begin{align*}
u_Xu_X^*
&=
\frac{1}{D}
\sum_{
\tikz[baseline=.07cm]{\draw[fill=\RThreeColor, rounded corners=3, very thin, baseline=1cm] (0,0) rectangle (.3,.3);}\,=R_j
}
\sum_{\substack{
x,y\in \Irr(\cC_{j1})
}}
\sqrt{d_x d_y}\,\,
\begin{tikzpicture}[baseline=-.1cm]
	\pgfmathsetmacro{\width}{.5}
	\pgfmathsetmacro{\height}{.8}
	\fill[\RTwoColor] (-\width,\height) arc (-180:0:\width);
	\fill[\RTwoColor] (-\width,-\height) arc (180:0:\width);
	\fill[\RThreeColor] ($ 3/2*(-\width, 0) + (0, -\height) $) -- ($ (-\width, 0) + (0, -\height) $) arc (180:0:\width) -- ($ 3/2*(\width, 0) + (0, -\height) $) -- ($ 3/2*(\width, 0) + (0, \height) $) -- ($ (\width, 0) + (0, \height) $) arc (0:-180:\width) -- ($ 3/2*(-\width, 0) + (0, \height) $);
	\draw (-\width,\height) arc (-180:-98:\width);
	\draw (\width,\height) arc (0:-82:\width);
	\draw (\width,-\height) arc (0:82:\width);
	\draw (-\width,-\height) arc (180:98:\width);
	\draw[very thick] (0,-\height) -- (0,\height);
	\node at ($ (0,-\height) + (0,-.2) $) {\scriptsize{$X_1$}};
	\node at ($ (-\width,-\height) + (0,-.2) $) {\scriptsize{$x$}};
	\node at ($ (\width,-\height) + (0,-.2) $) {\scriptsize{$\overline x$}};
	\node at ($ (0,\height) + (0,.2) $) {\scriptsize{$X_1$}};
	\node at ($ (-\width,\height) + (0,.2) $) {\scriptsize{$y$}};
	\node at ($ (\width,\height) + (0,.2) $) {\scriptsize{$\overline y$}};
	\node at (.25,0) {\scriptsize{$X_j$}};
\end{tikzpicture}
=
\frac{1}{D}
\sum_{
\tikz[baseline=.07cm]{\draw[fill=\RThreeColor, rounded corners=3, very thin, baseline=1cm] (0,0) rectangle (.3,.3);}\,=R_j
}
\sum_{\substack{
x,y\in \Irr(\cC_{j1})
\\
a\in \Irr(\cC_1)
}}
\sqrt{d_a}\,\,
\begin{tikzpicture}[baseline=-.1cm]
	\pgfmathsetmacro{\width}{.3}
	\pgfmathsetmacro{\height}{.9}
	\coordinate (f) at ($ (-\width,0) $);
	\coordinate (e) at ($ 2*(-\width,0) + (0,\width) $);
	\fill[\RTwoColor] ($ -2*(\width,0) + (0,-\height) $) -- ($ -2*(\width,0) + (0,\height) $) -- ($ (\width,0) + (0,\height) $) -- (f) -- ($ (\width,0) + (0,-\height) $);
	\fill[\RThreeColor] ($ 3*(-\width, 0) + (0, -\height) $) rectangle ($ 2*(-\width, 0) + (0, \height) $);
	\fill[\RThreeColor] ($ 2*(\width, 0) + (0, \height) $) -- ($ (\width, 0) + (0, \height) $) -- (f) -- ($ (\width, 0) + (0, -\height) $) -- ($ 2*(\width, 0) + (0, -\height) $);
	\draw ($ 2*(-\width,0) + (0, -\height) $) -- ($ 2*(-\width,0) + (0, \height) $);
	\draw (\width,-\height) -- ($ 3/8*(f) + 5/8*(\width,-\height) $);
	\draw (f) -- ($ 5/8*(f) + 3/8*(\width,-\height) $);
	\draw (\width,\height) -- ($ 3/8*(f) + 5/8*(\width,\height) $);
	\draw (f) -- ($ 5/8*(f) + 3/8*(\width,\height) $);
	\draw (e) -- (f);
	\draw[very thick] (0,-\height) -- (0,\height);
	\draw[fill=\betacolor] (e) circle (.05cm);
	\draw[fill=\betacolor] (f) circle (.05cm);
	\node at ($ (0,-\height) + (0,-.2) $) {\scriptsize{$X_1$}};
	\node at ($ 2*(-\width,0) + (0,-\height) + (0,-.2) $) {\scriptsize{$x$}};
	\node at ($ (\width,-\height) + (0,-.2) $) {\scriptsize{$\overline x$}};
	\node at ($ (0,\height) + (0,.2) $) {\scriptsize{$X_1$}};
	\node at ($ 2*(-\width,0) + (0,\height) + (0,.17) $) {\scriptsize{$y$}};
	\node at ($ (\width,\height) + (0,.2) $) {\scriptsize{$\overline y$}};
	\node at ($ 1/2*(e) + 1/2*(f) + (.05,.2) $) {\scriptsize{$a$}};
	\node at (.25,0) {\scriptsize{$X_j$}};
\end{tikzpicture}
\\&=
\frac{1}{D}
\sum_{
\tikz[baseline=.07cm]{\draw[fill=\RThreeColor, rounded corners=3, very thin, baseline=1cm] (0,0) rectangle (.3,.3);}\,=R_j
}
\sum_{\substack{
x,y\in \Irr(\cC_{j1})
\\
a\in \Irr(\cC_1)
}}
\sqrt{d_a}\,\,
\begin{tikzpicture}[baseline=-.1cm]
	\pgfmathsetmacro{\width}{.4}
	\pgfmathsetmacro{\height}{.6}
	\coordinate (a) at (-\width,.3);
	\coordinate (b) at (\width,-.3);
	\fill[\RTwoColor] (-\width,-\height) rectangle (\width,\height);
	\fill[\RThreeColor] ($ 2*(-\width, 0) + (0, -\height) $) rectangle ($ (-\width, 0) + (0, \height) $);
	\fill[\RThreeColor] ($ 2*(\width, 0) + (0, \height) $) rectangle ($ (\width, 0) + (0, -\height) $);
	\draw (-\width,-\height) -- (-\width,\height);
	\draw (\width,-\height) -- (\width,\height);
	\draw (a) -- ($ .58*(a) + .42*(b) $);
	\draw (b) -- ($ .42*(a) + .58*(b) $);
	\draw[very thick] (0,-\height) -- (0,\height);
	\draw[fill=\betacolor] (a) circle (.05cm);
	\draw[fill=\betacolor] (b) circle (.05cm);
	\node at ($ (0,-\height) + (0,-.2) $) {\scriptsize{$X_1$}};
	\node at ($ (-\width,-\height) + (0,-.2) $) {\scriptsize{$x$}};
	\node at ($ (\width,-\height) + (0,-.2) $) {\scriptsize{$\overline x$}};
	\node at ($ (0,\height) + (0,.2) $) {\scriptsize{$X_1$}};
	\node at ($ (-\width,\height) + (0,.17) $) {\scriptsize{$y$}};
	\node at ($ (\width,\height) + (0,.2) $) {\scriptsize{$\overline y$}};
	\node at ($ (b) + (-.15,.35) $) {\scriptsize{$a$}};
	\node at ($ (a) + (.2,.05) $) {\scriptsize{$a$}};
\end{tikzpicture}
=
p_{X_1}.
\end{align*}
For every object $X\in\cC'$, the isometry $u_X:X\to \Delta_1(X)$ therefore induces a unitary $X\to p_{X_1}(\Delta_1(X_1))=\Phi_1(X_1)=\Phi_1\circ\Pi_1(X)$.
The latter assemble to a unitary natural transformation $\id_\cC\stackrel{\scriptscriptstyle\simeq\,}\to \Phi_1\circ\Pi_1$.
\end{proof}

\begin{cor}
\label{cor:TwoCornersSameCommutant}
For any $i,j\in\{1,\ldots,k\}$, 
there is an equivalence $\cC_i'\simeq\cC_j'$ (equivalence of bi-involutive categories, respecting the positive structures).
\end{cor}
\begin{proof}
By Theorem~\ref{thm: C' --> C_i' is an equivalence of categories}, the two functors
\[
\cC'_i\stackrel{\Pi_i}\longleftarrow\cC'\stackrel{\Pi_j}\longrightarrow\cC'_j
\]
are equivalences of categories.
These functors are bi-involutive, and respect the positive structures.
\end{proof}

We are now in position to prove our main theorem.

\begin{thm}[Thm.~\ref{thm:MoritaEquivalentCommutantsEquivalent}]
\label{thm: main thm last section}
Let $\cC_0$ and $\cC_1$ be Morita equivalent unitary fusion categories, 
and let 
\[
\alpha_0:\cC_0\to \Bim(R_0),\qquad\,\,\alpha_1:\cC_1\to \Bim(R_1)
\]
be fully faithfully representations, where $R_0$ and $R_1$ are hyperfinite factors.
Assume that $R_0$ and $R_1$ are either both of type ${\rm II}$, or both of type ${\rm III}_1$.

Let $\cC_i'$ be the commutant category of $\cC_i$ inside $\Bim(R_i)$. 
Then $\cC_0'$ and $\cC_1'$ are equivalent as bicommutant categories
(equivalent as bi-involutive tensor categories with positive structures).
\end{thm}

\begin{proof}
Let $\cC=\big(\!\!\;\begin{smallmatrix}
\cC_0 & \cM
\\[.5mm]
\cM^* & \cC_1
\end{smallmatrix}\!\!\;\big)$
be a unitary $2\times 2$ multifusion category witnessing the Morita equivalence. 
By Theorem~\ref{thm:ExistsRepresenation}, there exists a fully faithful representation
\begin{equation}\label{eq: rep of multifusion category that witnesses the Morita equivalence}
\alpha:\cC \to\Bim(R^{\oplus 2}),
\end{equation}
where $R$ is a factor isomorphic to $R_1$.

Let $\tilde\alpha_i:=\alpha|_{\cC_i}:\cC_i\to\Bim(R)$ denote the restriction of $\alpha$ to $\cC_i\subset \cC$.
By Theorem \ref{thm:UniqueRepresentation}, $\tilde\alpha_i$ and $\alpha_i$ are isomorphic as $\Cstar$-representations:
\begin{equation*}
\begin{matrix}\begin{tikzpicture}
\node (1) at (-2,0) {$\cC_i$};
\node[inner ysep=2] (2) at (1,.7) {$\Bim(R)$};
\node[inner ysep=2] (3) at (1,-.7) {$\Bim(R_i)$};
\draw[->] ($(1.east)+(0,.07)$) --node[above]{$\scriptstyle \tilde\alpha_i$} ($(2.west)+(0,-.12)$);
\draw[->] ($(1.east)+(0,-.07)$) --node[below]{$\scriptstyle \alpha_i$} ($(3.west)+(0,.12)$);
\draw[->] (2) --node[right]{$\scriptstyle \simeq$} (3);
\node at (.1,0) {$\simeq$};
\end{tikzpicture}\end{matrix}
\end{equation*}
And by Theorem~\ref{thm: (i) and (iii) are equivalent}, they  are isomorphic as 
positive representations.
The commutant category of $\cC_i$ inside $\Bim(R)$ is therefore equivalent to the commutant category of $\cC_i$ inside $\Bim(R_i)$
(as as bi-involutive tensor categories with positive structures).
The result follows by applying Corollary \ref{cor:TwoCornersSameCommutant} to the representation \eqref{eq: rep of multifusion category that witnesses the Morita equivalence}.
\end{proof}

%% file: Multifusion.bbl
\providecommand{\bysame}{\leavevmode\hbox to3em{\hrulefill}\thinspace}
\providecommand{\MR}{\relax\ifhmode\unskip\space\fi MR }
\providecommand{\MRhref}[2]{%
  \href{http://www.ams.org/mathscinet-getitem?mr=#1}{#2}
}
\providecommand{\href}[2]{#2}
\begin{thebibliography}{EGNO15}

\bibitem[BDH14]{MR3342166}
Arthur Bartels, Christopher~L. Douglas, and Andr{\'e} Henriques,
  \emph{Dualizability and index of subfactors}, Quantum Topol. \textbf{5}
  (2014), no.~3, 289--345, \mathscinet{MR3342166} \doi{10.4171/QT/53}
  \arXiv{1110.5671}. \MR{3342166}

\bibitem[Con76]{MR0454659}
A.~Connes, \emph{Classification of injective factors. {C}ases {$II\sb{1},$}
  {$II\sb{\infty },$} {$III\sb{\lambda },$} {$\lambda \not=1$}}, Ann. of Math.
  (2) \textbf{104} (1976), no.~1, 73--115, \mathscinet{MR0454659}. \MR{0454659
  (56 \#12908)}

\bibitem[Con94]{MR1303779}
Alain Connes, \emph{Noncommutative geometry}, Academic Press Inc., San Diego,
  CA, 1994, \mathscinet{MR1303779}.

\bibitem[DGG14]{MR3178106}
Paramita Das, Shamindra~Kumar Ghosh, and Ved~Prakash Gupta, \emph{Perturbations
  of planar algebras}, Math. Scand. \textbf{114} (2014), no.~1, 38--85,
  \mathscinet{MR3178106} \arxiv{1009.0186}. \MR{3178106}

\bibitem[Egg11]{MR2861112}
J.~M. Egger, \emph{On involutive monoidal categories}, Theory Appl. Categ.
  \textbf{25} (2011), No. 14, 368--393, \mathscinet{MR2861112}. \MR{2861112}

\bibitem[EGNO15]{MR3242743}
Pavel Etingof, Shlomo Gelaki, Dmitri Nikshych, and Victor Ostrik, \emph{Tensor
  categories}, Mathematical Surveys and Monographs, vol. 205, American
  Mathematical Society, Providence, RI, 2015, \mathscinet{MR3242743}
  \doi{10.1090/surv/205}. \MR{3242743}

\bibitem[ENO05]{MR2183279}
Pavel Etingof, Dmitri Nikshych, and Viktor Ostrik, \emph{On fusion categories},
  Ann. of Math. (2) \textbf{162} (2005), no.~2, 581--642,
  \mathscinet{MR2183279} \doi{10.4007/annals.2005.162.581}
  \arXiv{math.QA/0203060}.

\bibitem[ENO10]{MR2677836}
Pavel Etingof, Dmitri Nikshych, and Victor Ostrik, \emph{Fusion categories and
  homotopy theory}, Quantum Topol. \textbf{1} (2010), no.~3, 209--273, With an
  appendix by Ehud Meir, \mathscinet{2677836} \doi{10.4171/QT/6}
  \arxiv{0909.3140}. \MR{2677836 (2011h:18007)}

\bibitem[FR13]{MR3028581}
S{\'e}bastien Falgui{\`e}res and Sven Raum, \emph{Tensor {${\rm
  C}\sp{*}$}-categories arising as bimodule categories of {${\rm II}\sb{1}$}
  factors}, Adv. Math. \textbf{237} (2013), 331--359, \mathscinet{MR3028581}
  \doi{10.1016/j.aim.2012.12.020} \arxiv{1112.4088v2}. \MR{3028581}

\bibitem[GLR85]{MR808930}
P.~Ghez, R.~Lima, and J.~E. Roberts, \emph{{$W\sp \ast$}-categories}, Pacific
  J. Math. \textbf{120} (1985), no.~1, 79--109, \mathscinet{MR808930}.
  \MR{808930 (87g:46091)}

\bibitem[Haa75]{MR0407615}
Uffe Haagerup, \emph{The standard form of von {N}eumann algebras}, Math. Scand.
  \textbf{37} (1975), no.~2, 271--283, \mathscinet{MR0407615}. \MR{0407615 (53
  \#11387)}

\bibitem[Haa87]{MR880070}
\bysame, \emph{Connes' bicentralizer problem and uniqueness of the injective
  factor of type {${\rm III}\sb 1$}}, Acta Math. \textbf{158} (1987), no.~1-2,
  95--148, \mathscinet{MR880070} \doi{10.1007/BF02392257}. \MR{880070
  (88f:46117)}

\bibitem[Haa16]{1606.03156}
Uffe Haagerup, \emph{On the uniqueness of injective $\rm {III}_1$ factor},
  2016, \arxiv{1606.03156}.

\bibitem[Hen17]{MR3747830}
Andr\'{e}~G. Henriques, \emph{What {C}hern-{S}imons theory assigns to a point},
  Proc. Natl. Acad. Sci. USA \textbf{114} (2017), no.~51, 13418--13423,
  \mathscinet{MR3747830} \doi{10.1073/pnas.1711591114} \arxiv{1503.06254}.
  \MR{3747830}

\bibitem[HI98]{MR1644299}
Fumio Hiai and Masaki Izumi, \emph{Amenability and strong amenability for
  fusion algebras with applications to subfactor theory}, Internat. J. Math.
  \textbf{9} (1998), no.~6, 669--722, \mathscinet{MR1644299}.

\bibitem[Hia88]{MR976765}
Fumio Hiai, \emph{Minimizing indices of conditional expectations onto a
  subfactor}, Publ. Res. Inst. Math. Sci. \textbf{24} (1988), no.~4, 673--678,
  \mathscinet{MR976765} \doi{10.2977/prims/1195174872}. \MR{976765}

\bibitem[HP17]{MR3663592}
Andr\'e Henriques and David Penneys, \emph{Bicommutant categories from fusion
  categories}, Selecta Math. (N.S.) \textbf{23} (2017), no.~3, 1669--1708,
  \mathscinet{MR3663592} \doi{10.1007/s00029-016-0251-0} \arxiv{1511.05226}.
  \MR{3663592}

\bibitem[HPT16]{1607.06041}
Andr\'e Henriques, David Penneys, and James~E. Tener, \emph{Planar algebras in
  braided tensor categories}, 2016, \arxiv{1607.06041}.

\bibitem[Izu03]{MR1953517}
Masaki Izumi, \emph{Canonical extension of endomorphisms of type {III}
  factors}, Amer. J. Math. \textbf{125} (2003), no.~1, 1--56,
  \mathscinet{MR1953517} \arxiv{math/0104228}. \MR{1953517}

\bibitem[Izu17]{MR3635673}
\bysame, \emph{A {C}untz algebra approach to the classification of near-group
  categories}, Proceedings of the 2014 {M}aui and 2015 {Q}inhuangdao
  conferences in honour of {V}aughan {F}. {R}. {J}ones' 60th birthday, Proc.
  Centre Math. Appl. Austral. Nat. Univ., vol.~46, Austral. Nat. Univ.,
  Canberra, 2017, \mathscinet{MR3635673} \arxiv{1512.04288}, pp.~222--343.
  \MR{3635673}

\bibitem[Jon99]{math.QA/9909027}
Vaughan F.~R. Jones, \emph{Planar algebras {I}}, 1999, \arXiv{math.QA/9909027}.

\bibitem[JP17]{MR3687214}
Corey Jones and David Penneys, \emph{Operator algebras in rigid {$\rm
  C^*$}-tensor categories}, Comm. Math. Phys. \textbf{355} (2017), no.~3,
  1121--1188, \mathscinet{MR3687214} \doi{10.1007/s00220-017-2964-0}
  \arxiv{1611.04620}. \MR{3687214}

\bibitem[JS91]{MR1113284}
Andr{\'e} Joyal and Ross Street, \emph{The geometry of tensor calculus. {I}},
  Adv. Math. \textbf{88} (1991), no.~1, 55--112, \mathscinet{MR1113284}.
  \MR{MR1113284 (92d:18011)}

\bibitem[KL92]{MR1172035}
Hideki Kosaki and Roberto Longo, \emph{A remark on the minimal index of
  subfactors}, J. Funct. Anal. \textbf{107} (1992), no.~2, 458--470,
  \mathscinet{MR1172035} \doi{10.1016/0022-1236(92)90118-3}. \MR{1172035}

\bibitem[Lon89]{MR1027496}
Roberto Longo, \emph{Index of subfactors and statistics of quantum fields.
  {I}}, Comm. Math. Phys. \textbf{126} (1989), no.~2, 217--247,
  \mathscinet{MR1027496}.

\bibitem[LR97]{MR1444286}
R.~Longo and J.~E. Roberts, \emph{A theory of dimension}, $K$-Theory
  \textbf{11} (1997), no.~2, 103--159, \mathscinet{MR1444286}
  \doi{10.1023/A:1007714415067}. \MR{1444286}

\bibitem[Lus87]{MR933415}
G.~Lusztig, \emph{Leading coefficients of character values of {H}ecke
  algebras}, The {A}rcata {C}onference on {R}epresentations of {F}inite
  {G}roups ({A}rcata, {C}alif., 1986), Proc. Sympos. Pure Math., vol.~47, Amer.
  Math. Soc., Providence, RI, 1987, \mathscinet{MR933415}, pp.~235--262.
  \MR{933415 (89b:20087)}

\bibitem[MT07]{MR2322913}
Toshihiko Masuda and Reiji Tomatsu, \emph{Classification of minimal actions of
  a compact {K}ac algebra with amenable dual}, Comm. Math. Phys. \textbf{274}
  (2007), no.~2, 487--551, \mathscinet{MR2322913}
  \doi{10.1007/s00220-007-0269-4} \arxiv{math/0604348}. \MR{2322913}

\bibitem[MT09]{MR2483716}
\bysame, \emph{Approximate innerness and central triviality of endomorphisms},
  Adv. Math. \textbf{220} (2009), no.~4, 1075--1134, \mathscinet{MR2483716}
  \doi{10.1016/j.aim.2008.10.005} \arxiv{0802.0344}. \MR{2483716}

\bibitem[M{\"u}g03]{MR1966524}
Michael M{\"u}ger, \emph{From subfactors to categories and topology. {I}.
  {F}robenius algebras in and {M}orita equivalence of tensor categories}, J.
  Pure Appl. Algebra \textbf{180} (2003), no.~1-2, 81--157,
  \mathscinet{MR1966524} \doi{10.1016/S0022-4049(02)00247-5}
  \arXiv{math.CT/0111204}.

\bibitem[MvN43]{MR0009096}
F.~J. Murray and J.~von Neumann, \emph{On rings of operators. {IV}}, Ann. of
  Math. (2) \textbf{44} (1943), 716--808, \mathscinet{MR0009096}. \MR{0009096
  (5,101a)}

\bibitem[Pen18]{1808.00323}
David Penneys, \emph{Unitary dual functors for unitary multitensor categories},
  2018, \arxiv{1808.00323}.

\bibitem[Pop90]{MR1055708}
Sorin Popa, \emph{Classification of subfactors: the reduction to commuting
  squares}, Invent. Math. \textbf{101} (1990), no.~1, 19--43,
  \mathscinet{MR1055708}, \doi{10.1007/BF01231494}.

\bibitem[Pop95a]{MR1334479}
\bysame, \emph{An axiomatization of the lattice of higher relative commutants
  of a subfactor}, Invent. Math. \textbf{120} (1995), no.~3, 427--445,
  \mathscinet{MR1334479} \doi{10.1007/BF01241137}.

\bibitem[Pop95b]{MR1339767}
\bysame, \emph{Classification of subfactors and their endomorphisms}, CBMS
  Regional Conference Series in Mathematics, vol.~86, Published for the
  Conference Board of the Mathematical Sciences, Washington, DC, 1995,
  \mathscinet{MR1339767}. \MR{1339767 (96d:46085)}

\bibitem[PV15]{MR3406647}
Sorin Popa and Stefaan Vaes, \emph{Representation theory for subfactors,
  {$\lambda$}-lattices and {$\rm C^*$}-tensor categories}, Comm. Math. Phys.
  \textbf{340} (2015), no.~3, 1239--1280, \mathscinet{MR3406647}
  \doi{10.1007/s00220-015-2442-5} \arxiv{1412.2732}. \MR{3406647}

\bibitem[Sau83]{MR703809}
Jean-Luc Sauvageot, \emph{Sur le produit tensoriel relatif d'espaces de
  {H}ilbert}, J. Operator Theory \textbf{9} (1983), no.~2, 237--252,
  \mathscinet{MR703809}.

\bibitem[Sel07]{10.1016/j.entcs.2006.12.018}
P.~Selinger, \emph{Dagger compact closed categories and completely positive
  maps: (extended abstract)}, Proceedings of the 3rd {I}nternational {W}orkshop
  on {Q}uantum {P}rogramming {L}anguages ({QPL} 2005), vol. 170, 2007,
  \doi{10.1016/j.entcs.2006.12.018}, Available at:
  \url{https://ncatlab.org/nlab/files/SelingerPositiveMaps.pdf}, pp.~139--163.

\bibitem[Sel11]{MR2767048}
\bysame, \emph{A survey of graphical languages for monoidal categories}, New
  structures for physics, Lecture Notes in Phys., vol. 813, Springer,
  Heidelberg, 2011, \mathscinet{MR2767048}
  \href{http://dx.doi.org/10.1007/978-3-642-12821-9_4}{\tt
  DOI:10.1007/978-3-642-12821-9\char`_4}, pp.~289--355. \MR{2767048
  (2012j:18011)}

\bibitem[SY17]{1705.05600}
Yusuke Sawada and Shigeru Yamagami, \emph{Notes on the bicategory of $\rm
  {W}^*$-bimodules}, 2017, \arxiv{1705.05600}.

\bibitem[Tak03]{MR1943006}
Masamichi Takesaki, \emph{Theory of operator algebras. {II}}, Encyclopaedia of
  Mathematical Sciences, vol. 125, Springer-Verlag, Berlin, 2003, Operator
  Algebras and Non-commutative Geometry, 6 \mathscinet{MR1943006}.

\bibitem[Tom18]{1812.04222}
Reiji Tomatsu, \emph{Centrally free actions of amenable {C}*-tensor categories
  on von {N}eumann algebras}, 2018, \arxiv{1812.04222}.

\end{thebibliography}
